\documentclass[11pt]{amsart}
\usepackage{amsmath, amssymb, amsthm, amsfonts, mathrsfs, tikz, circuitikz}
\usetikzlibrary {arrows.meta,bending,positioning}
\usepackage{comment}

\usepackage{booktabs}
\usepackage{array}
\usepackage{enumitem}
\usepackage{microtype}

\newtheorem{theorem}{Theorem}[section]
\newtheorem{proposition}[theorem]{Proposition}
\newtheorem{corollary}[theorem]{Corollary}

\newtheorem{lemma}[theorem]{Lemma}

\newtheorem{remark}[theorem]{Remark}

\theoremstyle{definition}
\newtheorem{definition}[theorem]{Definition}
\newtheorem{problem}[theorem]{Problem}

\newcommand{\dom}{{\mbox{\rm dom}}}

\DeclareSymbolFont{AMSb}{U}{msb}{m}{n}
\DeclareMathSymbol{\N}{\mathbin}{AMSb}{"4E}
\DeclareMathSymbol{\Z}{\mathbin}{AMSb}{"5A}
\DeclareMathSymbol{\R}{\mathbin}{AMSb}{"52}
\DeclareMathSymbol{\Q}{\mathbin}{AMSb}{"51}
\DeclareMathSymbol{\I}{\mathbin}{AMSb}{"49}
\DeclareMathSymbol{\C}{\mathbin}{AMSb}{"43}

\theoremstyle{definition}

\numberwithin{equation}{section}

\DeclareMathOperator{\ord}{ord}
\DeclareMathOperator{\Sym}{Sym}
\DeclareMathOperator{\Gal}{Gal}
\DeclareMathOperator{\Aut}{Aut}
\DeclareMathOperator{\GL}{GL}
\DeclareMathOperator{\Ker}{ker}

\newcommand{\Fp}{\mathbb F_p}
\newcommand{\F}{\mathbb F}

\begin{document}
\title{Extension Properties for Partial Permutations}
\author{Mahmood Etedadialiabadi}
\address{1979 Milky Way, Verona, WI 53593, U.S.A.}
\email{mahmood.etedadi@gmail.com}

\author{Su Gao}
\address{School of Mathematical Sciences and LPMC, Nankai University, Tianjin 300071, P.R. China}
\email{sgao@nankai.edu.cn}
\thanks{The second author acknowledges the partial support of his research by the National Natural Science Foundation of China (NSFC) grants 12271263 and 12250710128.}

\date{\today}

\begin{abstract} Motivated by the EPPA problem for finite tournaments, we consider various extension properties for partial permutations. We show that for any set $\Pi$ of prime numbers, the $\Pi$-extension property is equivalent to the $\Pi$-LERF. As a consequence of known results, the odd-extension property is then equivalent to the EPPA for finite tournaments. To study the $\Pi$-extension property, we reformate the property using the concepts of EP problems and $\Pi$-solutions. We then show that the existence of a $\Pi$-solution for an EP problem depends entirely on its fundamental group. In particular, when the fundamental group is trivial or cyclic, the EP problem has a $\Pi$-solution for any $\Pi$. These extend some known results of Huang, Pawliuk, Sabok and Wise \cite{HPSW19}. We give examples of EP problems without nilpotent-solutions; they witness that the $p$-extension property and the $p$-LERF fail for any prime $p$. Then we consider a special kind of EP problems whose fundamental groups have two generators. We show that any $1$-weakly wandering problem has a $\Pi$-solution for any $\Pi$. For $2$-weakly wandering problems, we show that they all have odd-solutions and we completely characterize those without nilpotent-solutions. 
\end{abstract}

\maketitle
\tableofcontents

\section{Introduction}

The research presented in this paper was motivated by the following long-standing open problem.

\begin{problem} [Herwig--Lascar \cite{HL00}]\label{prob:1.1} Does the class of all finite tournaments have the EPPA (or the Hrushovski property)? That is, given a finite tournament $A$, does there always exist a finite tournament $B$ so that $B$ is an extension of $A$ and any partial automorphism of $A$ can be extended to an automorphism of $B$?
\end{problem}

For recent developments on this problem, see \cite{HPSW19} and \cite{HJKS25}. In this paper, we study an equivalent formulation of this problem and some generalizations thereof. Thus the first part of the paper is devoted to proving the equivalence of this problem with our formulation.

Herwig--Lascar \cite[Section 6]{HL00} already showed that Problem~\ref{prob:1.1} is equivalent to a statement about the closedness of certain subgroups in the pro-odd topology of a finitely generated free group.

\begin{theorem}[Herwig--Lascar \cite{HL00}]\label{thm:HL00} The following are equivalent:
\begin{enumerate}
\item[\rm (a)] The class of all finite tournaments has the EPPA.
\item[\rm (b)] For any finitely generated free group $F$ and any finitely generated subgroup $H$ of $F$, if for all $a\in F$, we have $a\in H$ whenever $a^2\in H$, then $H$ is closed in the pro-odd topology of $F$.
\end{enumerate}
\end{theorem}

Here the pro-odd topology of $F$ is generated by the basic open sets of the form $gN$, where $g\in F$ and $N$ is a normal subgroup of $F$ of finite odd index.

In a broader context, there is an extensive literature on the EPPA, the profinite topology, and their connections. The EPPA was first proved by Hrushovski \cite{Hr92} for the class of all finite graphs. It has now been established for many other classes of finite structures (see, e.g. \cite{He95, HL00, So05, Co19, HPSW19, EHKN20, Ko20, HKN22, HJKS25}); see \cite{HKN22} and \cite{HJKS25} for good surveys and comprehensive lists of references for this very active topic of research. 

The profinite topology was introduced by M. Hall in \cite{Ha49, Ha50}. He proved in \cite[Theorem 5.1]{Ha49} that a finitely generated subgroup $H$ of any free group $F$ is closed in the profinite topology of $F$. One can clearly see the connection of Hall's theorem with the above statement (b). Hall's theorem was generalized by Ribes--Zalesskii \cite{RZ93}; they showed that if $F$ is any free group and $H_1, \dots, H_n$ are finitely generated subgroups of $F$, then $H_1\cdots H_n$ is closed in the profinite topology of $F$. The Ribes--Zalesskii theorem has inspired a lot of research, from new proofs of the Ribes--Zalesskii theorem \cite{AS05, Zh25} to other topologies of the free group \cite{RZ94, AS04, Au17}, to other groups \cite{Yo97, Co01, Mi06,  RZBook, Ro11, MS24, MM25}. See \cite{ABO26} for a good survey and a comprehensive list of references on this subject.

The equivalence between the EPPA for finite relational structures and the Ribes--Zalesskii theorem for free groups was already noticed by Herwig--Lascar in the proofs of \cite{HL00}. This equivalence has been extended and further explored in different contexts. See \cite{ABO26} for a clearer account of the history and the proofs of these equivalences. Our equivalent formulation of the EPPA for finite tournaments also has its roots in the proof of Theorem~\ref{thm:HL00}.

To state our equivalent formulation, consider a finite set $X$ and a set $P$ of partial permutations of $X$. Let $F(P)$ be the free group with alphabet $P$. For $p\in P$, we let $p^{-1}$ denote both the formal inverse of $p$ in $F(P)$ and the partial permutation which is the inverse of the partial permutation $p$. For $w=p_1\dots p_n\in F(P)$ where $p_1,\dots, p_n\in P\cup P^{-1}$, we regard $w$ both as an element of $F(P)$ and as the partial permutation $p_1\circ\cdots\circ p_n$. Then consider the following statement:
\begin{enumerate}
\item[(c)] \emph{For any finite set $X$ and any set $P$ of partial permutations of $X$, if for any $w\in F(P)$ and $x\in X$, we have $w(x)=x$ whenever $w^2(x)=x$, then there exist a finite set $Y\supseteq X$ and extensions $p\mapsto \tilde{p}$ from $P$ to $\operatorname{Sym}(Y)$ such that the group generated by $\{\tilde{p}\colon p\in P\}$ has odd order. }
\end{enumerate}
We will show that (c) is equivalent to both (a) and (b) in Theorem~\ref{thm:HL00}. We call (c) the \emph{odd-extension property for partial permutations}, or the \emph{odd-EP} for short.

Readers familiar with semigroup theory will readily recognize that the theory of partial permutations are usually axiomatized as \emph{inverse monoids}. The famous Wagner--Preston representation theorem (see, e.g. \cite[Section 1.5]{La98}) states exactly that any inverse monoid can be represented by a set of partial permutations on a set. In this paper we work with concrete constructions instead of algebraic algorithms, and so it is more convenient for us to consider directly sets of partial permutations, instead of thinking of them as inverse monoids.

In fact, we gain a bit more insight of our problem by considering a generalized extension property for partial permutations. Let $\Pi$ be a set of prime numbers. Consider the following statement which we call the \emph{$\Pi$-EP}:
\begin{enumerate}
\item[(d)] \emph{For any finite set $X$ and any set $P$ of partial permutations of $X$, if for any prime $p\not\in \Pi$, $w\in F(P)$ and $x\in X$, we have $w(x)=x$ whenever $w^p(x)=x$, then there exist a finite set $Y\supseteq X$ and extensions $p\mapsto \tilde{p}$ from $P$ to $\operatorname{Sym}(Y)$ such that the group generated by $\{\tilde{p}\colon p\in P\}$ is a $\Pi$-group. }
\end{enumerate}
The odd-EP is exactly the $\Pi$-EP with $\Pi$ being the set of all odd primes. We will show that the $\Pi$-EP is equivalent to a version of Hall's theorem which is called the \emph{$\Pi$-LERF} (LERF stands for \emph{locally extended residually finite}), as well as a version of the Ribes--Zalesskii theorem which is called the \emph{$\Pi$-RZ-property} (RZ stands, obviously, for Ribes--Zalesskii). One crucial ingredient of these proofs is Stallings's proof of Hall's theorem \cite{St83} which deals with multidigraphs of Serre's style \cite{SerreBook}. (See \cite{AS05, Zh25} for a proof of Ribes--Zalesskii theorem in the same style.) For the convenience of the reader we will recall Stallings's proof of Hall's theorem together with all the concepts it uses (\emph{foldings}, \emph{immersions}, \emph{fundamental groups}, etc.), which makes this paper mostly self-contained. 

After proving these equivalences, we will prove some special cases of the $\Pi$-EP in the rest of the paper. Our results extend some results of \cite{HPSW19}, and answer some questions left open there. To formulate our results, it is convenient to use the following terminology. If $X$ is a finite set and $P$ is a set of partial permutations of $X$, we call $(X, P)$ an \emph{EP problem}. A \emph{solution} for an EP problem $(X, P)$ is a pair consisting of another EP problem $(Y, Q)$ and a bijection $j\colon P\to Q$ such that $X\subseteq Y$, $Q\subseteq \operatorname{Sym}(Y)$ and for each $p\in P$, $j(p)$ extends $p$. When the bijection $j$ is clear from the context, we supress mentioning it and regard $(Y, Q)$ as a solution for $(X, P)$. It is trivial that any EP problem has a (finite) solution. If $\Pi$ is a set of primes, then a \emph{$\Pi$-solution} is an EP problem $(Y, Q)$  which is a solution for itself so that the subgroup of $\operatorname{Sym}(Y)$ generated by $Q$ is a $\Pi$-group. An \emph{odd-solution} is therefore a $\Pi$-solution where $\Pi$ is the set of all odd primes. Much of our effort in this paper is to determine whether a particular EP problem has a $\Pi$-solution for various $\Pi$. For instance, we can fix a particular prime $p$ and ask whether certain EP problem has a $p$-solution. Since every finite $p$-group is nilpotent (see, e.g. \cite[5.1.3]{RobinsonBook}), it makes sense for us to consider the notion of \emph{nilpotent-solutions}. By the celebrated Feit-Thompson theorem \cite{FT63}, every odd-group (group of odd order) is solvable; thus we are particularly interested in \emph{solvable-solutions} for EP problems. 

The collections of $p$-groups, nilpotent groups, odd-groups or solvable groups are all examples of pseudovarieties of groups. A \emph{pseudovariety} of groups is a class of finite groups closed under taking finite direct products, subgroups and quotients.  For any pseudovariety $\mathsf{H}$ of groups, one can consider the pro-$\mathsf{H}$ topology of a group $F$, which is generated by basic open sets of the form $gN$, where $g\in F$ and $N$ is a normal subgroup of $F$ of finite index so that $F/N\in \mathsf{H}$. Pro-$\mathsf{H}$ topologies have been studied extensively in the literature (see, e.g. \cite{RZ94, MSW01, AS04, AS05, AS06, RZBook, MST24a, MST24b, MST25, MST26}). In this paper we focus on the pseudovarieties mentioned above. 
We will show that, for any of these pseudovarieties $\mathsf{H}$, the existence of $\mathsf{H}$-solutions for an EP problem $(X, P)$ is correspondent (often equivalent) to the closedness of the fundamental group $\pi_1(X, P)$ of $(X, P)$ (the graph $G(X, P)$ and the fundamental group $\pi_1(X, P)$ are to be defined in Subsection~\ref{subsec:fundamentalgroup}) in the pro-$\mathsf{H}$ topology of the free group $F(P)$. In particular, whether an EP problem has an $\mathsf{H}$-solution depends entirely on its fundamental group. 

The extension problems for partial permutations has been studied by Auinger--Steinberg \cite{AS03}. The research problem they considered, stated in our terminology, is the existence of an $\mathsf{H}$-solution for an EP problem, except that their results are about the decidability of this existence (with the EP problem, which is a finite object and therefore can be coded by a natural number, as input). 

Let $\Pi$ be any set of primes. A necessary condition for an EP problem $(X, P)$ to have a $\Pi$-solution is that $(X, P)$ needs to be \emph{quo-$\Pi$}, i.e., for any $p\not\in \Pi$, $w\in F(P)$ and $x\in X$, we have $w(x)=x$ whenever $w^p(x)=x$. Extending results of \cite{HPSW19}, and more specifically \cite[Theorem 1.5]{HPSW19}, we have the following theorem.

\begin{theorem}\label{thm:basic} Let $\Pi$ be any set of primes. Let $(X, P)$ be an EP problem where $G(X, P)$ is weakly connected.
\begin{enumerate}
\item[\rm (1)] If $\pi_1(X, P)$ is trivial, then $(X, P)$ has a $\Pi$-solution.
\item[\rm (2)] If $\pi_1(X, P)$ is cyclic (i.e. generated by a single element) and $(X, P)$ is quo-$\Pi$, then $(X, P)$ has a $\Pi$-solution.
\end{enumerate}
\end{theorem}

Thus the $\Pi$-EP holds whenever the fundamental group $\pi_1(X, P)$ is trivial or cyclic. In the remainder of the paper, we focus on \emph{bicyclic} EP problems, i.e. those EP problems whose fundamental groups have 2 generators. We will show that the $\Pi$-EP for such problems is equivalent to the $\Pi$-EP for bicyclic EP problems $(X, P)$ where $P$ has exactly two elements. Among these problems, we will further restrict our attention to EP problems of the form $(X, P)$ where $P=\{a, b\}$ and $\pi_1(X,P)$ has the form $\langle a, w^{-1}aw\rangle$ for some $w\in F(a, b)$ (note here that $w$ can be assumed to both start and end with either $b$ or $b^{-1}$). We refer to such a problem as an \emph{$\langle a, w^{-1}aw\rangle$-problem}. 

We will show that any $\langle a, w^{-1}aw\rangle$-problem is quo-$\Pi$ for any $\Pi$. Thus we focus on determining whether they have $\Pi$-solutions. We introduce some concepts (\emph{wandering}, \emph{weakly wandering}, \emph{$\kappa$-weakly wandering}, etc.) according to how the walk given by the word $w$ behaves in $G(X,P)$. Our main results are as follows.

\begin{theorem}\label{thm:mainwandering} Let $\Pi$ be any set of primes. For any wandering $w$, an $\langle a, w^{-1}aw\rangle$-problem has a $\Pi$-solution. Equivalently, for any wandering $w$, $\langle a, w^{-1}aw\rangle$ is closed in the pro-$\Pi$ topology of $F(a, b)$.
\end{theorem}

\begin{theorem}\label{thm:mainweaklywandering} Let $\Pi$ be any set of primes. For any $w$, if either $w$ or $w^{-1}$ is weakly wandering, then an $\langle a, w^{-1}aw\rangle$-problem has a $\Pi$-solution. Equivalently, for any $w$,  if either $w$ or $w^{-1}$ is weakly wandering, then $\langle a, w^{-1}aw\rangle$ is closed in the pro-$\Pi$ topology of $F(a, b)$.
\end{theorem}

The above theorems deal with EP problems for which there exist $\Pi$-solutions for any $\Pi$, and in particular they have $p$-solutions for any prime $p$, and therefore they have nilpotent-solutions. It turns out that weakly wandering (which is the same as $1$-weakly wandering) is as far as we can go to get such results. We will consider a $2$-weakly wandering problem with fundamental group $\langle a, w^{-1}aw\rangle$ where $w=bab^{-1}a^{-1}b$. This example is equivalent to some example considered in \cite{HPSW19}, and in particular it was not known whether it has an odd-solution. Using arguments from \cite{MSW01} and \cite{HPSW19}, we will show that it does not have any nilpotent-solutions. In particular, our example witness the failure of the $p$-EP and $p$-LERF for any prime $p$.

\begin{theorem}\label{thm:pEP} There is a $2$-weakly wandering problem without $p$-solutions for any prime $p$. Thus for any prime $p$,  the $p$-{\rm EP} and the $p$-{\rm LERF} fail to hold.
\end{theorem}

In fact, we completely characterize those $2$-weakly wandering problems without $p$-solutions for any prime $p$. In view of the existence of these examples, it takes a new construction to find an odd-solution for a general $2$-weakly wandering problem. We give such a construction.

\begin{theorem}\label{thm:main2weaklywandering} Any $2$-weakly wandering problem has an odd-solution.
\end{theorem}

For general $\kappa$-weakly wandering problems, the question whether they have odd-solutions boils down to a special subclass of problems known as \emph{$(\kappa-1)$-strongly stationary} ones. We can prove that all strongly stationary problems do not have nilpotent-solutions, but we do not yet know if they all have odd-solutions.

\section{Preliminaries}

\subsection{Tournaments}

\begin{definition} A {\em tournament} is an irreflexive directed graph $G=(V, E)$ where for every pair of distinct vertices $x, y\in V$, exactly one of $(x,y)$ and $(y,x)$ is an element of $E$.
\end{definition}

\begin{lemma}\label{lem:aut} If $G$ is a tournament, then no elements of $\Aut(G)$ have finite even order. In particular, if $G$ is a finite tournament, then $|\Aut(G)|$ is odd.
\end{lemma}

\begin{proof} Toward a contradiction, assume $\lambda\in \Aut(G)$ has even order. Then some power of $\lambda$ has order $2$. Without loss of generality, we may assume $\lambda$ has order $2$. That is, $\lambda\neq \mbox{id}$, but $\lambda^2=\mbox{id}$. Since $\lambda\neq\mbox{id}$, there is some $x\in G$ such that $\lambda(x)\neq x$. Let $y=\lambda(x)\neq x$. Then $\lambda(y)=x$. However, there is exactly one of $(x, y)$ and $(y,x)$ as an edge in $G$. Say $(x,y)\in G$. Then $\lambda(x,y)=(y,x)\in G$ since $\lambda$ is an automorphism. But this is a contradiction to $G$ being a tournament. 
\end{proof}

\begin{definition} Given an tournament $G$, a {\em partial automorphism} of $G$ is a bijection $f\colon A\to B$ where $A, B\subseteq G$ such that for any $x, y\in A$, $(x, y)$ is an edge in $G$ if and only if $(f(x), f(y))$ is an edge in $G$.
\end{definition}

\begin{definition} The {\em EPPA} (standing for {\em Extension Property for Partial Automorphisms}) or the {\em Hrushovski property} for the the class of all finite tournaments is the statement:
\begin{quote}
 Given any finite tournament $A$, there exists a finite tournament $B$ so that $B$ is an extension of $A$ and any partial automorphism of $A$ can be extended to an automorphism of $B$.
 \end{quote}
 \end{definition}

\subsection{Psuedovarieties of finite groups}
Notation in this subsection follows \cite{RZBook}. Unexplained basic concepts and notation in group theory follow \cite{RobinsonBook}.

\begin{definition}\label{def:formation} A {\em formation} of finite groups is a nonempty class of finite groups $\mathcal{C}$ that satisfies the following:
\begin{enumerate}
\item[(1)] $\mathcal{C}$ is isomorphism-invariant, that is, if $A\in \mathcal{C}$ and $A\cong B$, then $B\in\mathcal{C}$;
\item[(2)] $\mathcal{C}$ is closed under taking quotients, that is, if $A\in\mathcal{C}$ and $N\unlhd A$, then $A/N\in\mathcal{C}$;
\item[(3)] $\mathcal{C}$ is closed under taking finite subdirect products, that is, if $A_i\in \mathcal{C}$, $i=1,\dots, n$, and $B$ is a subdirect product of $A_1,\dots, A_n$, then $B\in \mathcal{C}$.
\end{enumerate}
\end{definition}

Here a group $B$ is a {\em subdirect product} of $A_1, \dots, A_n$ if there are normal subgroups $N_1,\dots, N_n$ of $B$ such that $\bigcap_i N_i=\{1_B\}$ and $B/N_i\cong A_i$ for all $i=1,\dots, n$. It is equivalent to $B$ being isomorphic to a subgroup of the direct product $\prod_i A_i$. Condition (3) above is equivalent to: if $A$ is a finite group with normal subgroups $N_1$ and $N_2$ such that $A/N_1, A/N_2\in\mathcal{C}$, then $A/(N_1\cap N_2)\in\mathcal{C}$. 

Examples of formations include finite groups, finite $p$-groups (for any fixed prime $p$), finite abelian groups, finite nilpotent groups and finite solvable groups. The class of all finite cyclic groups is not a formation. 

\begin{definition} Let $\Pi$ be a set of prime numbers, let $k$ be a postive integer and let $\Gamma$ be a finite group. The integer $k$ is a {\em $\Pi$-number} if all prime factors of $k$ are in $\Pi$. The group $\Gamma$ is a {\em $\Pi$-group} if $|\Gamma|$ is a $\Pi$-number. $\Gamma$ is an {\em odd-group} if $|\Gamma|$ is odd.
\end{definition}

Obviously, a finite group is an odd-group if and only if it is a $\Pi$-group, where $\Pi$ is the set of all odd prime numbers. For any $\Pi$, the class of all finite $\Pi$-groups is a formation. In particular, the class of all finite odd-groups is a formation.

\begin{definition} A formation of finite groups is a {\em psuedovariety} if it is closed under subgroups and finite direct products.
\end{definition}

All formations mentioned in this subsection are in fact pseudovarieties.

\subsection{The pro-$\mathcal{C}$ topology}

\begin{definition} Let $\Gamma$ be a group and let $\mathcal{N}$ be a nonempty collection of normal subgroups of finite index of $\Gamma$. We say that $\mathcal{N}$ is {\em filtered from below} if whenever $N_1, N_2\in\mathcal{N}$, there exists $N\in\mathcal{N}$ such that $N\leq N_1\cap N_2$.
\end{definition}

\begin{lemma}\label{lem:filteredbelow} Let $\Gamma$ be a group and let $\mathcal{C}$ be a formation of finite groups. Define
$$ \mathcal{N}_{\mathcal{C}}(\Gamma)=\{N\unlhd \Gamma\colon \Gamma/N\in\mathcal{C}\}. $$
Then $\mathcal{N}_{\mathcal{C}}(\Gamma)$ is filtered from below.
\end{lemma}

\begin{proof} We show that $\mathcal{N}_{\mathcal{C}}(\Gamma)$ has the stronger property that if $N_1, N_2\in \mathcal{N}_{\mathcal{C}}(\Gamma)$ then $N_1\cap N_2\in\mathcal{N}_{\mathcal{C}}(\Gamma)$. This is equivalent to that for any $N_1, N_2\unlhd \Gamma$, if $\Gamma/N_1, \Gamma/N_2\in\mathcal{C}$ then $\Gamma/(N_1\cap N_2)\in \mathcal{C}$. To prove this, consider the following map 
$$\begin{array}{rrcl} \theta\colon & \Gamma/(N_1\cap N_2)& \to & \Gamma/N_1\times\Gamma/N_2 \\
& \gamma(N_1\cap N_2) & \mapsto & (\gamma N_1, \gamma N_2) 
\end{array}
$$
It is straightforward to check that $\theta$ is a well-defined injective homomorphism. Thus $\Gamma/(N_1\cap N_2)$ is a subdirect product of $\Gamma/N_1$ and $\Gamma/N_2$. By condition (3) of Definition~\ref{def:formation}, $\Gamma/(N_1\cap N_2)\in \mathcal{C}$. 
\end{proof}

\begin{definition} Let $\Gamma$ be a group and let $\mathcal{C}$ be a formation of finite groups. The {\em pro-$\mathcal{C}$ topology} on $\Gamma$ is the topological group topology with
$$ \{ N\unlhd \Gamma\colon \Gamma/N\in\mathcal{C}\} $$
as the neighborhood base of $1_\Gamma$.
\end{definition}

More explicitly, a topological base of the pro-$\mathcal{C}$ topology on $\Gamma$ is
$$ \{ \gamma N\colon \gamma\in \Gamma, N\unlhd G, G/N\in\mathcal{C}\}. $$
Abstractly, the pro-$\mathcal{C}$ topology on $\Gamma$ is the coarsest topology on $\Gamma$ that makes all the homomorphisms from $\Gamma$ to a $\mathcal{C}$-group continuous. 

This definition covers the usual notions of profinite topology as well as the less familiar notions of pro-$p$, pro-$\Pi$, pro-abelian, pro-nilpotent, pro-solvable, and pro-odd topologies.

\begin{definition} Let $\mathcal{C}$ be a formation of finite groups. A group $\Gamma$ is {\em residually $\mathcal{C}$} if 
$$ \bigcap_{N\in\mathcal{N}_{\mathcal{C}}(\Gamma)}N=\bigcap\{N\unlhd \Gamma\colon G/N\in \mathcal{C}\}=\{1_\Gamma\}. $$
\end{definition}

The following fact is immediate from the definitions.

\begin{lemma} Let $\mathcal{C}$ be a formation of finite groups. A group $\Gamma$ is residually $\mathcal{C}$ if and only if the pro-$\mathcal{C}$ topology on $\Gamma$ is Hausdorff.
\end{lemma}


\subsection{LERF and the RZ-property}

\begin{definition} Let $\Gamma$ be a group.
\begin{enumerate}
\item[(i)] $\Gamma$ is {\em LERF} ({\em locally extended residually finite}) if any finitely generated subgroup $H$ of $\Gamma$ is closed in the profinite topology of $\Gamma$.
\item[(ii)] $\Gamma$ has the {\em RZ-property} ({\em Ribes--Zalesskii property}) if for any finitely generated subgroups $H_1,\dots, H_n\leq \Gamma$, $H_1\cdots H_n$ is closed in the profinite topology of $\Gamma$.
\end{enumerate}
\end{definition}

\begin{theorem}[Hall \cite{Ha49}]\label{thm:Hall} Any free group is LERF.
\end{theorem}

We will give a self-contained proof of Theorem~\ref{thm:Hall}, which was originally given by Stallings \cite{St83}, in Subsection~\ref{subsec:Stallings}.

\begin{theorem}[Ribes--Zalesskii \cite{RZ93}] Any free group has the RZ-property.
\end{theorem}

\subsection{Quo-$\Pi$ subgroups}

\begin{definition} Let $\Pi$ be a set of prime numbers. Let $\Gamma$ be a group and $H\leq G$ is a subgroup of $G$. We say that $H$ is {\em quofinite} if $|G/H|$ is finite, equivalently, $H$ has finite index in $G$.
$H$ is a {\em quo-$\Pi$} subgroup of $\Gamma$ if for any $\gamma\in \Gamma$ and prime number $p\not\in \Pi$, if $\gamma^p\in H$ then $\gamma\in H$. If $H$ is quo-$\Pi$ and $\Pi$ is the set of all odd prime numbers, then we say that $H$ is {\em quo-odd}.
\end{definition}

It is important to note that we will deal with quo-$\Pi$ subgroups which are not quofinite.

The following is a classical result in group theory.

\begin{lemma}\label{lem:quofinitenormal} Let $\Gamma$ be a group and $H\leq \Gamma$ a quofinite subgroup of $\Gamma$. Then the group 
$N=\bigcap_{g\in \Gamma}gHg^{-1}$ is a quofinite normal subgroup of $\Gamma$ with $N\leq H$.
\end{lemma}

\begin{proof} Let $X=\Gamma/H=\{\gamma H\colon \gamma\in \Gamma\}$. Consider the action of $G$ on $X$ by left multiplication. The action induces a homomorphism $\lambda\colon \Gamma\to \mbox{Sym}(X)$ where 
$$\lambda(g)(aH)=gaH. $$
Next we show
$$ \Ker\lambda=\bigcap_{g\in \Gamma} gHg^{-1}=N. $$
First let $y\in N$. Write $y=ah_aa^{-1}$ for $a\in \Gamma$ and $h_a\in H$. Then for any $a\in \Gamma$, 
$$ \lambda(y)(aH)=ah_aa^{-1}aH=aH. $$
Thus $y\in \Ker\lambda$. Conversely, suppose $y\in \Ker\lambda$. Then for any $a\in \Gamma$, $yaH=aH$, and thus $y\in aHa^{-1}$. Thus $y\in N$. In particular, $N$ is normal.

Now we have
$$ \Gamma/\Ker\lambda=\Gamma/N\cong \lambda(\Gamma)\leq \mbox{Sym}(X). $$
Since $X$ is finite, we conclude that $N$ is of finite index in $\Gamma$.
\end{proof}

The group $N$ in this lemma is known as the {\em normal core} of $H$ in $\Gamma$.

\begin{lemma}\label{lem:quoPi} Let $\Pi$ be a set of prime numbers and let $\Gamma$ be a group. 
\begin{enumerate}
\item[(1)] If $H$ is a quo-$\Pi$ subgroup of $\Gamma$, then so is any conjugate of $H$ in $\Gamma$.
\item[(2)] If $\{H_i\}_{i\in I}$ is a family of quo-$\Pi$ subgroups of $\Gamma$, then $\bigcap_{i\in I}H_i$ is a quo-$\Pi$ subgroup of $\Gamma$.
\item[(3)] If $N\unlhd \Gamma$ is quofinite, then $N$ is a quo-$\Pi$ subgroup of $\Gamma$ if and only if $\Gamma/N$ is a $\Pi$-group.
\item[(4)] If $H\leq K\leq \Gamma$ and $H$ is quofinite and quo-$\Pi$ subgroup of $\Gamma$, then so is $K$.
\item[(5)] If $H$ is a quofinite and quo-$\Pi$ subgroup of $\Gamma$, then any prime factor of $|\Gamma/H|$ is in $\Pi$.
\end{enumerate}
\end{lemma}

\begin{proof} (1) and (2) are obvious. For (3), let $N\unlhd\Gamma$. First suppose $N$ is quo-$\Pi$ subgroup of $\Gamma$. If $\Gamma/N$ is not a $\Pi$-group, then there is a prime number $p\not\in \Pi$ so that $p$ is a factor of $|\Gamma/N|$. Let $gN$, $g\in \Gamma$, be a nonidentity element of order $p$ in $\Gamma/N$. Then $gN\neq N$ and $g^pN=N$. Thus $g^p\in N$ and $g\not\in N$. This contradicts our assumption. Next suppose $N$ is a $\Pi$-group. Let $p\not\in \Pi$ be a prime number. Assume $g\in \Gamma$ is such that $g^p\in N$. Then $g^pN=N$ and so $gN$ has order $p$ in the group $\Gamma/N$. This implies that $gN=N$ or $g\in N$. Thus $N$ is a quo-$\Pi$ subgroup of $\Gamma$. 

For (4), it is obvious that $K$ is quofinite. Assume $p\not\in \Pi$ is a prime number and $g\in \Gamma$ is such that $g^p\in K$ and $g\not\in K$. We claim that for any $1\leq i<p$, $g^i\not\in K$. To see this, suppose $g^i\in K$ for some $1\leq i<p$. Since $i$ and $p$ are coprime, there exist integers $a, b\in \mathbb{Z}$ such that $ai+bp=1$. This implies that $g=g^{ai+bp}\in K$, a contradiction. We also claim that there is a positive integer $n$ such that $g^n\in H$. To see this, assume $g^n\not\in H$ for any positive integer $n$. Then $H, gH, g^2H, \dots$ are all distinct cosets of $H$, contradicting the assumption that $H$ is quofinite in $\Gamma$. Now let $n$ be the least positive integer such that $g^n\in H$. Since $H\leq K$, this $n$ is a multiple of $p$, say $n=kp$. Now  $g^n=(g^k)^p\in H$. Since $H$ is pro-$\Pi$, we have $g^k\in H$. This contradicts the minimality of $n$. Thus $K$ is quo-$\Pi$.

For (5), suppose $H$ is a quofinite and quo-$\Pi$ subgroup of $\Gamma$. Let $N=\bigcap_{g\in \Gamma}gHg^{-1}$. Then by (1) and (2), $N$ is a quo-$\Pi$ subgroup of $\Gamma$. By Lemma~\ref{lem:quofinitenormal}, $N$ is a quofinite normal subgroup of $\Gamma$. By (3), $\Gamma/N$ is a $\Pi$-group. Since $|\Gamma/H|$ is a factor of $|\Gamma/N|$, we conclude that any prime factor of $|\Gamma/H|$ is in $\Pi$.
\end{proof}

\begin{lemma}[Herwig--Lascar \cite{HL00}] Let $\Gamma$ be a group and $H\leq G$ be a subgroup of $\Gamma$. Suppose $H$ is closed in the pro-odd topology of $\Gamma$. Then $H$ is a quo-odd subgroup of $\Gamma$.
\end{lemma}

\begin{proof} Suppose $\gamma\in \Gamma\setminus H$ and $\gamma^2\in H$. Since $H$ is closed in the pro-odd topology of $\Gamma$, there is a continuous homomorphism $\varphi$ from $\Gamma$ onto a finite odd-group $A$ such that, letting $B=\varphi[H]$ and $a=\varphi(\gamma)$, then $a^2\in B$ but $a\not\in B$. Now consider
$$ G_1=\{g\in A\colon gB=B, gaB=aB\} $$
and
$$ G_2=\{g\in A\colon \{gB, gaB\}=\{B, aB\}\}. $$
Then $G_1, G_2$ are both subgroups of $A$. Note that $a\in G_2\setminus G_1$, so we have $|G_2\colon G_1|=2$. This contradicts the assumption that $|A|$ is odd.
\end{proof}

\begin{remark}{\em Even in finite groups, two subgroups of odd index can have an intersection of even index. For example, if $\Gamma=S_3$ and $H_1, H_2$ are distinct subgroups of order $2$, then $|\Gamma\colon H|=|\Gamma\colon K|=3$ but $|\Gamma\colon H\cap K|=6$. More generally, if $\Gamma$ is a group of order $p^ak$ with $\mbox{gcd}(p,k)=1$, $k>1$, and $\Gamma$ has two distinct Sylow $p$-subgroups $H$ and $K$, then $|G:H|\, |G:K|=k^2$, but $|G:H\cap K|$ is divisible by $p$. }
\end{remark}

\subsection{$\Pi$-LERF and the $\Pi$-RZ-property}

\begin{definition} Let $\Pi$ be a set of prime numbers and let $\Gamma$ be a group.
\begin{enumerate}
\item[(i)] $\Gamma$ is {\em $\Pi$-LERF} if any finitely generated quo-$\Pi$ subgroup of $\Gamma$ is closed in the pro-$\Pi$ topology of $\Gamma$.
\item[(ii)] $\Gamma$ has the {\em $\Pi$-RZ-property} if for any finite generated quo-$\Pi$ subgroups $H_1, \dots, H_n$ of $\Gamma$, $H_1\cdots H_n$ is closed in the pro-$\Pi$ topology of $\Gamma$.
\end{enumerate}
\end{definition}

The following theorem of Ribes--Zalesskii shows that the $\Pi$-LERF and the $\Pi$-RZ-property are equivalent in a finitely generated free group.

\begin{theorem}[Ribes--Zalesskii {\cite[Theorem 5.1]{RZ94}}] Let $F$ be a finitely generated free group. Let $\mathcal{C}$ be an extension-closed variety of finite groups. Suppose $H_1, \dots, H_n$ are finitely generated subgroups of $F$ which are closed in the pro-$\mathcal{C}$ topology on $F$. Then $H_1\cdots H_n$ is closed in the pro-$\mathcal{C}$ topology on $F$. Consequently, the following are equivalent for any $\Pi$:
\begin{enumerate}
\item[(a)] Any finitely generated free group is $\Pi$-LERF.
\item[(b)] Any finitely generated free group has the $\Pi$-RZ-property.
\end{enumerate}
\end{theorem}

A more general version of the above theorem has beem proved by Auinger and Steinberg \cite[Theorem 3.5]{AS05}.

The following theorem of Herwig--Lascar connects the odd-LERF and the EPPA for the class of all finite tournaments.

\begin{theorem}[Herwig--Lascar \cite{HL00}]\label{thm:HL} The following are equivalent:
\begin{enumerate}
\item[(a)] The class of all finite tournaments has the EPPA.
\item[(b)] Any finitely generated free group is odd-LERF.
\end{enumerate}
\end{theorem}

\subsection{The $\Pi$-extension property}

Let $X$ be a set and $P$ be a set of partial permutations of $X$. Let $F(P)$ be the free group with alphabet $P$. Without loss of generality, we assume $P=P^{-1}$, where for $p\in P$, $p^{-1}$ is both the formal inverse of $p$ and the partial permutation inverse to $p$. For $w=p_1\dots p_n\in F(P)$ and $x\in X$, 
$$ w(x)=(p_1\circ \cdots \circ p_n)(x). $$
Thus $w(x)$ is only defined whenever all of $$p_n(x), (p_{n-1}\circ p_n)(x), \dots, (p_1\circ \cdots \circ p_n)(x)$$ are defined.

\begin{definition} Let $X$ be a set and $P$ be a set of partial permutations of $X$. Let $\Pi$ be a set of prime numbers. $P$ is {\em quo-$\Pi$} if for any $w\in F(P)$, $x\in X$ and prime number $p\not\in\Pi$, whenever $w^p(x)=x$, we have $w(x)=x$. 
\end{definition}

\begin{definition} Let $\Pi$ be a set of prime numbers. The {\em $\Pi$-extension property} ({\em for partial permutations}), or {\em $\Pi$-EP} for short, is the statement:
\begin{quote}
For any finite set $X$ and a set $P$ of partial permutations of $X$ which is quo-$\Pi$, there is a finite set $Y\supseteq X$ and extensions $p\mapsto \tilde{p}$ from $P$ to $\mbox{Sym}(Y)$ such that the group generated by $\{\tilde{p}\colon p\in P\}$ is a $\Pi$-group.
\end{quote}
\end{definition}

As usual, we are interested in the {\em odd-extension property}, or {\em odd-EP} for short, which is the $\Pi$-extension property where $\Pi$ is the set of all odd prime numbers. 

\begin{lemma}\label{lem:fullPi} Let $X$ be a set and $P$ be a set of permutations of $X$, i.e., $P\subseteq\mbox{\rm Sym}(X)$. Then $P$ is quo-$\Pi$ if and only if the group generated by $P$ is a $\Pi$-group.
\end{lemma}

\begin{proof} First suppose $P$ is quo-$\Pi$. To see that the group generated by $P$ is a $\Pi$-group, assume $q\not\in \Pi$ and $w\in F(P)$ is such that $w^q$ is the identity as a permutation of $X$. In particular, for any $x\in X$, $w^q(x)=x$. This implies that $w(x)=x$. Thus $w$ is the identity as a permutation of $X$. 

Conversely, suppose the group generated by $P$ is a $\Pi$-group. Assume $q\not\in\Pi$, $x\in X$, $w\in F(P)$ and $w^q(x)=x$. Suppose the order of $w$ is $m$. Then all prime factors of $m$ are in $\Pi$. Thus $\mbox{gcd}(q, m)=1$. Then for some $a, b\in \mathbb{Z}$, we have $aq+bm=1$. This implies that $w(x)=w^{aq+bm}(x)=(w^q)^a(x)=x$. Thus $P$ is quo-$\Pi$.
\end{proof}

We can also consider the {\em extension property}, which is the following statement:
\begin{quote}
For any finite set $X$ and a set $P$ of partial permutations of $X$, there are extensions $p\mapsto \tilde{p}$ from $P$ to $\mbox{Sym}(X)$ such that the group generated by $\{\tilde{p}\colon p\in P\}$ is a finite group.
\end{quote}
But this statement is trivially true since every partial permutation of $X$ can be extended to a full permutation of $X$, and $\mbox{Sym}(X)$ is finite.

\section{Equivalents of the $\Pi$-EP}
Our main goal in this section is to show that the EPPA for finite tournaments is equivalent to the odd-extension property. However, along the way we will show some other equivalences, for instance, the equivalence of the $\Pi$-EP with the $\Pi$-LERF. We will make heavy use of the concepts and methods in Stallings' proof \cite[Corollary 6.3]{St83} of Hall's theorem \cite{Ha49} (Theorem~\ref{thm:Hall}). Thus in the first two subsections we recall the concepts of multidigraphs, fundamental groups, immersions and foldings. In Subsection~\ref{subsec:Stallings} we give Stallings' proof of Hall's theorem. Then in the rest of the subsections we prove the equivalence of the $\Pi$-EP with the $\Pi$-LERF, and with the $\Pi$-EP with two partial permutations. These lead to a proof of the equivalence of the EPPA for finite tournaments with the odd-extension property. 


\subsection{The fundamental group of a multidigraph\label{subsec:multidigraph}} Following \cite{St83}, we will work with multidigraphs of Serre's style \cite{SerreBook}. Let $V$ be a set. An {\em edge} over $V$ is an object with a unique {\em source} $\mathsf{s}(e)\in V$ and a unique {\em target} $\mathsf{t}(e)\in V$. We view an edge $e$ over $V$ as an {\em arrow} from $\mathsf{s}(e)$ to $\mathsf{t}(e)$. Let $G=(V,E)$ be a {\em multidigraph}, i.e., $V$ is a set of {\em vertices} and $E$ is a set of {\em edges} over $V$.  For each $e\in E$, let $e^{-1}$ be a unique edge with $\mathsf{s}(e^{-1})=\mathsf{t}(e)$ and $\mathsf{t}(e^{-1})=\mathsf{s}(e)$. Let $E^{-1}$ be the set of all edges $e^{-1}$ for $e\in E$. We always assume $E\cap E^{-1}=\varnothing$ and formally define, for each $e\in E$, that $(e^{-1})^{-1}=e$. A multidigraph can be viewed as a digraph with multiple edges with labels (each edge is labelled by itself).

Suppose $G$ is {\em weakly connected}, i.e., $G$ is connected as an undirected multigraph. Let $v\in V$. We define the {\em fundamental group} of $G$ with base point $v$, denoted $\pi_1(G, v)$, as follows. 

An edge $e\in E(G)$ with $\mathsf{s}(e)=\mathsf{t}(e)$ is called a {\em loop}. A {\em walk} in $G$ is a sequence $w=e_0e_1\cdots e_k$ of edges in $E\cup E^{-1}$ such that for all $0\leq i<k$, $\mathsf{t}(e_i)=\mathsf{s}(e_{i+1})$. The {\em length} of a walk $w=e_0\cdots e_k$ is $k+1$, i.e., the number of edges appearing in $w$. A walk $w=e_0 \cdots e_k$ is a {\em circuit} if $\mathsf{t}(e_k)=\mathsf{s}(e_0)$. The {\em base point} of a circuit $w=e_0\cdots e_k$ is $\mathsf{s}(e_0)$. For each $v\in V$, there is a unique walk of length $0$ with base point $v$, which is the empty walk (sequence), which we also consider a circuit with base point $v$. Two walks $w_1, w_2$ are {\em elementarily homotopy equivalent} if one of them is obtained from the other by removing two adjacent terms of the form $e, e^{-1}\in E\cup E^{-1}$. A {\em reduced walk} is a walk without adjacent terms of the form $e, e^{-1}$. The reflexive and transitive closure of the elementary homotopy equivalence relation is called the {\em homotopy equivalence}. For each walk $w$ in $G$, the homotopy equivalence class of $w$ is denoted $[w]$. A reduced walk is the shortest walk in its homotopy equivalence class. Let $\pi_1(G, v)$ be the set of all homotopy equivalence classes of  circuits in $G$ with base point $v$. For circuits $w, w_1, w_2$ with base point $v$, define
$$ [w]^{-1}=[w^{-1}] \mbox{ and } [w_1]\cdot[w_2]=[w_1w_2]. $$
These operations make $\pi_1(G, v)$ a group, which we call the {\em fundamental group} of $G$ with base point $v$.

Let $G$ and $H$ be multidigraphs. A {\em homomorphism} $f\colon G\to H$ consists of a function $f_V\colon V(G)\to V(H)$ and a function $f_E\colon E(G)\to E(H)$ such that for any $e\in E(G)$, 
$$ f_V(\mathsf{s}(e))=\mathsf{s}(f_E(e)) \mbox{ and } f_V(\mathsf{t}(e))=\mathsf{t}(f_E(e)). $$
If $f\colon G\to H$ is a homomorphism, then for any walk $w=e_0\cdots e_k$ in $G$, $f(w)=f_E(e_0)\cdots f_E(e_k)$ is a walk in $H$; if $w$ is a circuit with base point $v$, then $f(w)$ is a circuit in $H$ with base point $f_V(v)$. In this sense $f$ induces a homomorphism from $\pi_1(G, v)$ to $\pi_1(H, f_V(v))$, which we still denote by $f$. We will drop the subscripts in $f_E$ and $f_V$ when there is no danger of confusion.

We consider a special multidigraph $G$ with only one vertex $v$. In this case there is no need to specify the base point when we speak of circuits in $G$, and we denote the fundamental group by $\pi_1(G)$. It is clear that $\pi_1(G)$ is the free group generated by the set $E(G)$. When $E(G)=\varnothing$, $\pi_1(G)$ is the trivial group.

\subsection{Immersions and foldings} 

Let $G$ be a multidigraph. For any $v\in V(G)$, the {\em star} of $v$ is the set
$$ \mbox{St}(v, G)=\{e\in E(G)\cup E(G)^{-1}\colon \mathsf{s}(e)=v\}. $$
Let $G$ and $H$ be multidigraphs and $f\colon G\to H$ be a homomorphism. Then for each $v\in V(G)$, 
$$ f(\mbox{St}(v, G)\cap E(G))\subseteq \mbox{St}(f(v), H)\cap E(H)$$
and $$ f(\mbox{St}(v, G)\cap E(G)^{-1})\subseteq \mbox{St}(f(v), H)\cap E(H)^{-1}. $$
Thus, $f$ gives rise to a function $f_v$ from $\mbox{St}(v, G)$ to $\mbox{St}(f(v), H)$. $f$ is an {\em immersion} if for each $v\in V(G)$, $f_v$ is injective; $f$ is a {\em covering} if for each $v\in V(G)$, $f_v$ is bijective.

A pair of edges $(e_1, e_2)$ in $G$ is {\em admissible} if either $\mathsf{s}(e_1)=\mathsf{s}(e_2)$ or $\mathsf{t}(e_1)=\mathsf{t}(e_2)$ (or both). If $(e_1,e_2)$ is an admissible pair in $G$, define $G/[e_1=e_2]$ to be the result of identifying $e_1$ and $e_2$ in $G$. The operation is called {\em folding} $e_1$ and $e_2$. The {\em folding map} $f_{e_1=e_2}\colon G\to G/[e_1=e_2]$ is naturally defined and is a homomorphism. 

The following fact is immediate from the definitions.

\begin{lemma}[{\cite[3.3]{St83}}] Let $G, H$ be multidigraphs and let $f\colon G\to H$ be a homomorphism. Suppose $(e_1, e_2)$ is a pair of edges in $G$ such that $f(e_1)=f(e_2)$ and either $\mathsf{s}(e_1)=\mathsf{s}(e_2)$ or $\mathsf{t}(e_1)=\mathsf{t}(e_2)$. Then $(e_1, e_2)$ is an admissible pair in $G$. Moreover, there is a unique homomorphism $g\colon G/[e_1=e_2]\to H$ such that $f=g\circ f_{e_1=e_2}$. 
\end{lemma}


In the situation of the above lemma, we say that $f$ {\em folds} $e_1$ and $e_2$, and that $f$ {\em factors} through $f_{e_1=e_2}$ (or $G/[e_1=e_2]$).

\begin{lemma}\label{lem:folds} Let $G$ be a finite multidigraph, let $H$ be a multidigraph, and let $f\colon G\to H$ be a homomorphism. Then there exist folding maps $f_i\colon G_i\to G_{i+1}$, $i=0,\dots, k$, where $G_0=G$ and an immersion $g\colon G_{k+1}\to H$ such that $g\circ f_k\circ \dots \circ f_0=f$. Moreover, $G_{k+1}$ and $g$ are unique.
\end{lemma}

\begin{proof} If $f$ is an immersion then the lemma holds trivially. Otherwise, $f$ is not an immersion, and there is $v\in V(G)$ such that $f_v$ is not injective on $\mbox{St}(v, G)$. This means that there exist $e_1, e_2\in \mbox{St}(v, G)$ such that $f(e_1)=f(e_2)$. If $e_1, e_2\in E(G)$ then $(e_1, e_2)$ is an admissible pair in $G$ by the preceding lemma. If $e_1, e_2\in E(G)^{-1}$ then $(e_1^{-1}, e_2^{-1})$ is an admissible pair in $G$. Note that the number of edges in $G/[e_1=e_2]$ or $G/[e_1^{-1}= e_2^{-1}]$ is less than that of $G$. This implies that the folding procedure will terminate after finitely many steps. At the end of this iterative procedure of folding, the result is an immersion.

To prove the uniqueness of $G_{k+1}$ and $g$, we need to analyze the folding procedure in more detail. First, for $G=G_0$, we consider the set
$$ \{(e_1, e_2)\in E(G)^2\colon f(e_1)=f(e_2) \mbox{ and } (\mathsf{s}(e_1)=\mathsf{s}(e_2) \mbox{ or } \mathsf{t}(e_1)=\mathsf{t}(e_2)) \} $$
of all admissible pairs given by $f$. In any sequence of foldings, these pairs remain admissible until they are folded at a particular step. Moreover, each of the foldings of these pairs is commutative with any other folding performed on $G$ or $G_i$ for $i>0$. Thus, we may assume that the sequence of foldings $f_0, \dots, f_k$ begins with the foldings of these pairs (in any order). Now it is obvious that such an iterative procedure of folding is unique and results in a unique $G_{k+1}$.
\end{proof}

The following facts are theorems of \cite{St83}. For the convenience of the reader we give their proofs here.

\begin{lemma}[{\cite[Corollary 4.4]{St83}}] \label{lem:surj} Let $G$ be a multidigraph, $v\in V(G)$ and let $(e_1,e_2)$ be an admissible pair in $G$. Then the folding map $f=f_{e_1=e_2}$ induces a surjective homomorphism from $\pi_1(G, v)$ to $\pi_1(G/[e_1=e_2], f(v))$. 
\end{lemma}

\begin{proof} Without loss of generality assume $\mathsf{s}(e_1)=\mathsf{s}(e_2)$. The case $\mathsf{t}(e_1)=\mathsf{t}(e_2)$ is similar. Suppose $W$ is a reduced circuit in $G/[e_1=e_2]$ with base point $f(v)$. First consider the case $\mathsf{t}(e_1)=\mathsf{t}(e_2)$. In this case let $w$ be the result obtained from $W$ by replacing all occurrences of $f(e_1)=f(e_2)$ by $e_1$ (and replacing all $f(e_1)^{-1}$ by $e_1^{-1}$). Then $w$ is a circuit in $G$ with base point $v$, and $f(w)=W$. The lemma is proved in this case. So assume $\mathsf{t}(e_1)\neq \mathsf{t}(e_2)$. In this case, we note that any maximal block (subwalk) of $W$ of the form $f(e_1)^\ell (\ell>0)$ followed by an edge $e\in E(G)\cup E(G)^{-1}$ can be uniquely lifted to a walk in $E(G)$ depending on $\mathsf{s}(e)$, and similarly for any maximal subwalk of $W$ of the form $f(e_1)^{-\ell} (\ell>0)$. In case $f(e_1)^2$ or $f(e_1)^{-2}$ occurs in $W$, it follows that exactly one of $e_1$ and $e_2$ is a loop. In all cases there is a unique lifting of $W$ to a circuit $w$ in $G$ with $f(w)=W$.
\end{proof}

\begin{lemma}[{\cite[Proposition 5.3]{St83}}] \label{lem:im2inj} Let $f\colon G\to H$ be an immersion and $v\in V(G)$. Then the group homomorphism $f\colon \pi_1(G, v)\to \pi_1(H, f(v))$ is injective.
\end{lemma}

\begin{proof} Let $\gamma\in \pi_1(G, v)$ be a nonidentity element. Then $\gamma=[w]$ for some reduced circuit $w$ of length $>0$ with base point $v$. Then $f(\gamma)=[f(w)]$ and $f(w)$ is a reduced circuit with base point $f(v)$. Moreover, $f(w)$ has length $>0$, and therefore $[f(w)]$ is not the identity. So $f(\gamma)$ is not the identity.
\end{proof}

\begin{lemma}[{\cite[5.4]{St83}}] \label{lem:rep} Let $H$ be a multidigraph, $u\in V(H)$ and let $S$ be a finitely generated subgroup of $\pi_1(H, u)$. Then there exist a finite multigraph $G$, $v\in V(G)$ and an immersion $f\colon G\to H$ so that $S=f(\pi_1(G, v))$.
\end{lemma}

\begin{proof} Suppose $S$ is generated by $\alpha_1, \dots, \alpha_n\in \pi_1(H, u)$. For each $1\leq i\leq n$, let $w_i$ be a circuit with base point $u$ such that $\alpha_i=[w_i]$. Let $G_0$ be a multidigraph constructed as the {\em wedge sum} of $n$ disjoint cycles of length $|w_1|, \dots, |w_n|$, respectively, with a base point $v$. Let $f_0\colon G_0\to H$ be the homomorphism of multidigraphs which sends the $i$-th cycle to the circuit $w_i$ for all $1\leq i\leq n$. Observe that $f_0(\pi_1(G_0, v))=S$. 

By Lemma~\ref{lem:folds}, there exist foldings $g_i\colon G_i\to G_{i+1}$, $i=0,\dots, k$, and an immersion $f\colon G_{k+1}\to H$ such that $f_0=f\circ g_k\circ \dots \circ g_0$. Let $G=G_{k+1}$. 

Now each group homomorphism $g_i\colon \pi_1(G_i, v)\to \pi_1(G_{i+1}, v)$ is surjective by Lemma~\ref{lem:surj}. It follows that $f(\pi_1(G, v))=S$. 
\end{proof}

\subsection{Stallings' proof of Hall's theorem\label{subsec:Stallings}}
In this subsection we recall Stallings' proof of Hall's theorem (Theorem~\ref{thm:Hall}) using the concepts in the previous two subsections.

\begin{theorem}[{\cite[Theorem 6.1]{St83}}] \label{thm:immersion2covering} Let $f\colon G\to H$ be an immersion of multidigraphs. Suppose that $H$ has only one vertex and $G$ has finitely many vertices. Then there exists a graph $G'$ extending $G$, with $G'-G$ consisting only of edges, and there exists a homomorphism $f'\colon G'\to H$ extending $f$, such that $f'$ is a covering. 
\end{theorem}

\begin{proof} For each $e\in E(H)$, let
$$ R_e=\{(u, v)\in V(G)^2\colon \exists e_1\in E(G)\ (f(e_1)=e\mbox{ and } \mathsf{s}(e_1)=u, \mathsf{t}(e_1)=v)\}. $$
Since $f$ is an immersion, the projections of $R_e$ to both coordinates are injective. It follows that $R_e$ is a bijection between two subsets of $V(G)$. Since $V(G)$ is finite, $R_e$ can be extended to a full bijection $S_e$ of $V(G)$ to $V(G)$.

Now define $G'$ by letting $V(G')=V(G)$ and 
$$ E(G')=\{(u,v, e)\colon V(G)^2\times E(H)\colon (u, v)\in S_e\}. $$
For each $(u,v, e)\in E(G')$, $\mathsf{s}(u,v,e)=u$ and $\mathsf{t}(u,v,e)=v$; $(u,v,e)^{-1}=(v,u,e^{-1})$.

Define $f'\colon G'\to H$ by letting $f'(u,v,e)=e$. Then $G'$ extends $G$ and $f'$ extends $f$. It is easy to verify that $f'$ is a covering.
\end{proof}

\begin{theorem}[Hall \cite{Ha49}]\label{thm:Hall2} Let $F$ be a finitely generated free group and $S\leq F$ be a finitely generated subgroup of $F$. Let $\gamma\in F\setminus S$. Then there exists a subgroup $S'$ of finite index in $F$ such that $S\leq S'$ and $\gamma\not\in S'$.
\end{theorem}

\begin{proof} (Stallings \cite{St83}) Let $H$ be a multidigraph with only one vertex and with $m$ edges, where $m$ is the number of  free generators of $F$. Then $\pi_1(H)=F$. Let $\alpha_1,\dots, \alpha_n$ be a set of generators for $S$. Let $G_0$ be the wedge sum of $n$ disjoint cycles $w_1, \dots, w_n$ of lengths $|\alpha_1|, \dots, |\alpha_n|$, respectively, and a walk $z$ of length $|\gamma|$ with base point $v$. Let $f_0\colon G_0\to H$ be the homomorphism which sends $\alpha_1, \dots, \alpha_n, \gamma$ to the corresponding circuits and walks in $H$. Then $f_0(\pi_1(G_0, v))=S$.

By Lemma~\ref{lem:folds}, there exist foldings $g_i\colon G_i\to G_{i+1}$, $i=0,\dots, k$, and an immersion $f\colon G_{k+1}\to H$ such that $f_0=f\circ g_k\circ \dots\circ g_0$. Let $G=G_{k+1}$. Then similar to Lemma~\ref{lem:rep}, we have $f(\pi_1(G, v))=S$. 

Let $g=g_k\circ \dots\circ g_0$. Then images of $w_1, \dots, w_n$ under $g$ are still circuits in $G$. Since $\gamma\not\in S$, the image of $z$ under $g$ is not a circuit in $G$. 

By Theorem~\ref{thm:immersion2covering}, there exist an extension $G'$ of $G$ and an extension $f'\colon G'\to H$ of $f$ so that $f'$ is a covering. Let $S'=f'(\pi_1(G',v))$. The extension map from $G$ to $G'$ is an immersion, and by Lemma~\ref{lem:im2inj}, $S\leq S'$. Since $f'$ is a covering, each vertex in $G'$ corresponds to a right coset of $S'$ in $F$. This implies that $S'$ has finite index in $F$. Since $z$ is not a circuit in $G$, we have $\gamma\not\in S'$. 
\end{proof}

To see that Theorem~\ref{thm:Hall} follows from the above theorem, let $N$ be the normal core of $S'$ in $F$. Then $N$ still has finite index in $F$ by Lemma~\ref{lem:quofinitenormal}. Since $NS\leq S'$, we have $\gamma\not\in NS$, and therefore $\gamma N\cap S=\varnothing$. This proves that $S$ is closed in the profinite topology of $F$.

\subsection{The equivalence of $\Pi$-EP and $\Pi$-LERF}

\begin{theorem}\label{thm:EPtoLERF} Let $\Pi$ be a set of prime numbers. Suppose the $\Pi$-extension property holds. Then any finitely generated free group is $\Pi$-LERF.
\end{theorem}

\begin{proof} Let $F$ be a free group with $m$ generators. Let $H$ be a multidigraph with only one vertex and with $m$ edges. Then $\pi_1(H)=F$. Let $S$ be a finitely generated subgroup of $F$ which is quo-$\Pi$. Let $\alpha_1,\dots, \alpha_n$ be a set of generators of $S$. Let $\gamma\in F\setminus S$. Let $G_0$ be the wedge sum of $n$ disjoint cycles $w_1, \dots, w_n$ of lengths $|\alpha_1|, \dots, |\alpha_n|$, respectively, and a walk $z$ of length $|\gamma|$, with base point $v$. Let $f_0\colon G_0\to H$ be a homomorphism that sends $w_1,\dots, w_n, z$ to their corresponding circuits or walk in $H$. Then $f_0(\pi_1(G_0,v))=S$. 

Apply Lemma~\ref{lem:folds} to obtain a sequence of foldings and an immersion. Let $G$ be the unique multidigraph, $g\colon G_0\to G$ be the composition of the foldings and $f\colon G\to H$ be the immersion so that $f_0=f\circ g$. Then $f(\pi_1(G, v))=S$. Note that $g(w_1), \dots, g(w_n)$ remain circuits in $G$ with base point $v$, and $g(z)$ is not a circuit in $G$ since $\gamma\not\in S$. 

For each $e\in E(H)$, let
$$ p_e=\{(u,v)\in V(G)^2\colon \exists e_1\in E(G)\ (f(e_1)=e \mbox{ and } \mathsf{s}(e_1)=u, \mathsf{t}(e_1)=v)\}. $$
Since $f$ is an immersion, the projections of $p_e$ to both coordinates are injective. Thus $p_e$ is a partial permutation of $V(G)$. We claim that $P=\{p_e\colon e\in E(H)\}$ is quo-$\Pi$. For this, suppose $w\in F(P)$, $q\not\in \Pi$ and $x\in V(G)$ are such that $w^q(x)=x$. This gives rise to a circuit $W^q$ with base point $x$. Note that $G$ is weakly connected, and therefore there is a walk $Z$ from $v$ to $x$. The walk $(Z^{-1}WZ)^q$ is a circuit with base point $v$. Hence $[(Z^{-1}WZ)^q]\in S$. Since $S$ is quo-$\Pi$ and $[(Z^{-1}WZ)^q]=[Z^{-1}WZ]^q$, we conclude that $[Z^{-1}WZ]\in S$ and therefore $Z^{-1}WZ$ is a circuit with base point $v$. Thus $W$ is a circuit with base point $x$, and so $w(x)=x$. This shows that $P$ is quo-$\Pi$.

Now by the $\Pi$-extension property, one can extend all elements of $P$ to full permutations $\tilde{P}$ on a finite set $X\supseteq V(G)$ so that the permutation group generated by $\tilde{P}$, denoted by $Q$, is a $\Pi$-group. Define a multidigraph $G'$ with $V(G')=X$ and
$$ E(G')=\{(u,v,\tilde{p}_e)\colon \tilde{p}_e\in \tilde{P}, \tilde{p}_e(u)=v\}. $$
Then $G'$ extends $G$. Since the extension map from $G$ to $G'$ is an immersion, we have $\pi_1(G, v)\leq \pi_1(G',v)$. Let $f'\colon G'\to H$ be the homomorphism such that $f'(u,v,\tilde{p}_e)=e$. Then $f'$ extends $f$, $f'$ is a covering, and $f'(\pi_1(G',v))\geq S$. 

Let $S'=f'(\pi_1(G', v))$. Then every vertex in $G'$ corresponds to a right coset of $S'$ in $F$. Consider the right action of $F$ on $F/S'$ by 
$$ \eta\cdot S'\beta=S'\beta\eta. $$
This induces a group homomorphism from $F$ to the permutation group of $F/S'$. Note that the kernel of this homomorphism is exactly the normal core of $S'$, which we denote by $N$. 
On the other hand, the above action corresponds to the following action of $F(P)$ on $X$: for each $e\in E(H)$, 
$$ p_e\cdot x=\tilde{p}_e(x). $$
Correspondingly, this induces a group homomorphism of $F(P)$ to the permutation group of $X$, with the image of the homomorphism exactly $Q$. It thus follows that $F/N$ is isomorphic to $Q$ as a group. Since $Q$ is a $\Pi$-group, $N$ is thus open in the pro-$\Pi$ topology of $F$. 

Now, since $g(z)$ is not a circuit in $G$, it is not a circuit in $G'$, and it follows that $\gamma\not\in S'$. Since $NS\leq S'$, we have $\gamma\not\in NS$, and therefore $\gamma N\cap S=\varnothing$. This shows that $S$ is closed in the pro-$\Pi$ topology of $F$. Hence $F$ is $\Pi$-LERF.
\end{proof}

The method of proof of the following theorem is essentially from \cite{HL00}.

\begin{theorem}\label{thm:LERF2EP} Let $\Pi$ be a set of prime numbers. Suppose that any finitely generated free group is $\Pi$-LERF. Then the $\Pi$-extension property holds.
\end{theorem}

\begin{proof} Suppose $X$ is a finite set and $P$ is a set of partial permutations of $X$. Assume that $P$ is quo-$\Pi$. For each $x\in X$, let
$$ X_x=\{y\in X\colon \exists w\in F(P)\ y=w(x)\}. $$
Then $P\upharpoonright X_x$ is a set of partial permutations of $X_x$ and $P\upharpoonright X_x$ is still quo-$\Pi$. Suppose there exists $X'_x$ extending $X_x$ and $\tilde{P}_x$ a set of extensions of $P\upharpoonright X_x$ that are permutations of $X'_x$ so that the permutation group $\Gamma_x$ of $X'_x$ generated by $\tilde{P}_x$ is a $\Pi$-group. Then $\{X_x\colon x\in X\}$ is a partition of $X$ into equivalence classes. Let $Y$ be a transversal of this equivalence relation. Let 
$$ X'=\bigsqcup_{x\in Y} X'_x. $$
For each $p\in P$, let 
$$ \tilde{p}=\bigsqcup_{x\in Y} \tilde{p}_x. $$
Then the group generated by $\{\tilde{p}\colon p\in P\}$ is a subdirect product of $\Gamma_x$, $x\in Y$. Since each $\Gamma_x$ is a $\Pi$-group, so is the group generated by $\{\tilde{p}\colon p\in P\}$. This argument shows that we may assume without loss of generality that $X=X_x$ for some (any) $x\in X$.

For the rest of the proof fix a particular $x\in X$. Let
$$ S=\{w\in F(P)\colon w(x)=x\}. $$
Then $S$ is a finitely generated subgroup of $F(P)$. We claim that $S$ is quo-$\Pi$. Suppose $q\not\in \Pi$ and $w\in F(P)$ with $w^q\in S$. Then $w^q(x)=x$. Since $P$ is quo-$\Pi$, we have $w(x)=x$ and therefore $w\in S$. This shows that $S$ is quo-$\Pi$. 

It follows from the assumption that $F(P)$ is $\Pi$-LERF that $S$ is closed in the pro-$\Pi$ topology of $F(P)$. Thus there is a normal subgroup $N$ of $F(P)$ such that $F(P)/N$ is a $\Pi$-group satisfying 
\begin{quote}
for each distinct pair $y, z\in X$, letting $u, v\in F(P)$ be arbitrary such that $y=u(x)$ and $z=v(x)$, we have that $u^{-1}vN\cap S=\varnothing$.
\end{quote} 
Note that for $u, v$ such that $u(x)\neq v(x)$, we have $u^{-1}v\not\in S$; only finitely many such pairs of words are considered in obtaining $N$. Now define an map $\phi\colon X\to F(P)/NS$ by $\phi(y)=wNS$ where $w\in F(P)$ is any such that $w(x)=y$. Then it is easy to check that $\phi$ is well defined and injective. Moreover, for each $p\in P$, the action $p\cdot wNS=pwNS$ extends the action of $p$ on $X$, where $X$ is identified with $\phi(X)$. 

Now define an action of $F(P)/N$ on $F(P)/NS$ by
$$ uN\cdot wNS=uwNS. $$
Define $\psi\colon P\to F(P)/N$ by $\psi(p)=pN$. We claim that $\psi$ is an embedding from $P$ to $F(P)/N$ as a set of permutations of $F(P)/NS$. To see this, let $p, q\in P$ be such that $p\neq q$. Then there is $y\in X$ such that $p(y)\neq q(y)$. Let $w\in F(P)$ be such that $y=w(x)$. Then $pw(x)\neq qw(x)$. Thus $(qw)^{-1}pwN\cap S=\varnothing$. This implies $pwNS\neq qwNS$, and so $\psi(p)\neq \psi(q)$. 

Thus we have that the action of $F(P)/N$ on $F(P)/NS$ extends the action of $P$ on $X$. Thus the $\Pi$-extension property holds.
\end{proof}

As a corollary of the previous two theorems and the Herwig--Lascar theorem (Theorem~\ref{thm:HL}), we obtain the equivalence between the EPPA for finite tournaments and the odd-extension property. In the following we give a direct argument for equivalence between the odd-extension property and the EPPA for finite tournaments.

\begin{theorem}\label{thm:EPPAtoOEP} The following are equivalent:
\begin{enumerate}
\item[(a)] The EPPA for finite tournaments. 
\item[(b)] The odd-extension property.
\end{enumerate}
\end{theorem}

\begin{proof} We first show (a)$\Rightarrow$(b). Let $X$ be a finite set and let $P$ be a set of partial permutations of $X$. Suppose $P$ is quo-odd, i.e., for any $w\in F(P)$ and $x\in X$, if $w^2(x)=x$ then $w(x)=x$. We define a tournament on $X$ by specifying the set $E_X$ of edges as follows.

First we note that for any pair of distinct elements $(x_0,x_1)\in X^2$, the set
$$ R_{x_0, x_1}=\{(w(x_0), w(x_1))\colon w\in F(P)\} $$
does not contain both $(x,y)$ and $(y,x)$ for any $x,y\in X$. That is, there do not exist $w, w'\in F(P)$ such that $(w(x_0), w(x_1))=(w'(x_1), w'(x_0))$. To see this, suppose $w(x_0)=w'(x_1)$ and $w(x_1)=w'(x_0)$.Then $w'^{-1}ww'^{-1}w(x_0)=x_0$. Since $P$ is quo-odd, we have $w'^{-1}w(x_0)=x_0$, i.e., $w(x_0)=w'(x_0)$. This implies that $w(x_0)=w(x_1)$ or $x_0=x_1$, a contradiction. 

Now $\{R_{x_0, x_1}\colon x_0\neq x_1\}$ is a partition of $[X]^2$ into equivalence classes. Let $Y$ be a complete set of representatives for the partition with the property that if $(x_0,x_1)\in Y$ then $(x_1, x_0)\in Y$. Then $Y=Y_0\sqcup Y_1$ where  $Y_0$ and $Y_1$ satisfy that $(x_0,x_1)\in Y_0$ if and only if $(x_1,x_0)\in Y_1$. Finally, define 
$$ E_X=\bigcup\{ R_{x_0,x_1}\colon (x_0, x_1)\in Y_0\}. $$
Then $E_X$ is a tournament on $X$.

By the EPPA, there exist a finite $Y\supseteq X$ and a tournament $E_Y\supseteq E_X$ on $Y$ so that any partial isomorphism of $E_X$ extends to an automorphism of $E_Y$. We note that each $p\in P$ is a partial isomorphism of $E_X$. That is, if $p(x_0)=y_0$ and $p(x_1)=y_1$, then $(x_0, x_1)\in E_X$ if and only if $(y_0, y_1)\in E_X$. This is because $(y_0, y_1)\in R_{x_0, x_1}$. Now let $\tilde{p}$ be an extension of $p$ to an automorphism of $E_Y$. Then $\tilde{p}$ is a permutation of $Y$. Let $\Gamma$ be the permutation group generated by $\{\tilde{p}\colon p\in P\}$. Then $\Gamma$ is an odd group. This is because, otherwise there is $g\in \Gamma$ such that $g\neq \mbox{id}$ but $g^2=\mbox{id}$. Thus there is $y\in Y$ such that $g(y)\neq y$ and $g^2(y)=y$. Without loss of generality, assume $(y, g(y))\in E_Y$. Then $g(y, g(y))=(g(y), y)\in E_Y$, contradicting to the assumption that $E_Y$ is a tournament. 

Next we show (b)$\Rightarrow$(a). For this let $T$ be a finite tournament. Let $P$ be the set of all partial isomorphisms of $T$. Then each element of $P$ is a partial permutation of $X=V(T)$. We claim that $P$ is quo-odd. To see this, let $w\in F(P)$ and $x\in X$ with $w^2(x)=x$. Let $y=w(x)$. Then $w$ as a partial isomorphism of $T$ satisfies $w(x)=y$ and $w(y)=x$, contradicting the definition of a tournament. Thus $P$ is quo-odd. 

By the odd-extension property, there is $Y\supseteq X$ and permutations $\tilde{p}\in \mbox{Sym}(Y)$ extending $p$ for all $p\in P$ such that the permutation group $\Gamma\leq \mbox{Sym}(Y)$ generated by $\{\tilde{p}\colon p\in P\}$ is odd. For each distinct pair of elements $y_0, y_1\in Y$, consider
$$ R_{y_0,y_1}=\{(g(y_0), g(y_1))\colon g\in \Gamma\}. $$
We claim that for no such pairs $(y_0, y_1)$, $(y_1, y_0)\in R_{y_0,y_1}$. To see this, suppose $g\in \Gamma$, $g(y_0)=y_1$ and $g(y_1)=y_0$. Then we have $g^2(y_0)=y_0$. By Lemma~\ref{lem:fullPi}, $\Gamma$ is quo-odd. Thus $g(y_0)=y_0$, contradicting our assumption that $g(y_0)=y_1\neq y_0$.

Now $\{R_{y_0, y_1}\colon y_0\neq y_1\}$ is a partition of $[Y]^2$ into equivalence classes. Decompose $[Y]^2$ into $Z_0\sqcup Z_1$ so that 
\begin{enumerate}
\item[(i)] for any $(x_0,x_1)\in T$, $R_{x_0,x_1}\subseteq Z_0$;
\item[(ii)] for any $(y_0, y_1)\in[Y]^2$, $(y_0, y_1)\in Z_0$ if and only if $(y_1,y_0)\in Z_1$.
\end{enumerate}
Then define a tournament $T'\supseteq T$ by letting $(y_0, y_1)\in T'$ if and only if $R_{y_0,y_1}\in Z_0$. By (ii), $T'$ is well defined. By (i), $T'$ is an extension of $T$. Now for each $p\in P$, $\tilde{p}$ is an automorphism of $T'$ extending $p$. This proves the EPPA.
\end{proof}

\subsection{$\Pi$-EP with $2$ partial permutations}

Throughout this subsection let $\Pi$ be a set of prime numbers.

\begin{definition} The {\em $\Pi$-extension property with $k$ partial permutations}, denoted $\Pi$-EP$(k)$, is the statement:
\begin{quote}
For any finite set $X$ and a set $P$ of partial permutations of $X$ with $P$ quo-$\Pi$ and $|P|\leq k$, there is a finte set $Y\supseteq X$ and extensions $p\mapsto \tilde{p}$ from $P$ to $\mbox{Sym}(Y)$ such that the generated by $\{\tilde{p}\colon p\in P\}$ is a $\Pi$-group.
\end{quote}
\end{definition}

\begin{lemma} $\Pi$-{\rm EP}$(1)$ always holds.
\end{lemma}

\begin{proof} Let $X$ be a finite set and let $p$ be a partial permutation of $X$ so that $\{p\}$ is quo-$\Pi$. Consider a digraph $G=(X, E)$ where $(x,y)\in E$ if and only if $p(x)=y$. A {\em weakly connected component} of $G$ is a connected component of $G$ as an undirected graph. Any weakly connected component of $G$ is either a directed path, a directed cycle, a loop or a singleton without edges. Since $\{p\}$ is quo-$\Pi$, any weakly connected component of $G$ which is a directed cycle must have its length $\ell$ a $\Pi$-number. Now extend $p$ to $\tilde{p}$ by the following procedure: for each directed path $\alpha$, add vertices if necessary, and extend $p\upharpoonright \alpha$ to a directed cycle of length $\ell_\alpha$ so that $\ell_\alpha$ is a $\Pi$-number; for each singleton without edges, extend $p$ to a loop. The resulting $\tilde{p}$ is a permutation of an extended set $Y\supseteq X$ such that the group generated by $\tilde{p}$ is a subdirect product of cyclic $\Pi$-groups, which is a $\Pi$-group. 
\end{proof}

\begin{theorem}\label{thm:EP2} $\Pi$-{\rm EP}$(2)$ is equivalent to the $\Pi$-extension property.
\end{theorem}
\begin{proof} We only need to show that $\Pi$-EP$(2)$ implies the $\Pi$-extension property. Let $X$ be a finite set and let $p_1,\dots, p_k$ be partial permutations of $X$ so that $P=\{p_1,\dots, p_k\}$ is quo-$\Pi$. Let $Y=X\times\{1,\dots, k\}$. Define two partial permutations $a, b$ of $Y$ as follows. For all $1\leq i\leq k$ and $x\in \mbox{dom}(p_i)$, let
$$ a(x, i)=(p_i(x), i). $$
For all $x\in X$ and $1\leq i<k$, let
$$ b(x,i)=(x,i+1). $$
We claim that $\{a, b\}$ is quo-$\Pi$. To see this, suppose $q\not\in \Pi$, $y\in Y$ and $w\in F(a,b)$ such that $w^q(y)=y$. Suppose $y=(x,i)$ and 
$$ w=b^{\beta_n}a^{\alpha_n}\cdots b^{\beta_1}a^{\alpha_1}b^{\beta_0} $$
for $\alpha_1,\dots, \alpha_n, \beta_0, \beta_1,\dots, \beta_n\in \mathbb{Z}$. We have that $q(\beta_0+\beta_1+\cdots+\beta_n)=0$. Hence $\beta_0+\beta_1+\cdots+\beta_n=0$.  Let 
$$j_0=i+\beta_0, j_1=j_0+\beta_1, \dots, j_{n-1}=j_{n-2}+\beta_{n-1}. $$ Then $x\in\dom(p_{j_0})$ and
$$ (p_{j_{n-1}}^{\alpha_n}\cdots p_{j_1}^{\alpha_2}p_{j_0}^{\alpha_1})^q(x)=x. $$
Since $P$ is quo-$\Pi$, we have
$$ p_{j_{n-1}}^{\alpha_n}\cdots p_{j_1}^{\alpha_2}p_{j_0}^{\alpha_1}(x)=x. $$
Hence $w(y)=y$. 

By $\Pi$-EP$(2)$, there are $Z\supseteq Y$ and $\tilde{a}, \tilde{b}\in \mbox{Sym}(Z)$ such that $\tilde{a}$ extends $a$,  $\tilde{b}$ extends $b$, and the group generated by $\{\tilde{a}, \tilde{b}\}$ is a $\Pi$-group. Identify $X$ with $X\times \{1\}\subseteq Y\subseteq Z$. For each $1\leq i\leq k$, let
$$ \tilde{p}_i=\tilde{b}^{-i+1}\tilde{a}\tilde{b}^{i-1}. $$
Then $\tilde{p}_i\in \mbox{Sym}(Z)$ and the group generated by $\{\tilde{p}_1, \dots, \tilde{p}_k\}$ is a subgroup of the group generated by $\{\tilde{a}, \tilde{b}\}$, and hence is a $\Pi$-group.

It remains to verify that each $\tilde{p}_i$ extends $p_i$. For this, let $x\in \dom(p_i)$. Then
$$ \tilde{p}_i(x, 1)=\tilde{b}^{-i+1}\tilde{a}\tilde{b}^{i-1}(x, 1)=\tilde{b}^{-i+1}\tilde{a}(x, i)=\tilde{b}^{-i+1}(p_i(x), i)=(p_i(x), 1). $$
The proof of the theorem is complete.
\end{proof}

\section{EP Problems and EP Maps}

Starting from this section we study the $\Pi$-EP property for partial permutations. For this we introduce the concepts of EP problems, solutions and EP maps,  and study the property of some specific EP maps. Our main result of this section is that the existence of solutions of an EP problem depends solely on its fundamental group. This leads to a generalization of a theorem of \cite{HPSW19} about EP problems with cyclic fundamental groups. We also perform a case study of an EP problem which answers a question left open by \cite{HPSW19}. In particular, we find examples which witness the failure of $p$-EP and $p$-LERF for any prime $p$.

\subsection{Definitions\label{subsec:EPproblem}}

\begin{definition}[EP problem]\label{def:EPproblem} Let $X$ be a finite set and $P$ be a set of partial permutations of $X$.
\begin{enumerate}
\item[(1)] We say that $P$ is {\em irreducible} if there do not exist distinct $p, q\in P$ such that $p\cup q$ is still a partial permutation of $X$.
\item[(2)] We say that $(X, P)$ is an {\em EP problem} if $P$ is a set of partial permutations of $X$. An EP problem $(X, P)$ is {\em irreducible} if $P$ is irreducible.
\end{enumerate}
\end{definition}

There is an obvious reduction process to turn any set of partial permutations into an irreducible one by repeatedly combining partial permutations that can be combined. 

\begin{definition}[Subproblem]\label{def:subproblem} Let $(X, P)$ and $(Y, Q)$ be EP problems.
\begin{enumerate}
\item[(1)] An {\em embedding} from $(X, P)$ to $(Y, Q)$ is a pair of maps $(i, j)$ where $i\colon X\to Y$ is an injective map and $j\colon P\to Q$ is a bijection such that for any $p\in P$ and $x\in \dom(p)$, $j(p)(i(x))=i(p(x))$.
\item[(2)] We say that $(Y, Q)$ is an {\em extension} of $(X, P)$ if there is an embedding from $(X, P)$ to $(Y, Q)$.
\item[(3)] We say that $(X, P)$ is a {\em subproblem} of $(Y, Q)$ if $X\subseteq Y$ and there is an embedding $(i,j)$ from $(X, P)$ to $(Y, Q)$ where $i\colon X\to Y$ is the inclusion map.
\end{enumerate}
\end{definition}

The more restrictive notion of embedding by requiring the map $j$ to be bijective will facilitate our presentations of proofs in the sequel. Whenever there is a bijection $j$ between $P$ and $Q$, we use the usual convention to extend the definition of $j$ to $P^{-1}$, i.e., for $p\in P$ we let $j(p^{-1})=j(p)^{-1}$. This can be further extended to all elements of $F(P)$, i.e., for $w=p_1\cdots p_k$ where $p_1, \dots, p_k\in P\cup P^{-1}$, let 
$$ j(w)=j(p_1)\cdots j(p_k). $$ In fact, this extended $j$ is an isomorphism between the free groups $F(P)$ and $F(Q)$. Furthermore, if $\Gamma$ is a group with a set $S$ of generators and there is a bijection $\theta\colon P\to S$, then for any $w\in F(P)$ we can define $w_{\theta, \Gamma}$ (or $w_\Gamma$ when the bijection $\theta$ is understood) as follows: if $w=p_1\cdots p_k$ where $p_1,\dots, p_k\in P\cup P^{-1}$, then
$$ w_\Gamma=\theta(p_1)\cdots \theta(p_k)\in \Gamma. $$

\begin{definition}[Solution]\label{def:solution} Let $(X, P)$ and $(Y, Q)$ be EP problems.
\begin{enumerate}
\item[(1)] We say that $(X, P)$ is an {\em EP solution} (or simply {\em solution}) if $P$ is a set of full permutations of $X$, i.e., $P\subseteq \mbox{\rm Sym}(X)$. 
\item[(2)] If $(X, P)$ is a solution, then the {\em solvent} of $(X, P)$ is the subgroup $\langle P\rangle$ of $\mbox{\rm Sym}(X)$ generated by $P$.
\item[(3)] We say that $(X, P)$ is a {\em solution for} $(Y, Q)$ if $(X, P)$ is a solution and an extension of $(Y, Q)$.
\end{enumerate}
\end{definition}

If $(X, P)$ is a solution with solvent $\Gamma$ and $w=p_1\dots p_k\in F(P)$ where $p_1, \dots, p_k\in P\cup P^{-1}$, then $w_\Gamma$ is given by 
$$ p_1\circ \cdots \circ p_k,$$
which is an element of $\Gamma$. Furthermore, if $(X, P)$ is a soluton for $(Y, Q)$ with solvent $\Gamma$ and $w\in F(Q)$, we also write $w_\Gamma$ for the element $j(w)_\Gamma$, where $j$ is the bijection from $Q$ to $P$.

\begin{definition}[$\mathsf{H}$-solution] Let $(X, P)$ be an EP problem. Let $\mathsf{H}$ be a pseudovariety of  finite groups.
We say that $(X, P)$ is an {\em $\mathsf{H}$-solution} if $(X, P)$ is a solution and the solvent of $(X, P)$ is in $\mathsf{H}$.
\end{definition}

In this paper we tacitly assume that $\mathsf{H}$ is one of the following pseudovarieties: 
finite $p$-groups ($p$ a prime number), finite $\Pi$-groups ($\Pi$ a set of prime numbers), finite odd-groups, finite nilpotent-groups.

With the terminology defined in this subsection, the odd-extension property is the statement that every quo-odd EP problem has an odd-solution.

\begin{definition}[EP map]\label{def:EPmap} An {\em EP map} $\Phi$ is a function $\Phi\colon \mathcal{C}\times S\to \mathcal{D}$, where $\mathcal{C}$ and $\mathcal{D}$ are collections of EP problems, and $S$ is a set of parameters.
\end{definition}

The concept of EP map is defined to capture several constructions of EP problems from other given EP problems. For us the most important collection of EP problems consists of quo-odd ones, but we will also consider quo-$p$ EP problems and nilpontent-solutions.

\subsection{The EP map {\tt Multiverse}}

Throughout this subsection let $\Pi$ be a set of prime numbers. We define an EP map based on the construction given in proof of Theorem~\ref{thm:EP2}.

\begin{definition}\label{def:multiverse} Given any EP problem $(X, P)$ where $P=\{p_1,\dots, p_k\}$, define ${\tt Multiverse}(X, P)$ (or simply $\mbox{\tt Multiverse}(P)$ if there is no danger of confusion) to be the EP problem $(Y, Q)$, where $Y=X\times\{1,\dots, k\}$ and $Q=\{a, b\}$
with
$$ \begin{array}{rcl} a(x,i)&=& (p_i(x),i) \mbox{ for $1\leq i\leq k$ and $x\in\dom(p_i)$;} \\ 
b(x,i)&=& (x, i+1), \mbox{ for $1\leq i<k$ and $x\in X$.}
\end{array}
$$
\end{definition}

We also consider an EP map with parameters as follows.

\begin{definition}\label{def:multiverse2} Given any EP problem $(X, P)$ where $P=\{p_1,\dots, p_k\}$ and a $\Pi$-number $m\geq k$, define 
$${\tt Multiverse}(X, P; m)$$ (or simply $\mbox{\tt Multiverse}(P; m)$ if there is no danger of confusion) to be the EP problem $(Y, Q)$, where $Y=X\times\{1,\dots, m\}$ and $Q=\{a, b\}$
with
$$ a(x,i)= \left\{\begin{array}{ll} (p_i(x),i), &\mbox{ for $1\leq i\leq k$ and $x\in\dom(p_i)$, } \\
(x, i), & \mbox{ if $k<i\leq m$;} \end{array}\right. 
$$
and 
$$ 
b(x,i)=\left\{\begin{array}{ll} (x, i+1), & \mbox{ if $1\leq i<m$,} \\ (x, 1), & \mbox{ if $i=m$.}\end{array}\right.
$$
\end{definition}

\begin{theorem}\label{thm:multiverse2} Let $(X, P)$ be an EP problem and let $m$ be a sufficiently large $\Pi$-number. Let
$$ (Y, Q)={\tt Multiverse}(X, P; m). $$
Then the following hold:
\begin{enumerate}
\item[\rm (i)] If $(X, P)$ is quo-$\Pi$, then so is $(Y, Q)$;
\item[\rm (ii)] If $(X, P)$ is a solution, then so is $(Y, Q)$. In particular, if $(X, P)$ is a $\Pi$-solution, then so is $(Y, Q)$.
\end{enumerate}
\end{theorem}

\begin{proof} The proof of (i) is similar to the proof of Theorem~\ref{thm:EP2}. In fact, suppose $P=\{p_1, \dots, p_k\}$ is quo-$\Pi$ and let $m\geq k$ be a $\Pi$-number. We verify that $Q=\{a, b\}$ is quo-$\Pi$. To see this, suppose $q\not\in \Pi$, $y\in Y$ and $w\in F(a,b)$ such that $w^q(y)=y$. Suppose $y=(x,i)$ and 
$$ w=b^{\beta_n}a^{\alpha_n}\cdots b^{\beta_1}a^{\alpha_1}b^{\beta_0} $$
for $\alpha_1,\dots, \alpha_n, \beta_0, \beta_1,\dots, \beta_n\in \mathbb{Z}$. We have that $q(\beta_0+\beta_1+\cdots+\beta_n)\equiv 0\pmod m$. Hence $\beta_0+\beta_1+\cdots+\beta_n\equiv 0\pmod{m}$.  Let 
$$\begin{array}{l} j_0=i-1+\beta_0\pmod{m},\\  j_1=j_0+\beta_1\pmod{m}, \\
\dots\dots\\  j_{n-1}=j_{n-2}+\beta_{n-1}\pmod{m}. 
\end{array}$$ 
Let 
$$ A=\big\{a\in\{0,\dots,n-1\}\colon j_a<k\big\}. $$
Let $a_0, \dots, a_{\ell}$ enumerate all elements of $A$ in increasing order. Then we have that $x\in \dom(p_{j_{a_0}})$ and
$$ (p_{1+j_{a_\ell}}^{\alpha_{a_\ell}}\cdots p_{1+j_{a_1}}^{\alpha_{a_1}}p_{1+j_{a_0}}^{\alpha_{a_0}})^q(x)=x. $$
Since $P$ is quo-$\Pi$, we have
$$ p_{1+j_{a_\ell}}^{\alpha_{a_\ell}}\cdots p_{1+j_{a_1}}^{\alpha_{a_1}}p_{1+j_{a_0}}^{\alpha_{a_0}}(x)=x. $$
It follows that $w(y)=y$.

(ii) follows from the definition and Lemma~\ref{lem:fullPi}. 
\end{proof}

\begin{theorem}\label{thm:multiverse} Let $(X, P)$ be an EP problem and let 
$$(Y, Q)={\tt Multiverse}(X, P).$$ Then the following hold.
\begin{enumerate}
\item[\rm (i)] If $(X, P)$ is quo-$\Pi$, then so is $(Y, Q)$;
\item[\rm (ii)] If $(Y, Q)$ has a solution with solvent $\Lambda$, then $(X, P)$ has a solution with solvent $\Gamma$, where $\Gamma$ is isomorphic to a subgroup of $\Lambda$;
\item[\rm (iii)] $(X, P)$ has a $\Pi$-solution if and only if $(Y, Q)$ has a $\Pi$-solution.
\end{enumerate}
\end{theorem}

\begin{proof} (i) and (ii) were shown in the proof of Theorem~\ref{thm:EP2}. We only show (iii). If $(Y, Q)$ has a $\Pi$-solution with solvent $\Lambda$, which is a $\Pi$-group, then by (ii), $(X, Q)$ has a solution with solvent $\Gamma$, where $\Gamma$ is isomorphic to a subgroup of $\Lambda$. Hence $\Gamma$ is also a $\Pi$-group. 

Conversely, suppose $(X, P)$ has a $\Pi$-solution $(X', P')$ with solvent $\Gamma=\langle P'\rangle\leq \mbox{\rm Sym}(X')$. Let $P'=\{p_1,\dots, p_k\}$ and $Q=\{a, b\}$. Let $m\geq k$ be a $\Pi$-number. Let 
$$ (Y',Q')={\tt Multiverse}(X', P'; m) $$
where $Y'=X'\times\{1,\dots, m\}$ and $Q'=\{a',b'\}$. Then $(Y', Q')$ is a $\Pi$-solution by Theorem~\ref{thm:multiverse2}. It is clear from the definitions that $(Y, Q)$ is a subproblem of $(Y', Q')$. Hence $(Y', Q')$ is a $\Pi$-solution for $(Y, Q)$.
\end{proof}

\subsection{The EP map {\tt Reroute}} Throughout this subsection let $\Pi$ be a set of prime numbers. We define an EP map with parameters.

\begin{definition}\label{def:covering} Let $(X, P)$ and $(Y, Q)$ be EP problems. We say that $(X, P)$ is a {\em covering} of $(Y, Q)$ if there are functions $f_E\colon P\cup P^{-1}\to Q\cup Q^{-1}$ and $f_V\colon X\to Y$ such that 
\begin{enumerate}
\item[(i)] for each $p\in P$, $f_E(p)\in Q$ and $f_E(p^{-1})=f_E(p)^{-1}$;
\item[(ii)] for each $p\in P\cup P^{-1}$ and $x\in \dom(p)$, $f_V(x)\in \dom(f_E(p))$ and 
$$f_E(p)(f_V(x))=f_V(p(x));$$
\item[(iii)] for each $x\in X$, letting the {\em star} of $x$ be 
$$\mbox{\rm St}(x)=\{p\in P\cup P^{-1}\colon x\in\dom(p)\},$$ then $f_E$ is bijective between $\mbox{\rm St}(x)$ and $\mbox{\rm St}(f_V(x))$.
\end{enumerate}
\end{definition}

\begin{lemma}\label{lem:cov} Let $(X, P)$ and $(Y, Q)$ be EP problems. Suppose $(X, P)$ is a covering of $(Y,Q)$. If $(Y,Q)$ is a solution, then so is $(X, P)$.
\end{lemma}

\begin{proof} Since $(X,P)$ is a covering of $(Y, Q)$, we have that for any $x\in X$ and $p\in P$, $p\in \mbox{\rm St}(x)$ if and only if $f_E(p)\in \mbox{\rm St}(f_V(x))$. Since $(Y, Q)$ is a solution, we always have $f_V(x)\in \dom(f_E(p))$. Thus $x\in \dom(p)$. This shows that $(X, P)$ is a solution.
\end{proof}

\begin{definition} \label{def:reroute}
Let $(X,P)$ be an EP problem. Let $k$ be a $\Pi$-number. Let $p_0\in P$ and $x_0\in X$ be such that $x_0\in\dom(p_0)$ and $p_0(x_0)\neq x_0$. Define an EP problem $\mbox{\tt Reroute}(X, P; k, p_0, x_0)=(Y, Q)$, where $Y=X\times\{0,\dots, k-1\}$ and $Q=\{\hat{p}\colon p\in P\}$ with
$$ \hat{p}(x, i)=\left\{\begin{array}{ll} (p(x), i), & \mbox{ if $p\neq p_0$ and $x\in \dom(p)$,} \\
& \mbox{ or if $p=p_0$ and $x\in \dom(p)\setminus\{x_0\}$, } \\
\big(p(x), i+1\!\!\!\!\mod{k}\big), & \mbox{ if $p=p_0$ and $x=x_0$.}
\end{array}\right.
$$
For $p_0\in P^{-1}$ and $x_0\in X$ such that $x_0\in \dom(p_0)$ and $p_0(x_0)=x_0$, similarly define $\mbox{\tt Reroute}(X, P; k, p_0, x_0)$.
\end{definition}

\begin{lemma}\label{lem:inverse} Let $(X, P)$ be an EP problem, let $k$ be a $\Pi$-number, and let $p_0\in P$, $x_0\in X$ be such that $x_0\in \dom(p_0)$ and $p_0(x_0)\neq x_0$. Then the EP problems 
$$\mbox{\tt Reroute}(X, P; k, p_0, x_0) \mbox{ and } \mbox{\tt Reroute}(X, P; k, p_0^{-1}, p_0(x_0))$$ are isomorphic.
\end{lemma}

\begin{proof} An isomorphism between the two problems is given by $(x,i)\mapsto (x, -i \mbox{ mod } k)$.
\end{proof}

\begin{theorem}\label{thm:reroute} Let $(X, P)$ be an EP problem, let $k$ be a $\Pi$-number, and let 
$$(Y, Q)=\mbox{\tt Reroute}(X, P; k, p_0, x_0)$$ with parameters $k, p_0, x_0$. Then the following hold.
\begin{enumerate}
\item[\rm (i)] If $(X, P)$ is quo-$\Pi$, then so is $(Y, Q)$;
\item[\rm (ii)] $(Y, Q)$ is a covering of $(X, P)$. Consequently, if $(X, P)$ is a solution, then so is $(Y, Q)$. In particular, if $(X,P)$ is a $\Pi$-solution, then so is $(Y, Q)$.
\item[\rm (iii)] If $(X, P)$ has a $\Pi$-solution, then so does $(Y,Q)$.
\end{enumerate}
\end{theorem}

\begin{proof} Note that $p\mapsto \hat{p}$ is a bijection from $P$ to $Q$, which induces a bijection $w\mapsto \hat{w}$ from $F(P)$ to $F(Q)$.  

(i) Suppose $(X, P)$ is quo-$\Pi$. We verify that $(Y, Q)$ is quo-$\Pi$. For this, let $\hat{w}\in F(Q)$, $y=(x,i)\in Y$ and $q\not\in \Pi$ be such that $\hat{w}^q(y)=y$. If $\hat{w}$ does not involve $\hat{p}_0$ then there is $w\in F(P)$ such that $\hat{w}^q(x,i)=(w^q(x),i)=(x,i)$, and so $w(x)=x$ and $\hat{w}(y)=y$. 

Assume 
$$ \hat{w}=\hat{u}_m \hat{p}_0 \hat{u}_{m-1} \cdots\hat{p}_0\hat{u}_0 $$
for some $\hat{u}_0,\dots, \hat{u}_m\in F(Q)$ which do not involve $\hat{p}_0$. Let
$$ w=u_mp_0u_{m-1}\cdots p_0u_0 $$
be the corresponding element of $F(P)$. Let $\pi\colon Y\to X$ be the projection map. Then we have $\pi(\hat{w}(y))=w(x)$ and $x=\pi(y)=\pi(\hat{w}^q(y))=w^q(x)$. Since $P$ is quo-$\Pi$, we have $w(x)=x$. 

Now let
$$y_0=y,\ y_1=\hat{u}_0(y_0),\ \dots,\ y_m=\hat{u}_{m-1}\hat{p}_0(y_{m-1}).$$
Let $\alpha_1, \dots, \alpha_\ell$ list all indices $\alpha\in\{1, \dots, m\}$ such that $\pi(y_\alpha)=x_0$. Since $\hat{w}^q(y)=y$, we have that $q\ell\equiv 0\pmod{k}$. Since $\mbox{gcd}(q,k)=1$, we have $\ell\equiv 0\pmod{k}$. This implies that $\hat{w}(y)=y$. Thus $(Y, Q)$ is quo-$\Pi$.

(ii) It is easy to see that $(Y,Q)$ is a covering of $(X, P)$. Thus by Lemma~\ref{lem:cov}, if $(X, P)$ is a solution, then $(Y, Q)$ is a solution. The last conclusion follows from (i) and Lemma~\ref{lem:fullPi}.

(iii) Suppose $(X, P)$ has a $\Pi$-solution $(X', P')$, with $p\mapsto p'$ being the bijection from $P$ to $P'$. Let $(Y', Q')=\mbox{\tt Reroute}(X', P'; k, p_0', x_0)$. Then it is easy to see that $(Y, Q)$ is a subproblem of $(Y', Q')$. By (ii), $(Y', Q')$ is a $\Pi$-solution. Thus $(Y', Q')$ is a $\Pi$-solution for $(Y, Q)$.
\end{proof}

\subsection{Fundamental groups of EP problems\label{subsec:fundamentalgroup}} Throughout this subsection let $\Pi$ be a set of prime numbers.

Any EP problem $(X, P)$ gives rise naturally to a labelled multidigraph on $X$, which we denote by $G(X, P)$. Let $G^*(X, P)$ be the corresponding unlabelled, undirected multigraph. Recall that $G(X, P)$ is weakly connected if and only if $G^*(X, P)$ is connected. 

Decompose $G(X, P)$ into weakly connected components $Y_1, \dots, Y_k$. For each $1\leq i\leq k$, we obtain a subproblem $(Y_i, P_i)$, where $P_i=P\cap (Y_i\times Y_i)$. If $(Y'_i, P'_i)$ is a $\Pi$-solution for $(Y_i, P_i)$, then the disjoint union of $(Y'_i, P'_i)$ for $1\leq i\leq k$ is a $\Pi$-solution for $(X, P)$. Thus, we have that $(X, P)$ has a $\Pi$-solution if and only if all of $(Y_i, P_i)$ have $\Pi$-solutions. In view of this observation, we may consider, without loss of generality, only EP problems $(X, P)$ where $G(X, P)$ is weakly connected. 

\begin{definition} Let $(X, P)$ be an EP problem where $G(X, P)$ is weakly connected. Let $x_0\in X$. The {\em fundamental group} of $(X, P)$ with base point $x_0$ is defined as
$$ \pi_1(X, P; x_0)=\{w\in F(P)\colon w(x_0)=x_0\}. $$
We write $\pi_1(X, P)$ or even $\pi_1(P)$ when there is no danger of confusion.
\end{definition}

Since $X$ is finite, $\pi_1(P)$ is finitely generated. 

Recall that the notions of walk, circuit, loop, etc. for multidigraphs have been defined in Subsection~\ref{subsec:multidigraph}.  These notions can be defined for labelled multidigraphs verbatim. Here we need to define more terminology. Let $G$ be a labelled multidigraph, and let $G^*$ be the corresponding unlabelled, undirected multigraph. We define a walk $w=e_0e_1\dots e_k$ in $G$ to be a {\em path} if the corresponding walk in $G^*$ is a path, i.e., in the sequence $\mathsf{s}(e_0), \mathsf{s}(e_1), \dots, \mathsf{s}(e_k), \mathsf{t}(e_k)$, no vertex occurs more than once. An edge $e$ has two {\em endpoints} $\mathsf{s}(e)$ and $\mathsf{t}(e)$; a path $w=e_0e_1\dots e_k$ also has two {\em endpoints} $\mathsf{s}(e_0)$ and $\mathsf{t}(e_k)$. A walk $w=e_0e_1\dots e_k$ in $G$ is a {\em cycle} if either $k=0$ and $w$ is a loop, i.e., $w=e_0$ with $\mathsf{s}(e_0)=\mathsf{t}(e_0)$, or $k>0$ and the corresponding walk in $G^*$ is a cycle, i.e., $\mathsf{s}(e_0)=\mathsf{t}(e_k)$ and $e_1\dots e_k$ is a path. 

\begin{lemma}\label{lem:singleedgeextension} Let $(X, P)$ and $(Y, Q)$ be EP problems satisfying
\begin{enumerate}
\item[\rm (a)] $(X, P)$ is a subproblem of $(Y, Q)$; 
\item[\rm (b)] both $G(X, P)$ and $G(Y, Q)$ are weakly connected;
\item[\rm (c)]  $G(Y, Q)-G(X, P)$ consists of exactly one edge, one of whose endpoints is not in $X$.
\end{enumerate}
Then $(X, P)$ has a $\Pi$-solution if and only if $(Y, Q)$ has a $\Pi$-solution.
\end{lemma}

\begin{proof} Since $(X, P)$ is a subproblem of $(Y, Q)$, it follows immediately that if $(Y, Q)$ has a $\Pi$-solution $(Y', Q')$, then $(Y', Q')$ is also a $\Pi$-solution for $(X, P)$. 

Let $p\mapsto \tilde{p}$ be the bijection from $P$ to $Q$ such that $\tilde{p}$ extends $p$ for each $p\in P$. Suppose the only edge $e$ in $G(Y, Q)-G(X, P)$ is labeled by $\tilde{p}_0\in Q$. Since $G(Y, Q)$ is weakly connected, one of the endpoints of $e$ is in $X$, which we denote as $x_0\in X$. By assumption (c), the other endpoint of $e$ is not in $X$. 

First suppose $x_0\in \mbox{\rm dom}(\tilde{p}_0)$. Let $y_0=\tilde{p}_0(x_0)\in Y\setminus X$. Then $x_0\not\in \dom(p_0)$.  Since $G(Y, Q)$ is weakly connected, we must have $Y=X\cup\{y_0\}$. Moreover, we have that $\tilde{p}_0=p_0\cup\{(x_0, y_0)\}$ and for any $p\in P\setminus\{p_0\}$, $\tilde{p}=p$.

Now suppose $(X, P)$ has a $\Pi$-solution $(X', P')$. Let $p\mapsto p'$ be the bijection from $P$ to $P'$ such that $p'$ extends $p$ for each $p\in P$. Let $k$ be a $\Pi$-number. Now consider the EP problem 
$$ (Y', Q')=\mbox{\tt Reroute}(X', P'; k, p'_0, x_0). $$
By Theorem~\ref{thm:reroute} (ii), $(Y', Q')$ is a $\Pi$-solution. To see that $(Y', Q')$ is a $\Pi$-solution for $(Y, Q)$, we only need to verify that $(Y, Q)$ is a subproblem of $(Y', Q')$. To see this, consider the map $f_V\colon Y\to Y'$ defined by
$$ f_V(y)=\left\{\begin{array}{ll} (y, 0), & \mbox{ if $y\in X$,} \\ (p'_0(x_0), 1), & \mbox{ if $y=y_0$.}\end{array}\right. $$
Then $f_V$ is injective and the map $\tilde{p}\mapsto \widehat{p'}$ (in the notation of Definition~\ref{def:reroute}) witnesses that $(Y, Q)$ is a subproblem of $(Y', Q')$.

The case $x_0\in \dom((\tilde{p}_0)^{-1})$ now follows from Lemma~\ref{lem:inverse} and a similar argument.
\end{proof}

\begin{lemma}\label{lem:singlepathextension} Let $(X, P)$ and $(Y, Q)$ be EP problems satisfying
\begin{enumerate}
\item[\rm (a)] $(X, P)$ is a subproblem of $(Y, Q)$; 
\item[\rm (b)] both $G(X, P)$ and $G(Y, Q)$ are weakly connected;
\item[\rm (c)] $G(Y, Q)-G(X, P)$ consists of a single path, all of whose vertices except for one endpoint are not in $X$.
\end{enumerate}
Then $(X, P)$ has a $\Pi$-solution if and only if $(Y, Q)$ has a $\Pi$-solution.
\end{lemma}

\begin{proof} Define a finite sequence of EP problems $(X_0, P_0), (X_1, P_1), \dots, (X_\ell, P_\ell)$ so that the following hold:
\begin{enumerate}
\item[\rm (i)] $(X_0, P_0)=(X, P)$ and $(X_\ell, P_\ell)=(Y, Q)$;
\item[\rm (ii)] for all $0\leq i\leq \ell$, $G(X_i, P_i)$ is weakly connected;
\item[\rm (iii)] for each $0\leq i<\ell$, $(X_i, P_i)$ is a subproblem of $(X_{i+1}, P_{i+1})$ so that $G(X_{i+1}, P_{i+1})-G(X_i, P_i)$ consists of a single edge, one of whose endpoints is not in $X_i$.
\end{enumerate}
 Here $\ell$ is the length of the path that constitutes $G(Y, Q)-G(X, Q)$. Now applying Lemma~\ref{lem:singleedgeextension} to each pair $(X_i, P_i)$ and $(X_{i+1}, P_{i+1})$ for $0\leq i<\ell$, we conclude that for each $0\leq i<\ell$, $(X_i, P_i)$ has a $\Pi$-solution if and only if $(X_{i+1}, P_{i+1})$ has a $\Pi$-solution. By transitivity, the conclusion of the lemma holds.
\end{proof}

\begin{corollary}\label{cor:trivial} Let $(X, P)$ be an EP problem where $G(X, P)$ is weakly connected. Suppose $\pi_1(X, P)$ is trivial. Then $(X, P)$ has a $\Pi$-solution.
\end{corollary}

\begin{proof} Since $G(X, P)$ is weakly connected and $\pi_1(X, P)$ is trivial, the undirected multigraph corresponding to $G(X, P)$ is a tree (and in fact a graph). The corollary is now obtained from inductively applying Lemma~\ref{lem:singlepathextension}, where in the base case just noting that the singleton problem (with no edge) has a $\Pi$-solution.
\end{proof}

\begin{theorem}\label{thm:solutionprotopology} Let $(X, P)$ be an EP problem where $G(X, P)$ is weakly connected. Then the following hold.
\begin{enumerate}
\item[\rm (i)] $(X, P)$ is quo-$\Pi$ if and only if $\pi_1(X, P)$ is a quo-$\Pi$ subgroup of $F(P)$.
\item[\rm (ii)] $(X, P)$ has a $\Pi$-solution if and only if $\pi_1(X, P)$ is closed in the pro-$\Pi$ topology of $F(P)$.
\end{enumerate}
\end{theorem}

\begin{proof} (i) It is immediate that if $(X, P)$ is quo-$\Pi$ then $\pi_1(X, P)$ is a quo-$\Pi$ subgroup of $F(P)$. Conversely, suppose $\pi_1(X, P)$ is a quo-$\Pi$ subgroup of $F(P)$. Suppose $w\in F(P)$, $x\in X$ and $q\not\in \Pi$ such that $w^q(x)=x$. Let $u\in F(P)$ be such that $u(x_0)=x$. Such $u$ exists since $G(X, P)$ is weakly connected. Then $u^{-1}w^qu(x_0)=x_0$. This means that $(u^{-1}wu)^q\in \pi_1(X, P)$. Since $\pi_1(X, P)$ is a quo-$\Pi$ subgroup of $F(P)$, we have $u^{-1}wu\in \pi_1(X, P)$. Thus $w(x)=x$. 

(ii) First suppose $S=\pi_1(X, P)$ is closed in the pro-$\Pi$ topology of $F(P)$. Then the second part of the proof of Theorem~\ref{thm:LERF2EP} gives an solution of the form $X'=F(P)/NS$ for some quo-$\Pi$ normal subgroup $N$ of $F(P)$, where the extension $P'$ of $P$ comes from the action of $F(P)/N$ on $X'$ by left multiplication. This gives a solvent isomorphic to $F(P)/N$, which is a $\Pi$-group.

Conversely, suppose $(X, P)$ has a $\Pi$-solution. Assume $w_0\not\in \pi_1(X, P)$.  We define an extension $(Y, Q)$ of $(X, P)$ so that they satisfy the assumptions of Lemma~\ref{lem:singlepathextension} and, letting $p\mapsto \hat{p}$ be the bijection from $P$ to $Q$, we have $\hat{w}_0(x_0)$ is defined and $\hat{w}(x_0)\neq x_0$. If $w_0(x_0)$ is defined in $(X, P)$ and $w_0(x_0)\neq x_0$, then we just let $(Y,Q)=(X, P)$ and $p\mapsto \hat{p}$ be the identity. Otherwise, there are $u_0\in F(P)$ and $p_0,p_1,\dots, p_m\in P\cup P^{-1}$ such that 
$$w_0=p_mp_{m-1}\dots p_1p_0u_0, $$
$u_0(x_0)$ is defined in $(X, P)$, and $u_0(x_0)\not\in\mbox{\rm dom}(p_0)$. Define $y_0=x_0$, $Y=X\cup \{y_1,\dots, y_{m+1}\}$ and $Q=\{\hat{p}\colon p\in P\}$, where for each $p\in P$,  
$$ \hat{p}=p\cup\left\{(y_i, y_{i+1})\colon p_i=p\right\}\cup \left\{(y_{i+1}, y_i)\colon p_i=p^{-1}\right\}. $$
Then $(Y, Q)$ is an extension of $(X, P)$, and $\hat{w}_0(x_0)=y_{m+1}\neq x_0$ in $(Y, Q)$. 

By Lemma~\ref{lem:singlepathextension}, $(Y, Q)$ has a $\Pi$-solution $(Y', Q')$. Let $\hat{p}\mapsto \hat{p}'$ be the bijection from $Q$ to $Q'$ witnessing that $(Y, Q)$ is a subproblem of $(Y', Q')$.  Let $\Gamma=\langle Q'\rangle\leq \mbox{\rm Sym}(Y')$ be the solvent. $\Gamma$ is a $\Pi$-group. Consider $\varphi\colon F(P)\to \Gamma$ defined by
$$ \varphi(w)=\varphi(q_1\dots q_n)=\hat{q}'_1\dots \hat{q}'_n=w_\Gamma $$
where $q_1, \dots, q_n\in P\cup P^{-1}$. Then $\varphi$ is a group homomorphism. Let $N=\mbox{\rm ker}\,\varphi$. Since $F(Q)/N\cong \Gamma$, $N$ is a quo-$\Pi$ normal subgroup of $F(Q)$. Thus $N$ is open in the pro-$\Pi$ topology of $F(Q)$.

Now let $S'=\{w\in F(P)\colon \hat{w}'(x_0)=x_0\}$. Then $N\pi_1(X,P)\leq S'$. Since $\hat{w}'_0(x_0)\neq x_0$, we have $w_0\not\in S'$, and hence $w_0\not\in N\pi_1(X, P)$ and $w_0N\cap \pi_1(X, P)=\varnothing$. This shows that $\pi_1(X, P)$ is closed in the pro-$\Pi$ topology of $F(P)$.
\end{proof}

\begin{corollary}\label{cor:fundgp} Let $(X, P)$ and $(Y, Q)$ be EP problems where $G(X, P)$ and $G(Y, Q)$ are weakly connected. Suppose there is a bijection $j$ between $P$ and $Q$. If $j(\pi_1(X, P))=\pi_1(Y, Q)$, then $(X, P)$ has a $\Pi$-solution if and only if $(Y, Q)$ has a $\Pi$-solution.
\end{corollary}

\begin{proof} This follows immediately from Theorem~\ref{thm:solutionprotopology} (ii).
\end{proof}

Theorem~\ref{thm:solutionprotopology} can be viewed as a strengthening of Theorems~\ref{thm:EPtoLERF} and \ref{thm:LERF2EP}.
Corollary~\ref{cor:trivial} and Theorem~\ref{thm:solutionprotopology} (ii) also give a direct proof that a finitely generated free group $F$ is {\em residually $\Pi$}, i.e., for any $w\neq 1_F$, there exists a quo-$\Pi$ normal subgroup $N$ of $F$ such that $w\not\in N$, or equivalently, there exists a finite $\Pi$-group $\Gamma$ and an epimorphism $\varphi\colon F\to \Gamma$ such that $\varphi(w)\neq 1_{\Gamma}$.

\subsection{Cyclic fundamental groups} Throughout this subsection let $\Pi$ be a set of prime numbers.

In this subsection we consider the EP problems whose fundamental groups are cyclic, i.e., generated by one element. We show that the $\Pi$-EP holds in this case. This case was essentially treated in \cite{HPSW19}.
Here we give an alternative proof of their main result \cite[Theorem 1.5]{HPSW19}.

\begin{lemma} [{\cite[Lemma 9.1]{HPSW19}}]\label{lem:HPSW} Let $F$ be a finitely generated free group and let $C$ be a maximal cyclic subgroup of $F$. Then $C$ is closed in the pro-$\Pi$ topology of $F$. 
\end{lemma}

\begin{proof} Let $C=\langle u\rangle$. Then the maximality of $C$ means that $u$ cannot be written as $u_0^k$ for any $u_0$ and $k>1$. Suppose $w\not\in C$. Then $wuw^{-1}u^{-1}\neq 1_F$. Since $F$ is residually $\Pi$, there exists a quo-$\Pi$ normal subgroup $N$ of $F$ such that $wuw^{-1}u^{-1}\not\in N$. It follows that $wN\not\subseteq NC$, since otherwise $wN=u^\ell N$ for some $\ell$, and $wuw^{-1}u^{-1}N$ would be $N$. Thus $w\not\in NC$ and $wN\cap C=\varnothing$. 
\end{proof}

\begin{lemma}\label{lem:pfk} Let $(X, P)$ be an EP problem. Suppose $G(X, P)$ is a cycle and $\pi_1(X, P)=\langle w\rangle$. Let $w=w_0^k$  where $k\geq 1$ is the largest such. Then $P$ is quo-$\Pi$ if and only if $k$ is a $\Pi$-number.
\end{lemma}

\begin{proof} First suppose $P$ is quo-$\Pi$. Let $x_0\in X$ be the base point for the computation of $\pi_1(X,P)$. Then $w(x_0)=x_0$ and the length of the cycle $G(X,P)$ is $|w|$. Let $p$ be a prime factor of $k$ and write $k=pr$. Then $w(x_0)=w_0^{pr}(x_0)=x_0$. If $p\not\in \Pi$ then we have $w_0^r(x_0)=x_0$ since $P$ is quo-$\Pi$, but this means that $G(X,P)$ is a cycle of length $r|w_0|<|w|$, a contradiction.

Conversely, suppose every prime factor of $k$ is in $\Pi$. By Theorem~\ref{thm:solutionprotopology} (i), we only need to show that $\pi_1(X, P)$ is a quo-$\Pi$ subgroup of $F(P)$. For this, assume $q\not\in\Pi$ and $u^q\in \pi_1(X, P)$. Then there is $\ell$ such that $u^q=w^\ell=w_0^{k\ell}$. By the assumption on the maximality of $k$, we must have $u=w_0^r$ for some $r$. This implies $qr=k\ell$. Since $q\not\in \Pi$, $q$ is not a factor of $k$. Thus $q$ is a factor of $\ell$. Write $\ell=qt$. Then $r=kt$
and $u=w_0^r=(w_0^k)^t=w^t\in \pi_1(X, P)$ as required.
\end{proof}

\begin{theorem}\label{thm:cyclic} Let $(X, P)$ be an EP problem where $G(X, P)$ is weakly connected. Suppose $\pi_1(X, P)$ is cyclic. If $(X, P)$ is quo-$\Pi$ then $(X, P)$ has a $\Pi$-solution.
\end{theorem}

\begin{proof} By Corollary~\ref{cor:fundgp} we may assume $G(X, P)$ is a cycle and $\pi(X, P)=\langle w\rangle$. Let $w=w_0^k$ where $k\geq 1$ is the largest such. 

If $k=1$ then by Lemma~\ref{lem:HPSW} and Theorem~\ref{thm:solutionprotopology} (ii) we get that $(X,P)$ has a $\Pi$-solution.

Assume $k>1$ and suppose $(X, P)$ is quo-$\Pi$. Then by Lemma~\ref{lem:pfk} every prime factor of $k$ is in $\Pi$. We define an EP problem $(Y, Q)$ so that $G(Y, Q)$ is a cycle of length $m=|w_0|$. Write $w_0=p_{m-1}\cdots p_0$ with $p_0,\dots, p_{m-1}\in P\cup P^{-1}$. Let $p\mapsto \hat{p}$ be a bijection from $P$ to $Q$. Let $Y=\{y_0, y_1, \dots, y_{m-1}\}$, let $y_m=y_0$, and let $Q=\{\hat{p}\colon p\in P\}$ where, for each $p\in P$, 
$$ \hat{p}=\{(y_i, y_{i+1})\colon p_i=p\}\cup \{(y_{i+1}, y_i)\colon p_i=p^{-1}\}. $$
Then $\pi_1(Y, Q)=\langle \hat{w}_0\rangle$ with base point $y_0$. By Lemma~\ref{lem:HPSW} and Theorem~\ref{thm:solutionprotopology} (ii), $(Y, Q)$ has a $\Pi$-solution $(Y', Q')$. Let $\hat{p}\mapsto p'$ be the bijection from $Q$ to $Q'$. Let 
$$ (Z, R)=\mbox{\tt Reroute}(Y', Q'; k, p'_0, y_0). $$
Then $(Z, R)$ is a $\Pi$-solution by Theorem~\ref{thm:reroute} (ii). To see that $(Z, R)$ is a $\Pi$-solution for $(X, P)$, we only need to verify that $(X, P)$ is a subproblem of $(Z, R)$. The bijection from $P$ to $R$ is $p\mapsto \widehat{p'}$ (in the notation of Definition~\ref{def:reroute}). We define an injective map $f\colon X\to Z$ as follows. For any initial segment $u$ of $w_0^k=w$, define
$$ f(u(x_0))=\widehat{u'}(y_0). $$
It is straightforward to check that $f$ and $p\mapsto \widehat{p'}$ witness that $(Z, R)$ is an extension of $(X, P)$. Thus $(Z, R)$ is a $\Pi$-solution of $(X, P)$.
\end{proof}

\begin{corollary}[{\cite[Theorem 1.5]{HPSW19}}] Let $F$ be a finitely generated subgroup and let $C=\langle w\rangle$ be a cyclic subgroup of $F$. Let $w=w_0^k$ where $k\geq 1$ is the largest such. Let $\Pi$ be a set of prime numbers containing all prime factors of $k$. Then $C$ is closed in the pro-$\Pi$ topology of $F$.
\end{corollary}

\begin{proof} This follows from Theorem~\ref{thm:cyclic}, Lemma~\ref{lem:pfk}, and Theorem~\ref{thm:solutionprotopology}.
\end{proof}

\subsection{A case study\label{subsec:case}}

In this subsection we consider the following simple EP problem $P_1$.

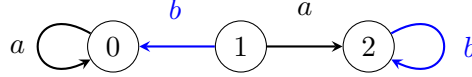
\begin{figure}[hb]
\centering
\begin{tikzpicture}
  \node [circle,draw] (A)              {0};
  \node [circle,draw] (B) [right=of A] {1};
  \node [circle,draw] (C) [right=of B] {2};

\draw [draw = black, thick, arrows={-stealth} ]
         
	(B)  edge [blue]  (A)
    (B) edge  (C)
    (C) edge [blue, in=-30, out=30,looseness=8] (C)
    (A) edge [in=210, out=150,looseness=8] (A) ;
    
    \node[right] at (-1.5,0) {$a$};
    \node[right] at (2.3,0.5) {$a$};
    \node[right] at (4.5,0) {$\textcolor{blue}{b}$};
    \node[right] at (0.6,0.5) {$\textcolor{blue}{b}$};
\end{tikzpicture}
\caption{\label{fig:P1} An odd-EP(2) problem $P_1$.}
\end{figure}

There are two partial permutations of the set $X=\{0,1,2\}$:
$$\begin{array}{rl}
a\colon & 0\mapsto 0, 1\mapsto 2 \\
b\colon & 1\mapsto 0, 2\mapsto 2
\end{array}
$$
In Figure~\ref{fig:P1} $a$ is represented by black arrows and $b$ is represented by blue arrows. It is straightforward to check that $\{a,b\}$ is quo-odd.

We first present an odd-solution for $P_1$ in Figure~\ref{fig:P1solution}. The extension set has $7$ elements. We still represent the extension of $a$, $\tilde{a}$, by black arrows, and the extension of $b$, $\tilde{b}$, by blue arrows. 

\begin{figure}[htbp]
\centering

\begin{tikzpicture}
  \node [circle,draw] (A)              {0};
   \node[circle, draw] (D) [above right= of A] {3};
   \node [circle,draw] (B) [below right=of D] {1};
  \node[circle, draw] (E) [above right = of B] {4}; 
  \node [circle,draw] (C) [below right=of E] {2};
  \node[circle, draw] (F) [above= of B] {5};
  \node[circle, draw] (G) [above= of F] {6};

\draw [draw = black, thick, arrows={-stealth} ]
         
         
	(B)  edge [blue]  (A)
    (B) edge (C)
    (C) edge (E)
    (E) edge (B)
    (D) edge (G)
    (G) edge [bend left=-20] (F)
    (F) edge (D)
    (A) edge [blue] (D)
    (D) edge [blue] (B)
    (E) edge [blue] (F)
    (G) edge [blue] (E)
    (F) edge [blue, bend left=-20] (G)
    (C) edge [blue, in=-30, out=30,looseness=8] (C)
    (A) edge [in=210, out=150,looseness=8] (A) ;
    
\end{tikzpicture}
\caption{\label{fig:P1solution} An odd-solution $S_1$ for $P_1$.}
\end{figure}
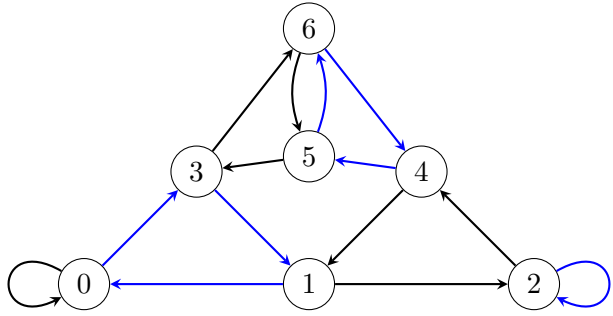

Each of $\tilde{a}$ and $\tilde{b}$ has order $3$. But the group $\Gamma$ generated by $\{\tilde{a}, \tilde{b}\}$ has order $21$. In fact, it is known to be the Frobenius group of order $21$ and it is denoted $F_{21}$ in the literature. To see this, let $u=\tilde{a}\tilde{b}^2$. Then $u$ has order $7$ and we have $\tilde{a}u\tilde{a}^{-1}=u^2$. This means that $\langle u\rangle$ is a normal subgroup of $\Gamma$ of order $7$, and then it is easy to check that $\tilde{b}=u^{-1}\tilde{a}\in\langle \tilde{a}, u\rangle$ and hence $\Gamma/\langle u\rangle$ has order $3$. In summary, $\Gamma$ is a meta-abelian group isomorphic to a semi-direct product $\mathbb{Z}/3\mathbb{Z}\ltimes \mathbb{Z}/7\mathbb{Z}$.

In the following we will unravel how the solution is obtained by our techniques. We will also show in Subsection~\ref{subsec:nilsolution} that the problem does not have a nilpotent-solution. 

First we apply the EP map {\tt Multiverse} to $P_1$ and obtain the following EP problem $P_2$ (see Figure~\ref{fig:P2}). In drawing $P_2$, we relabel the input partial permutations by $x$ (which represents both $a$ and $b$) and the new partial permutation (between the multiverses) by $y$. 

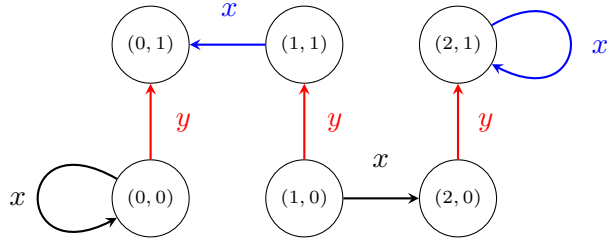
\begin{figure}[htbp]
\centering
\begin{tikzpicture}
  \node [circle,draw] (A)              {\tiny $(0,0)$};
  \node [circle,draw] (B) [right=of A] {\tiny $(1,0)$};
  \node [circle,draw] (C) [right=of B] {\tiny $(2,0)$};
   \node [circle,draw] (D) [above =of A]             {\tiny $(0,1)$};
  \node [circle,draw] (E) [above=of B] {\tiny $(1,1)$};
  \node [circle,draw] (F) [above=of C] {\tiny $(2,1)$};
 
\draw [draw = black, thick, arrows={-stealth} ]
         
	(E)  edge [blue]  (D)
    (B) edge  (C)
    (F) edge [blue, in=-30, out=30,looseness=8] (F)
    (A) edge [in=210, out=150,looseness=8] (A) 
    (A) edge [red] (D)
    (B) edge [red] (E)
    (C) edge [red] (F);
    
    \node[right] at (-2,0) {$x$};
    \node[right] at (2.8,0.5) {$x$};
    \node[right] at (5.7,2) {$\textcolor{blue}{x}$};
    \node[right] at (0.8,2.5) {$\textcolor{blue}{x}$};
    \node[right] at (0.2,1) {$\textcolor{red}{y}$};
    \node[right] at (2.2,1) {$\textcolor{red}{y}$};
    \node[right] at (4.2,1) {$\textcolor{red}{y}$};
\end{tikzpicture}
\caption{\label{fig:P2} The problem $P_2=\mbox{\tt Multiverse}(P_1)$.}
\end{figure}

By Theorem~\ref{thm:multiverse} (iii), $P_1$ has an odd-solution if and only if $P_2$ has an odd-solution. 

Next, we identify the nodes $(0,0)$ and $(2,1)$ in $P_2$ and obtain another EP problem $P_3$ as follows (see Figure~\ref{fig:P3}). In drawing the problem the partial permutation $x$ is represented by black arrows and $y$ is represented by red arrows. We label the arrows, however, by the partial permutations $x$, $y$, $x^{-1}$ or $y^{-1}$ so that the labels correspond to the applications of the partial permutations in the clockwise direction of the cycle. Note that $P_2$ is a covering of $P_3$.


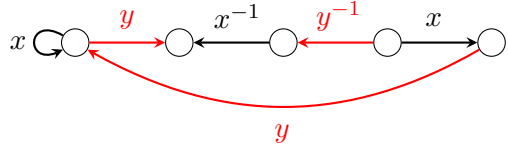
\begin{figure}[htbp]
\centering
\begin{tikzpicture}
  \node [circle,draw] (A)              {};
  \node [circle,draw] (B) [right=of A] {};
  \node [circle,draw] (C) [right=of B] {};
   \node [circle,draw] (D) [right =of C]             {};
  \node [circle,draw] (E) [right=of D] {};

\draw [draw = black, thick, arrows={-stealth} ]
         
	
    (C) edge  (B)    
    (A) edge [in=210, out=150,looseness=8] (A) 
    (A) edge [red] (B)
    (D) edge [red] (C)
    (D) edge (E)
    (E) edge [red, in=-30, out=-150] (A);
    
    \node[right] at (-1,0) {$x$};
    \node[right] at (1.7,0.35) {$x^{-1}$};
    \node[right] at (4.5,0.3) {$x$};
    \node[right] at (0.45,0.3) {$\textcolor{red}{y}$};
    \node[right] at (3.05,0.35) {$\textcolor{red}{y^{-1}}$};
    \node[right] at (2.5,-1.2) {$\textcolor{red}{y}$};
\end{tikzpicture}
\caption{\label{fig:P3} An odd-EP(2) problem $P_3$.}
\end{figure}

Now let $Q_3$ be the EP problem obtained from $P_3$ by applying the EP map {\tt Reroute}, with parameters $k=3$, $p_0=y$ and $x_0$ being the leftmost node of $P_3$. $Q_3$ is illustrated in Figure~\ref{fig:Q3}. By Theorem~\ref{thm:reroute} (iii), if $P_3$ has an odd-solution, then $Q_3$ has an odd-solution. It is clear that $P_2$ is a subproblem of $Q_3$. Thus if $Q_3$ has an odd-solution, then $P_2$ has an odd-solution.

\begin{figure}[htbp]
\centering
\begin{tikzpicture}
  \node [circle,draw] (A1)              {};
  \node [circle,draw] (B1) [right=of A1] {};
  \node [circle,draw] (C1) [right=of B1] {};
   \node [circle,draw] (D1) [right =of C1]             {};
  \node [circle,draw] (E1) [right=of D1] {};

     \node [circle,draw] (A2) [above=of A1, yshift=18pt]             {};
  \node [circle,draw] (B2) [right=of A2] {};
  \node [circle,draw] (C2) [right=of B2] {};
   \node [circle,draw] (D2) [right =of C2]             {};
  \node [circle,draw] (E2) [right=of D2] {};

         \node [circle,draw] (A3) [above=of A2, yshift=18pt]             {};
  \node [circle,draw] (B3) [right=of A3] {};
  \node [circle,draw] (C3) [right=of B3] {};
   \node [circle,draw] (D3) [right =of C3]             {};
  \node [circle,draw] (E3) [right=of D3] {};

 \draw [draw = black, thick, arrows={-stealth} ]
         
	
    (C1) edge  (B1)    
    (A1) edge [in=210, out=150,looseness=8] (A1) 
    (A1) edge [red] (B2)
    (D1) edge [red] (C1)
    (D1) edge (E1)
    (E1) edge [red, in=-30, out=-150] (A1);
    
    \node[right] at (-1,0) {$x$};
    \node[right] at (1.7,0.35) {$x^{-1}$};
    \node[right] at (4.5,0.3) {$x$};
    \node[right] at (0.45,0.3) {$\textcolor{red}{y}$};
    \node[right] at (3.05,0.35) {$\textcolor{red}{y^{-1}}$};
    \node[right] at (2.5,-1.2) {$\textcolor{red}{y}$};

\draw [draw = black, thick, arrows={-stealth} ]
         
	
    (C2) edge  (B2)    
    (A2) edge [in=210, out=150,looseness=8] (A2) 
    (A2) edge [red] (B3)
    (D2) edge [red] (C2)
    (D2) edge (E2)
    (E2) edge [red, in=-30, out=-150] (A2);
    
    \node[right] at (-1,2) {$x$};
    \node[right] at (1.7,2.35) {$x^{-1}$};
    \node[right] at (4.5,2.3) {$x$};
    \node[right] at (0.45,2.3) {$\textcolor{red}{y}$};
    \node[right] at (3.05,2.35) {$\textcolor{red}{y^{-1}}$};
    \node[right] at (2.5,0.8) {$\textcolor{red}{y}$};

\draw [draw = black, thick, arrows={-stealth} ]
         
	
    (C3) edge  (B3)    
    (A3) edge [in=210, out=150,looseness=8] (A3) 
    (A3) edge [red, in=240, out=210,looseness=2] (B1)
    (D3) edge [red] (C3)
    (D3) edge (E3)
    (E3) edge [red, in=-30, out=-150] (A3);
    
    \node[right] at (-1,4) {$x$};
    \node[right] at (1.7,4.35) {$x^{-1}$};
    \node[right] at (4.5,4.3) {$x$};
    \node[right] at (-2,1.5) {$\textcolor{red}{y}$};
    \node[right] at (3.05,4.35) {$\textcolor{red}{y^{-1}}$};
    \node[right] at (2.5,2.8) {$\textcolor{red}{y}$};
    
\end{tikzpicture}
\caption{\label{fig:Q3} The problem $Q_3=\mbox{\tt Reroute}(P_3)$.}
\end{figure}
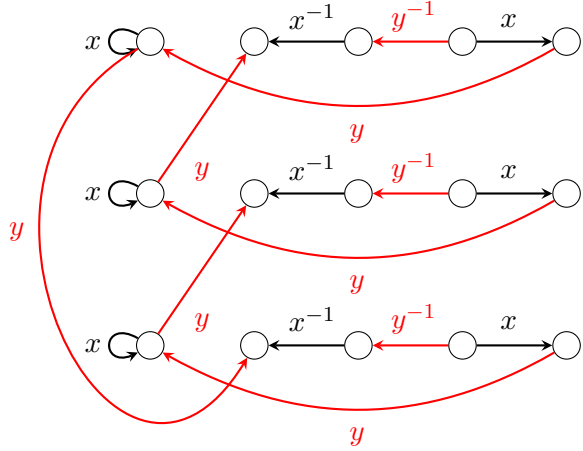
To summarize, we have the following relations
$$ P_1\Leftrightarrow P_2\Leftarrow Q_3 \Leftarrow P_3,$$
where $A\Rightarrow B$ stands for ``if $A$ has an odd-solution then $B$ has an odd-solution.''

Finally, we give an odd-solution $S_3$ for $P_3$. For this, consider
$$ \Gamma=\langle s, t\colon s^3=t^7=1, st=t^2s\rangle. $$
Let $N=\langle t\rangle\cong \mathbb{Z}/7\mathbb{Z}$. Then $N\unlhd \Gamma$ and $\Gamma/N\cong\mathbb{Z}/3\mathbb{Z}$. We have that $\Gamma$ is a meta-abelian group of order $21$. It is straightforward to check that $\Gamma$ is not nilpotent. 

Let $H=\langle s\rangle$ and $X=\Gamma/H$ be the collection of left cosets of $H$. Define
$$\begin{array}{rl} x\colon & uH\mapsto suH \\ y\colon & uH\mapsto tuH\end{array} $$
Then $|X|=7$, $x, y\in \mbox{\rm Sym}(X)$, and we have
$$ x^3=y^7=1, xy=y^2x. $$
This implies that $\langle x,y\rangle\cong\Gamma$ and hence is an odd group. It is now straightforward to check that $S_3=(X, \{x,y\})$ is an odd-solution for $P_3$. This is illustrated in Figure~\ref{fig:S3}.

\begin{figure}[htbp]
\centering
\begin{tikzpicture}[scale=0.8]
  \node [circle,draw] (A)              {$H$};
  \node [circle,draw] (B) [right=of A] {\small $tH$};
   \node [circle,draw] (F) [above right=of B] {\tiny $t^2H$};
   \node [circle,draw] (C) [below right=of F] {\tiny $t^4H$};
   \node [circle,draw] (D) [right =of C]             {\tiny $t^3H$};
  \node[circle, draw](G) [above right=of D] {\tiny $t^5H$};
  \node [circle,draw] (E) [below right=of G] {\tiny $t^6H$};

\draw [draw = black, thick, arrows={-stealth} ]
         
	(B)  edge  (F)
    (C) edge  (B)
    (F) edge (C)
    (A) edge [in=210, out=150,looseness=8] (A) 
    (D) edge (E)
    (E) edge (G)
    (G) edge (D)
    (A) edge [red] (B)
    (B) edge [red, in=-150, out=60] (F)
    (F) edge [red, in=135 , out=0] (D)
    (D) edge [red] (C)
    (C) edge [red, in=180, out=45] (G)
    (G) edge [red, in=120, out=-30]  (E)
    (E) edge [red, in=-20, out=-160] (A);
    
    \node[right] at (-2,0) {$x$};
    \node[right] at (4,0.3) {$x$};
    \node[right] at (3.2,0.7) {$x$};
    \node[right] at (5.2,1.2) {$x$};
    \node[right] at (10.5,0.3) {$x$};
    \node[right] at (11.3,0.7) {$x$};
    \node[right] at (9.3,1.2) {$x$};

    \node[right] at (0.8,0.3) {$\textcolor{red}{y}$};
    \node[right] at (2.5,1.2) {$\textcolor{red}{y}$};
    \node[right] at (7.3,0.3) {$\textcolor{red}{y}$};
      \node[right] at (6.1,-1.1) {$\textcolor{red}{y}$};
    \node[right] at (6.3,2) {$\textcolor{red}{y}$};
    \node[right] at (8.4,2) {$\textcolor{red}{y}$};
    \node[right] at (12,1.4) {$\textcolor{red}{y}$};
\end{tikzpicture}
\caption{\label{fig:S3} An odd-solution $S_3$ for $P_3$.}
\end{figure}
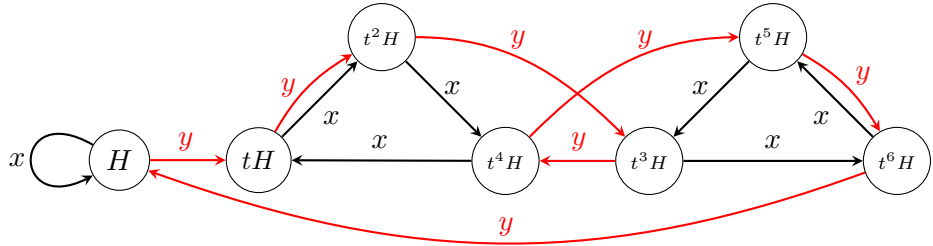

By our discussions above, $S_3$ induces solutions for $P_1, P_2$ and $Q_3$. For $P_1$, the induced solution is $S_1$.

The EP problem $P_3$ appeared in \cite{HPSW19}, where it was shown essentially that it does not have a $p$-solution for any prime $p$.

Finally, we compute the fundamental groups of these EP problems as follows:
$$\begin{array}{rcl} \pi_1(P_1)&=&\langle a, ba^{-1}bab^{-1}\rangle, \\
\pi_1(P_3)&=&\langle x, yxy^{-1}x^{-1}y\rangle,
\end{array}
$$
and if we let $u=yxy^{-1}x^{-1}y$, then
$$\begin{array}{rcl}
\pi_1(P_2)&=&\langle x, u^{-1}xu\rangle, \\
\pi_1(Q_3)&=&\langle x, u^{-1}xu, u^{-2}xu^2, u^3\rangle.
\end{array}
$$

\subsection{Nilpotent-solutions and $p$-solutions\label{subsec:nilsolution}}

The following are standard facts in group theory. 

\begin{theorem}[{\cite[5.1.3]{RobinsonBook}}]\label{thm:pnil} Any finite $p$-group is nilpotent.
\end{theorem}


The following fact follows immediately from Theorem~\ref{thm:pnil}.

\begin{proposition} Let $F$ be a finitely generated free group and let $H$ be a finitely generated subgroup of $F$. If for some prime number $p$, $H$ is closed in the pro-$p$ topology of $F$, then $H$ is closed in the pro-nilpotent topology of $F$.
\end{proposition}

We also have the following corollary of the proof of Theorems~\ref{thm:LERF2EP} and \ref{thm:solutionprotopology}.

\begin{proposition}\label{prop:pronil} Let $(X, P)$ be an EP problem where $G(X, P)$ is weakly connected. If $\pi_1(X,P)$ is closed in the pro-nilpotent topology of $F(P)$ then $(X, P)$ has a nilpotent-solution.
\end{proposition}

The following are results from Margolis--Sapir--Weil \cite{MSW01} regarding subgroups of a finitely generated free group $F$ that are dense in the pro-$p$ or the pro-nilpotent topology of $F$. 


\begin{theorem}[{\cite[Corollary 3.5]{MSW01}}]\label{thm:MSW} Let $A$ be a finite alphabet, let $H$ be a finitely generated subgroup of $F(A)$, and let $p$ be a prime number. Let $\sigma_p\colon F(A)\to (\mathbb{Z}/p\mathbb{Z})^A$ be group the homomorphism defined by letting, for any $a\in A$ and $w\in F(A)$, $\sigma_p(w)(a)$ be the mod $p$ sum of all powers of $a$ in $w$. Then $H$ is dense in the pro-$p$ topology of $F(A)$ if and only if $\sigma_p^{-1}\sigma_p(H)=F(A)$.
\end{theorem}

\begin{theorem}[{\cite[Corollary 4.2]{MSW01}}] \label{thm:MSW2} Let $A$ be a finite alphabet, let $H$ be a finitely generated subgoup of $F(A)$. Then the following are equivalent:
\begin{enumerate}
\item[\rm (i)] $H$ is dense in the pro-nilpotent topology of $F(A)$;
\item[\rm (ii)] for any prime $p$, $H$ is dense in the pro-$p$ topology of $F(A)$;
\item[\rm (iii)] for any prime $p$, $\sigma_p^{-1}\sigma_p(H)=F(A)$.
\end{enumerate}
\end{theorem}

Finally, we observe the following result.

\begin{theorem}\label{thm:nonil} Let $(X,P)$ be an EP problem where $G(X, P)$ is weakly connected. Suppose $\pi_1(X, P)\neq F(P)$. If $(X, P)$ has a nilpotent-solution, then $\pi_1(X, P)$ is not dense in the pro-nilpotent topology of $F(P)$.
\end{theorem}

\begin{proof} Suppose $(Y, Q)$ is a nilpotent-solution for $(X, P)$, with solvent $\Gamma$. Let $p\mapsto \hat{p}$ be the bijection from $P$ to $Q$ witnessing that $(Y, Q)$ is an extension of $(X, P)$. Let $x_0$ be the base point for the computation of $\pi_1(X, P)$. Since $\pi_1(X, P)\neq F(P)$, there is $w\in F(P)$ such that $\hat{w}(x_0)\neq x_0$. Thus $w_\Gamma(x_0)\neq x_0$ and $\varphi(w)=w_\Gamma\not\in \varphi(\pi_1(X, P))$. Thus $\pi_1(X, P)$ is not dense in the pro-nilpotent topology of $F(P)$.
\end{proof}

The following diagram summarizes the known implications among the statements we considered in this subsection. In the following we assume $\pi_1(X, P)\neq F(P)$.

$$
\begin{array}{ccc}
\mbox{$\exists p$ $\pi_1(X, P)$ is closed in the } & &  \mbox{$\pi_1(X, P)$ is closed in the } \\
\mbox{pro-$p$ topology of $F(P)$} & \Longrightarrow & \mbox{pro-nilpotent topology of $F(P)$} \\
\displaystyle\left\Updownarrow\vphantom{\sum}\right. & & \displaystyle\left\Downarrow\vphantom{\sum}\right. \\
\mbox{$\exists p$ $(X, P)$ has a $p$-solution} & \Longrightarrow & \mbox{$(X, P)$ has a nilpotent-solution} \\
\displaystyle\left\Downarrow\vphantom{\sum}\right. & & \displaystyle\left\Downarrow\vphantom{\sum}\right. \\
\mbox{$\exists p$ $\pi_1(X, P)$ is not dense} & \Longleftrightarrow & \mbox{$\pi_1(X, P)$ is not dense in the } \\
\mbox{ in the pro-$p$ toplogy of $F(P)$} & & \mbox{pro-nilpotent topology of $F(P)$} \\ \\
\end{array}
$$

We do not know the answers to the following questions.

\begin{problem}\label{prob:nil} If an EP problem $(X, P)$ has a nilpotent-solution, does it follow that $\pi_1(X, P)$ is closed in the pro-nilpotent topology of $F(P)$?
\end{problem}

In particular, the following problem is open.

\begin{problem}\label{prob:ope} If an EP problem $(X, P)$ has a nilpotent-solution, does it follow that an one-edge extension of $(X, P)$ also has a nilpotent-solution?
\end{problem}

By our proof of Theorem~\ref{thm:solutionprotopology}, if Problem~\ref{prob:ope} has a positive answer, then so does Problem~\ref{prob:nil}.

Next we discuss some examples. Recall that
$$ \begin{array}{rcl}
\pi_1(P_1)\!\!\!&=&\!\!\! \langle a, ba^{-1}bab^{-1}\rangle \\
\pi_1(P_3)\!\!\!&=&\!\!\! \langle x, yxy^{-1}x^{-1}y\rangle  
\end{array}
$$
from the examples in Subsection~\ref{subsec:case}. It is easy to see from applying Theorem~\ref{thm:MSW} that they are dense in the pro-$p$ topology of $F$ for any prime $p$, and by Theorem~\ref{thm:MSW2} they are dense in the pro-nilpotent topology of $F$. Since they are not $F$, it follows from Theorem~\ref{thm:nonil} that they do not have nilpotent-solutions, nor $p$-solutions for any prime $p$. 

For the problem $P_2$ in Subsection~\ref{subsec:case}, recall 
$$ \pi_1(P_2)=\langle x, u^{-1}xu\rangle $$
for some $u$. By Theorem~\ref{thm:MSW}, it is not dense in the pro-$p$ topology of $F$. However, since $P_2$ is obtained from applying the EP map {\tt Multiverse} to $P_1$, we have by Theorem~\ref{thm:multiverse} (ii) that $P_2$ does not have a nilpotent solution. In fact, if $P_2$ has a solution with solvent $\Lambda$, then $P_1$ has a solution with solvent $\Gamma$, where $\Gamma$ is isomorphic to a subgroup of $\Lambda$. If $\Lambda$ were nilpotent, then so would be $\Gamma$, contradicting our conclusion above that $P_1$ does not have nilpotent-solutions. 

For the problem $Q_3$ in Subsection~\ref{subsec:case}, recall
$$ \pi_1(Q_3)=\langle x, u^{-1}xu, u^{-2}xu^2, u^3\rangle $$
for some $u$. Since $Q_3$ is an extension of $P_2$,  we conclude that $Q_3$ does not have nilpotent-solutions.

In summary, all examples in Subsection~\ref{subsec:case} do not have nilpotent-solutions and their fundamental groups are not closed in the pro-nilpotent topology.

Before closing this section, we note the following fact.

\begin{proposition} Let $\Pi$ be a set of prime numbers. Then $\pi_1(P_1)$ is quo-$\Pi$.
\end{proposition}

\begin{proof} In the definition of $P_1$ (see Figure~\ref{fig:P1}) denote the vertices as 
$$x_0=0,\ x_1=1, \ x_2=2. $$ 
Since $G(P_1)$ is weakly connected, it suffices to show that for any $w\in F(P_1)$ and any prime $q>1$, if $w^q(x_0)=x_0$, then $w(x_0)=x_0$. 

Since $w^q\in \pi_1(P_1)$, we have that $w^q$ can be written as a concatenation of powers of $a$ and powers of $ba^{-1}bab^{-1}$. In particular, $w^q$ contains a suffix of the form $z=ba^{-1}b^{\alpha}ab^{-1}a^{\beta}$ where $\beta\in \mathbb{Z}$ and $\alpha\in \mathbb{Z}\setminus\{0\}$. It follows that either $w$ also contains $z$ as a suffix or $w$ is a suffix of $z$. Without loss of generality, we may assume $w$ is not a power of $a$. By inspection, we also have that $w$ is not of the form $b^{-1}a^{\beta}$.

By our assumption, $w(x_0)$ is defined. Assume toward a contradiction that $w(x_0)\neq x_0$. Then $w(s_0)\in \{x_1, x_2\}$. First assume $w(x_0)=x_1$. By our assumption, $w(x_1)$ is defined. This implies that $\beta=1$. Note that $(b^{-1}a)(x_1)=x_2$. We now get a contradiction since the next letter to apply in $w$ is $a$, but $a(x_2)$ is undefined. Next assume $w(x_0)=x_2$. Then either $\beta=-1$ or $\beta=0$. In the first case $w$ contains $b^{-1}a^{-1}$ as a suffix; however, $(b^{-1}a^{-1})(x_2)=b^{-1}(x_1)$ is undefined, which is a contradiction. In the second case $w$ contains $ab^{-1}$ as a suffix; however, $(ab^{-1})(x_2)=a(x_2)$ is undefined. Thus we must have $w(x_0)=x_0$ as required.
\end{proof}

\begin{corollary} For any prime $p$, the $p$-extension property and $p$-LERF do not hold.
\end{corollary}

\begin{proof} As demonstrated in the preceding proposition and the discussions before it, $\pi_1(P_1)$ is quo-$p$ for any $p$ but it is not closed in the pro-$p$ topology of $F(P_1)$. 
\end{proof}

By Theorem~\ref{thm:multiverse} (i) we also have that $P_2$ is quo-$\Pi$. Since $P_2$ does not have a $p$-solution for any prime $p$, $P_2$ is also a counterexample for the $p$-extension property.

\section{Bicyclic EP Problems}

Starting from this section, we study EP problems whose fundamental groups have exactly two generators. In fact we concentrate on cases which are even more restrictive. Throughout the rest of this section let $\Pi$ be a set of prime numbers.

\subsection{Bicyclic EP$(2)$ problems} 

\begin{definition} Let $A$ be a finite alphabet and let $F(A)$ be the free group generated by $A$. Define a homomorphism $\sigma\colon F(A)\to \mathbb{Z}^A$ by letting, for any $a\in A$ and $w\in F(A)$, $\sigma(w)(a)$ be the sum of all powers of $a$ in $w$. 
\end{definition}

\begin{theorem}\label{thm:kcyclic} Let $(X, P)$ be an EP problem and let
$$ (Y, Q)=\mbox{\tt Multiverse}(X, P). $$
Then $\pi_1(X, P)$ and $\pi_1(Y, Q)$ have the same number of generators.
\end{theorem}

\begin{proof} Combining Definition~\ref{def:multiverse} and the proofs of Theorems~\ref{thm:EP2} and \ref{thm:multiverse}, we get the following correspondence between $\pi_1(X, P)$ and $\pi_1(Y, Q)$. Suppose $P=\{p_1,\dots, p_k\}$ and $Q=\{a, b\}$. Let $x_0\in X$ and $y_0=(x_0,1)$ be the base points of the computations of $\pi_1(X, P)$ and $\pi_1(Y, Q)$, respectively. Define a homomorphism $\varphi\colon F(P)\to F(Q)$ by letting, for any $1\leq i\leq k$,
$$ \varphi(p_i)=b^{-i+1}ab^{i-1}. $$
Then we have that $w\in \pi_1(X, P)$ if and only if $\varphi(w)\in \pi_1(Y, Q)$. Moreover, note that if $u\in \pi_1(Y, Q)$ then $\sigma(u)(b)=0$ and $u$ can be written as a concatenation of elements of $\{\varphi(p)\colon p\in P\}$. Thus $\pi_1(Y, Q)\subseteq \varphi(F(P))$. Thus $\varphi$ is an isomorphism between $\pi_1(X, P)$ and $\pi_1(Y, Q)$. In particular, $\pi_1(X, P)$ and $\pi_1(Y, Q)$ have the same number of generators.
\end{proof}

\begin{definition} For any $k\geq 0$, we call an EP problem $(X, P)$ {\em $k$-cyclic} if $\pi_1(X, P)$ has $k$ generators. When $k=2$ we call $(X, P)$ {\em bicyclic}.
\end{definition}

Corollary~\ref{cor:trivial} and Theorem~\ref{thm:cyclic} established the $\Pi$-EP for $k$-cyclic problems where $k=0$ and $1$. Theorem~\ref{thm:kcyclic} implies that the $\Pi$-EP for $k$-cyclic problems is equivalent to $\Pi$-EP(2) for $k$-cyclic problems. 

In the rest of this paper, we will often focus on bicyclic EP(2) problems. That is, unless otherwise specified, we will typically consider $(X, P)$ where $P=\{a, b\}$, $G(X, P)$ is weakly connected, and $\pi_1(X, P)=\{u, v\}$ for some $u, v\in F(P)$. Note that every vertex has degree at most $4$ in $G^*(X, P)$. In view of Corollary~\ref{cor:fundgp}, without loss of generality, we also assume that $G^*(X, P)$ does not contain vertices of degree $1$; in particular, every edge of $G(X, P)$ is a part of a cycle. 

We may further classify all bicyclic EP(2) problems by the geometry of their corresponding undirected multigraphs into three classes as shown in Figure~\ref{fig:BEP2}.

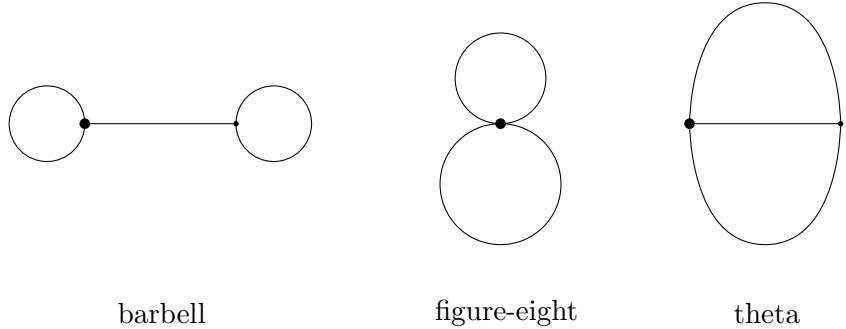
\begin{figure}[htbp]
\centering
\begin{tikzpicture}
  	\coordinate (A)  at (0,0);
  	\coordinate (B) at (2,0);
 	\coordinate (C) at (5.5,0);
 	\coordinate (D) at (8,0);
 	\coordinate (E) at (10,0);
 
  	\fill (A) circle (2pt);
	\fill (B) circle (1pt);
	\fill (C) circle (2pt);
	\fill (D) circle (2pt);
	\fill (E) circle (1pt);

\draw (-0.5,0) circle (0.5);
\draw  (A) -- (B) ;
\draw (2.5,0) circle (0.5);
\draw (5.5,0.6) circle (0.6);
\draw (5.5, -0.8) circle (0.8);
\draw (D) -- (E);
\draw plot [smooth, tension=1.8] coordinates {(8,0) (9,1.6) (10,0)};
\draw plot [smooth, tension=1.8] coordinates {(8,0) (9,-1.6) (10,0)};

    \node[right] at (0.3,-2.5) {barbell};
    \node[right] at (4.5,-2.5) {figure-eight};
    \node[right] at (8.45,-2.5) {theta};
\end{tikzpicture}
\caption{\label{fig:BEP2} A coarse classification of bicyclic EP(2) problems.}
\end{figure}
This classification is inspired by some discussions about voltage graphs in \cite{GTBook}. In Figure~\ref{fig:BEP2}, the undrawn vertices all have degree $2$. A ``barbell'' problem contains two cycles that are disjoint and connected by a path; a ``figure-eight'' problem contains two cycles which share a common vertex; a ``theta'' problem contains two cycles which share a common subpath. The highlighted vertices are the base points for the computations of the fundamental groups. 

\begin{definition} Let $u, v, u_1, v_1, u_2, v_2\in F(a, b)$. 
\begin{enumerate}
\item[(1)] We call a bicyclic EP problem $(X, P)$ an {\em $\langle u, v\rangle$-problem} if $\pi_1(X, P)=\langle u, v\rangle$.
\item[(2)] We use $\langle u_1, v_1\rangle\rightsquigarrow_{\Pi}\langle u_2, v_2\rangle$ to denote the statement that 
\begin{quote}
if an $\langle u_1, v_1\rangle$-problem has a $\Pi$-solution, then an $\langle u_2, v_2\rangle$-problem has a $\Pi$-solution.
\end{quote}
\item[(3)] We write $\langle u_1, v_1\rangle\leftrightsquigarrow_{\Pi} \langle u_2, v_2\rangle$ if both $\langle u_1, v_1\rangle\rightsquigarrow_{\Pi} \langle u_2, v_2\rangle$ and $\langle u_2, v_2\rangle\rightsquigarrow_{\Pi} \langle u_1, v_1\rangle$.
\end{enumerate}
\end{definition}

We will consider a subclass of bicyclic EP(2) problems, namely $\langle a, w\rangle$-problems for irreducible words $w\in F(a, b)$ which starts and ends with $b$ or $b^{-1}$. We use the following notation.

\begin{definition} We let $\Sigma(b, b^{-1})$ denote the set of all irreducible words $w\in F(a, b)$ which starts with $b$ and ends with $b^{-1}$. Similarly define $\Sigma(b^{-1}, b)$, $\Sigma(b, b)$ and $\Sigma(b^{-1}, b^{-1})$. Let $\Sigma$ be the set of all irreducible words $w\in F(a, b)$ which starts and ends with $b$ or $b^{-1}$, i.e., 
$$ \Sigma=\Sigma(b,b)\cup \Sigma(b^{-1}, b)\cup \Sigma(b, b^{-1})\cup \Sigma(b^{-1}, b^{-1}). $$
\end{definition}

If $w\in \Sigma(b, b)\cup \Sigma(b^{-1}, b^{-1})$, then an $\langle a, w\rangle$-problem is a figure-eight problem; they are in fact some of the simplest figure-eight problems. The problem $P_3$ in Subsection~\ref{subsec:case} is an example. The problem $P_2$ in Subsection~\ref{subsec:case} is an example of the simplest barbell problem, which is an $\langle a, w^{-1}aw\rangle$-problem for $w\in \Sigma$. 

We will also use the following notation throughout the rest of the paper.

\begin{definition}\label{def:*} For $u\in \Sigma$, let $\phi_u\colon F(a, b)\to F(a, b)$ be the homomorphism given by $\phi_u(a)=a$ and $\phi_u(b)=u$. Let ${}^*=\phi_{b^{-1}ab}$.
\end{definition}

We have the following basic facts.

\begin{proposition}\label{prop:awbasic} Let $w\in \Sigma$. Then the following hold:
\begin{enumerate}
\item[\rm (i)] $\langle a, w\rangle\leftrightsquigarrow_{\Pi}\langle a, w^{-1}\rangle$.
\item[\rm (ii)] $\langle a, w\rangle\leftrightsquigarrow_{\Pi}\langle a, w^*\rangle$.
\item[\rm (iii)] For any $u\in \Sigma$, $\langle a, \phi_u(w)\rangle\rightsquigarrow_{\Pi} \langle a, w\rangle$. Moreover, if an $\langle a, \phi_u(w)\rangle$-problem has a solution with solvent $\Lambda$, then an $\langle a, w\rangle$-problem has a solution with solvent $\Gamma$, where $\Gamma$ is a subgroup of $\Lambda$.
\item[\rm (iv)] If $w\in \Sigma(b, b)\cup \Sigma(b^{-1}, b^{-1})$, then $ \langle a, w\rangle \rightsquigarrow_{\Pi} \langle a, w^{-1}aw\rangle$.
\item[\rm (v)] For any $\Pi$-number $k$, we have $\langle a, w\rangle\rightsquigarrow_{\Pi}\langle a^k, w\rangle$.
\item[\rm (vi)] For any $\Pi$-numbers $k, \ell$, we have $\langle a,w^{-1}aw\rangle\rightsquigarrow_{\Pi}\langle a^k, w^{-1}a^\ell w\rangle$.
\item[\rm (vii)] $\langle a, waw^{-1}\rangle\leftrightsquigarrow_{\Pi}\langle a, w^{-1}aw\rangle$.
\end{enumerate}
\end{proposition}

\begin{proof} (i) holds because $\langle a, w\rangle=\langle a, w^{-1}\rangle$ as subgroups of $F(a,b)$. 

(ii) is obtained by applying the EP map {\tt Multiverse} to an $\langle a, w\rangle$-problem. The resulting EP problem is an $\langle a, w^*\rangle$-problem. 

For (iii),  consider a $\Pi$-solution $(X, P)$, with solvent $\Lambda$, of an $\langle a, \phi_u(w)\rangle$-problem. Define an EP problem $(X, \{\tilde{a}, \tilde{b}\})$ by $\tilde{a}(x)=a(x)$ and $\tilde{b}(x)=u_\Lambda(x)$. Then $(X, \{\tilde{a}, \tilde{b}\})$ is still a solution. If $\Gamma$ is the solvent of $(X, \{\tilde{a}, \tilde{b}\})$, then $\Gamma$ is a subgroup of $\Lambda$, hence it is still a $\Pi$-group. It is then clear that $(X, \{\tilde{a}, \tilde{b}\})$ is a $\Pi$-solution for an $\langle a, w\rangle$-problem.

For (iv), consider an $\langle a, w\rangle$-problem $(X, P)$ where $w\in \Sigma(b,b)\cup \Sigma(b^{-1},b^{-1})$. Let $x_0$ be the base point of the computation of $\pi_1(X, P)$. Let $k\geq 2$ be a $\Pi$-number. Then ${\tt Reroute}(X, P; k, b, x_0)$ is an $\langle a, w^{-1}aw\rangle$-problem. 

For (v), consider an $\langle a, w\rangle$-problem $(X, P)$. Let $x_0$ be the base point of the computation of $\pi_1(X, P)$. Then ${\tt Reroute}(X, P; k, a, x_0)$ is an $\langle a^k, w\rangle$-problem. 

For (vi), consider an $\langle a, w^{-1}aw\rangle$-problem $(X, P)$. Let $x_0$ be the base point of the computation of $\pi_1(X, P)$. Let $x_1=w(x_0)$. Let $(Y, Q)={\tt Reroute}(X, P; k, a, x_0)$. Then $(Y, Q)$ is an $\langle a^k, w^{-1}aw\rangle$-problem. Furthermore, let $(Z, R)={\tt Reroute}(Y, Q; \ell, a, x_1)$. Then $(Z, R)$ is an $\langle a^k, w^{-1}a^{\ell}w\rangle$-problem.

Finally, for (vii), just note that $\langle a, waw^{-1}\rangle$ is a conjugate of $\langle a, w^{-1}aw\rangle$ in $F(a, b)$. Thus they are both closed or both not closed in the pro-$\Pi$ topology of $F(a, b)$.
\end{proof}

\subsection{Wandering problems}

For the rest of the paper we focus on $\langle a, w^{-1}aw\rangle$-problems for $w\in \Sigma$. We classify all $\langle a, w^{-1}aw\rangle$ as $\kappa$-weakly wandering problems (to be defined in full generality in Subsection~\ref{subsec:2ww}) for $\kappa\geq 1$. In this and the next subsections we treat the wandering problems and (1-)weakly wandering problems, and in the next section we treat $2$-weakly wandering problems. The difference between 1-weakly wandering problems and $2$-weakly wandering ones is that the former admit $\Pi$-solutions for any $\Pi$.

In this subsection we consider the wandering problems, which form a special subclass of weakly wandering problems, to illustrate the technique which is to be reused repeatedly in more general proofs.

We first note that all $\langle a, w^{-1}aw$-problems are quo-$\Pi$ for any $\Pi$.

\begin{proposition} Let $\Pi$ be a set of prime numbers. Then for any $w\in \Sigma$, $\langle a, w^{-1}aw\rangle$ is a quo-$\Pi$ subgroup of $F(a, b)$.
\end{proposition}

\begin{proof} Let $(X, P)$ be an EP problem with $\pi_1(X, P)=\langle a, w^{-1}aw\rangle$. More specifically, suppose $w$ can be written in reduced form 
$$ w=c_k\cdots c_1 $$
where $c_1, c_k\in \{b, b^{-1}\}$ and $c_2, \dots, c_{k-1}\in \{a, a^{-1}, b, b^{-1}\}$. Let 
$$X=\{x_0, x_1, \dots, x_k\}$$ and $P=\{\tilde{a}, \tilde{b}\}$ where
$$ \tilde{a}=\{(x_0, x_0), (x_k, x_k)\}\cup \big\{(x_{i-1}, x_{i})\colon c_i=a\big\}\cup\big\{(x_i, x_{i-1})\colon c_i=a^{-1}\big\} $$
and
$$ \tilde{b}=\big\{(x_{i-1}, x_i)\colon c_i=b\big\}\cup\{(x_i, x_{i-1})\colon c_i=b^{-1}\big\}. $$
We define a {\em level function} $\ell\colon X\to \{0, \dots, k\}$ by letting $\ell(x_i)=i$ for all $0\leq i\leq k$.

Now write $w$ as $w=w_0^{-1}uw_0$ where there is no cancellation in $uu$ (i.e. $uu$ is in reduced form). Let $h=|u|$. Assume $q>1$ is any prime number such that $w^q(x_0)=x_0$. Then for some $0\leq i_0\leq k$ we have $u^q(x_{i_0})=x_{i_0}$ (in fact $x_{i_0}=w_0(x_0)$). We show that $u(x_{i_0})=x_{i_0}$. From this it follows that $w(x_0)=x_0$. 

Note that our assumption implies that for any $N$, $u^N(x_{i_0})$ is defined. By writing out the results of applying each suffix of $u^N$ to $x_{i_0}$, we get an infinite sequence in $X$
$$ y_0=x_{i_0}, y_1, y_2, \dots $$
Our assumption implies that $y_{qh}=y_0$, and thus the sequence has period $qh$. Note that for any $n\in\mathbb{N}$, $\ell(y_n)-\ell(y_{n+1})\in\{-1,0,1\}$. 

Define a function $g\colon \mathbb{N}\to\{0,\dots, k\}$ by $g(n)=\ell(y_{n+h})-\ell(y_n)$. We claim that $g(n)=0$ for some $n\in\mathbb{N}$. Toward a contradiction, assume $g(n)\neq 0$ for all $n\in\mathbb{N}$. We observe that for any $n\in\mathbb{N}$, 
$$ \begin{array}{rcl} g(n+1)-g(n)&=&\ell(y_{n+h+1})-\ell(y_{n+1})-\ell(y_{n+h})+\ell(y_n) \\
& \leq &|\ell(y_{n+h+1})-\ell(y_{n+h})|+|\ell(y_{n+1})-\ell(y_n)| \leq 2.
\end{array}$$
Also observe that for any $n\in\mathbb{N}$, we have
$$ g(n)+g(n+h)+\cdots+g(n+(q-1)h)=\ell(y_{qh+n})-\ell(y_n)=0. $$
Thus 
$$ \displaystyle\sum_{n=0}^{qh-1}g(n)=\displaystyle\sum_{j=0}^{h-1}(g(j)+g(j+h)+\cdots+g(j+(q-1)h))=0. $$
It follows that there is $0\leq n\leq qh-1$ such that $g(n)g(n+1)<0$, and since $|g(n+1)-g(n)|\leq 2$, we have $(g(n), g(n+1))=(1,-1)$ or $(-1,1)$. Without loss of generality, assume $g(n)=1$ and $g(n+1)=-1$ (the argument for the other case is similar). 
Then $\ell(y_{n+h})-\ell(y_n)=1$ and $\ell(y_{n+h+1})-\ell(y_{n+1})=-1$. Since $|\ell(y_{n+h})-\ell(y_{n+h+1})|\leq 1$, we must have $\ell(y_{n+1})=\ell(y_n)+1$. This means that $\ell(y_{n+h+1})=\ell(y_{n+1})-1=\ell(y_n)$ and $\ell(y_{n+h})=\ell(y_n)+1=\ell(y_{n+1})$. Thus $y_{n+h+1}=y_n$ and $y_{n+h}=y_{n+1}$. Let $c\in\{a, a^{-1}, b, b^{-1}\}$ be the letter such that $c(y_n)=y_{n+1}$. By our definition of the sequence $(y_n)$, we get that $c(y_{n+h})=y_{n+h+1}$, and thus $c(y_{n+1})=y_n$. This contradicts our definition of $(X, P)$.

Thus we have shown that there is $n\in\mathbb{N}$ such that $g(n)=0$, which implies that $y_{n+h}=y_n$. Let $v\in F(a, b)$ be the word such that $v(y_0)=y_n$. Then by the definition of the sequence $(y_n)$, we have $v(y_h)=y_{n+h}$. Since $v$ is a partial permutation, we conclude that $y_h=y_0$, or $u(x_{i_0})=x_{i_0}$ as required.
\end{proof}

\begin{definition} Let $w\in \Sigma$. Assume
$$ w=b^{\beta_k}a^{\alpha_{k-1}}b^{\beta_{k-1}}\cdots a^{\alpha_1}b^{\beta_1}, $$
where $k\geq 1$ and $\alpha_1,\dots, \alpha_{k-1}, \beta_1,\beta_1, \dots, \beta_k\in \mathbb{Z}\setminus\{0\}$. 
\begin{enumerate}
\item[(1)] The {\em $a$-power sequence} of $w$ is 
$$ (\alpha_1, \dots, \alpha_{k-1}); $$
The {\em $b$-power sequence} of $w$ is 
$$ (\beta_1, \dots, \beta_k). $$
\item[(2)] The {\em level sequence} of $w$ is $(\ell_0,\dots, \ell_k)$ where
$$ \ell_0=0, \ell_{i}=\ell_{i-1}+\beta_i \mbox{ for all $1\leq i\leq k$.} $$
\item[(3)] We say that a sequence $(n_0, \dots, n_k)$ is {\em wandering} if for any $0<i<k$, $n_0\neq n_i$ and $n_k\neq n_i$, or equivalently $\{n_1, \dots, n_{k-1}\}\cap \{n_0, n_k\}=\varnothing$. We say $w$ is {\em wandering} if the level sequence of $w$ is wandering. If $w$ is wandering, we call an $\langle a, w^{-1}aw\rangle$-problem a {\em wandering problem}.
\item[(4)] We say that a sequence $(n_0, \dots, n_k)$ is {\em returning} if $k$ is even and for all even $0\leq i\leq k$, $n_i=0$ and for all odd $0\leq i\leq k$, $n_i=1$. We say $w$ is {\em returning} if the level sequence of $w$ is returning. If $w$ is returning, we call an $\langle a,w^{-1}aw\rangle$-problem a {\em returning problem}.
\end{enumerate}
\end{definition}

The only word in $F(a, b)$ that is both wandering and returning has the level sequence $(0,1,0)$, and therefore is of the form $b^{-1}a^\alpha b$ for some $\alpha\neq 0$. Recall from Definition~\ref{def:*} that ${}^*\colon F(a, b)\to F(a, b)$ is the homomorphism given by $a^*=a$ and $b^*=b^{-1}ab$. Then for any $w\in \Sigma$, $w^*\in \Sigma(b^{-1}, b)$ is returning. For any $w\in \Sigma$, note that 
$$ (w^{-1}aw)^*=(w^*)^{-1}aw^*. $$
By Proposition~\ref{prop:awbasic} (ii), any $\langle a, w^{-1}aw\rangle$-problem has a $\Pi$-solution if and only if any $\langle a, (w^*)^{-1}aw^*\rangle$-problem, which is a returning problem, has a $\Pi$-solution.

\begin{theorem}\label{thm:wandering} Let $\Pi$ be a set of prime numbers. Any wandering problem has a $\Pi$-solution. I.e., for any wandering $w\in \Sigma$, an $\langle a, w^{-1}aw\rangle$-problem has a $\Pi$-solution. Equivalently, for any wandering $w\in \Sigma$, $\langle a, w^{-1}aw\rangle$ is closed in the pro-$\Pi$ topology of $F(a,b)$.
\end{theorem}

\begin{proof} Let 
$$w=b^{\beta_k}a^{\alpha_{k-1}}b^{\beta_{k-1}}\cdots a^{\alpha_1}b^{\beta_1} $$
with $k\geq 1$ and $\beta_1,\dots, \beta_k, \alpha_1, \dots, \alpha_{k-1}\in \mathbb{Z}\setminus\{0\}$. Assume $w$ is wandering, that is, the level sequence $(\ell_0,\dots, \ell_k)$ of $w$ is wandering. Let 
$$ L^+=\max\{\ell_i\colon 0\leq i\leq k\} \mbox{ and } L^-=\min\{\ell_i\colon 0\leq i\leq k\}. $$
Let $m\geq L^+-L^-+1$ be a $\Pi$-number. We use a variation of the EP map {\tt Multiverse} as follows. 

Let $X$ be a set with $|\alpha_1|+\cdots+|\alpha_{k-1}|+1$ many elements. We name the elements of $X$ as follows. $X$ contains a subset $Y=\{y_1,\dots, y_k\}$ and we can write $X$ as a union
$$ X=X_1\cup \cdots \cup X_{k-1}. $$
For each $1\leq i\leq k-1$, 
$$X_i=\{x^i_0,\dots, x^i_{|\alpha_i|}\}. $$
In addition, we have for each $1\leq i\leq k-1$,
$$ y_i=x^i_0 \mbox{ and } y_{i+1}=x^i_{|\alpha_i|}. $$
This finishes naming all the elements of $X$. Thus all elements of $X$ can be enumerated as
$$\begin{array}{rl} y_1=x^1_0, x^1_1,\dots, &\!\!\!\!x^1_{|\alpha_1|}=y_2=x^2_0, x^2_1,\dots, x^2_{|\alpha_2|}=y_3=x^3_0,\cdots \\ \\
 \cdots, &\!\!\!\!x^{k-2}_{|\alpha_{k-2}|}=y_{k-1}=x^{k-1}_0, x^{k-1}_1,\dots,  x^{k-1}_{|\alpha_{k-1}|}=y_k. 
\end{array}$$
We define $k-1$ many partial permutations $p_1,\dots, p_{k-1}$ on $X$ as follows. For each $1\leq i\leq k-1$, let
$$ p_i=\left\{\begin{array}{ll}\left\{ (x^i_j, x^i_{j+1})\colon 0\leq j<|\alpha_i|\right\}, & \mbox{ if $\alpha_i>0$,} \\ \\
\left\{ (x^i_{j+1}, x^i_j)\colon 0\leq j<|\alpha_i|\right\}, & \mbox{ if $\alpha_i<0$.} 
\end{array}\right.
$$
Let 
$$P=\left\{p_{i_1}\cup \cdots \cup p_{i_\lambda}\colon 1\leq i_1<\cdots<i_\lambda\leq k-1, \ell_{i_1}=\cdots=\ell_{i_\lambda}\right\}. $$
 Then $(X, P)$ is an EP problem with $G(X, P)$ weakly connected. Moreover, $\pi_1(X, P)$ is trivial. By Corollary~\ref{cor:trivial}, $(X,P)$ has a $\Pi$-solution $(X', P')$. Let $p\mapsto p'$ be the bijection from $P$ onto $P'$. For notational simplicity, we use $p'_i$ to denote the extension of $p_i$ in $P'$, with the understanding that if $\ell_{i_1}=\ell_{i_2}$, then $p'_{i_1}=p'_{i_2}$.

Next we define an EP problem $(Z, R)$ similar to $\mbox{\tt Multiverse}(X', P';m)$. Let $c\leq L^-$ and $d\geq L^+$ be such that $d-c+1=m$. Let $$Z=X'\times \{c, c+1,\dots, d-1, d\}$$ and $R=\{\tilde{a}, \tilde{b}\}$, where
$$ \tilde{b}(x, t)=\left\{\begin{array}{ll} (x, t+1), & \mbox{ if $c\leq t<d$,} \\
(x, c), & \mbox{ if $t=d$;}\end{array}\right.
$$
and
$$ \tilde{a}(x, t)=\left\{\begin{array}{ll} (p'_i(x), t), & \mbox{ if $t=\ell_i$ for $1\leq i\leq k-1$,} \\
(x, t), & \mbox{ otherwise.}
\end{array}\right.
$$
Then $(Z,R)$ is a solution. We verify that $R$ is quo-$\Pi$. 

Let $q\not\in\Pi$, $u\in F(R)$ and $z=(x, t)\in Z$ be such that $u^q(z)=z$. Suppose 
$$ u=\tilde{b}^{\gamma_n}\tilde{a}^{\eta_n}\cdots \tilde{b}^{\gamma_1}\tilde{a}^{\eta_1}\tilde{b}^{\gamma_0} $$
where $\gamma_0, \gamma_n\in\mathbb{Z}$ and $\eta_1, \dots, \eta_n, \gamma_1,\dots, \gamma_{n-1}\in \mathbb{Z}\setminus\{0\}$. We have that $q(\gamma_0+\dots+\gamma_n)\equiv 0\pmod{m}$. It follows that $\gamma_0+\dots+\gamma_n\equiv 0\pmod{m}$. Let 
$$\begin{array}{rcl} s_0&=&(t+\gamma_0-c\!\!\!\!\mod m)+c, \\
s_1&=&(s_0+\gamma_1-c\!\!\!\!\mod m)+c,  \\
&&\cdots\cdots \\
s_{n-1}&=&(s_{n-2}+\gamma_{n-1}-c\!\!\!\!\mod m)+c. 
\end{array}$$
Let 
$$ A=\big\{ h\in\{0,\dots, n-1\}\colon s_h\in\{\ell_1,\dots, \ell_{k-1}\}\big\}. $$
Let $h_0,\dots,h_r$ enumerate all elements of $A$ in increasing order. For each $h\in A$, let $1\leq i_h\leq k-1$ be such that $s_h=\ell_{i_h}$. Then we have that
$$ \big({(p'_{i_{h_r}})}^{\eta_{h_r}}\cdots {(p'_{i_{h_1}})}^{\eta_{h_1}}{(p'_{i_{h_0}})}^{\eta_{h_0}}\big)^q(x)=x. $$
Since $q\not\in\Pi$ and $P'$ is quo-$\Pi$, we have
$$ {(p'_{i_{h_r}})}^{\eta_{h_r}}\cdots {(p'_{i_{h_1}})}^{\eta_{h_1}}{(p'_{i_{h_0}})}^{\eta_{h_0}}(x)=x $$
and $u(z)=z$. This shows that $R$ is quo-$\Pi$.

Hence $(Z,R)$ is a $\Pi$-solution. To finish the proof, we only need to show that an $\langle a, w^{-1}aw\rangle$-problem is a subproblem of $(Z,R)$. For this, we consider the path in $G(Z,R)$ consisting of the following vertices:
$$ (y_1, \ell_0)=(y_1, 0), (y_0, \mbox{\rm sgn}(\beta_1)), \dots, (y_1, \beta_1)=(x^1_0, \ell_1), \dots, (x^1_{|\alpha_1}, \ell_1)=(y_2, \ell_1), $$
$$ \cdots, (y_2, \ell_2)=(x^2_0, \ell_2),  \cdots\cdots, (y_{k-1}, \ell_{k-1})=(x^{k-1}_0, \ell_{k-1}), $$
$$ \cdots, (x^{k-1}_{|\alpha_{k-1}|}, \ell_{k-1})=(y_k, \ell_{k-1}), \dots, (y_k, \ell_k). $$
The path consists of $k$ many subpaths with labels $\tilde{b}$ or $\tilde{b}^{-1}$, namely from $(y_i, \ell_{i-1})$ to $(y_i, \ell_i)$ for $1\leq i\leq k$, and $k-1$ many subpaths with labels $\tilde{a}$ or $\tilde{a}^{-1}$, namely from $(x^i_0, \ell_i)$ to $(x^i_{|\alpha_i|}, \ell_i)$ for $1\leq i\leq k-1$. It is clear that the path corresponds to the word $w$, and for the two endpoints $z=(y_1, \ell_0)$ and $z=(y_k, \ell_k)$, we have $\tilde{a}(z)=z$. Thus we have found an $\langle a, w^{-1}aw\rangle$-problem $(Y,Q)$ which is a subproblem of $(Z,R)$, and for which $(Z,R)$ is a $\Pi$-solution.
\end{proof}

The $\langle a, w^{-1}aw\rangle$-subproblem in the above proof, which is a barbell problem,  together with a part of its $\Pi$-solution, are illustrated in Figure~\ref{fig:wandering}. For simplicity, Figure~\ref{fig:wandering} only illustrates the situation where the level sequence of $w$ is increasing. The proof works for the general case as long as $\{\ell_1,\dots, \ell_{k-1}\}\cap \{\ell_0, \ell_k\}=\varnothing$.

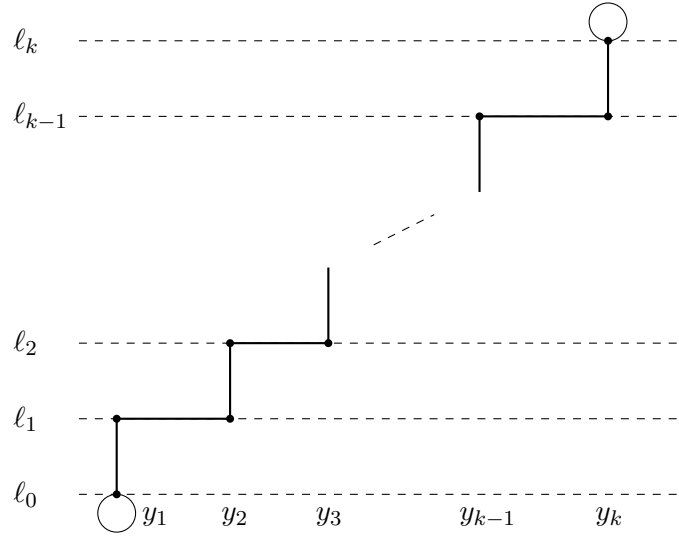
\begin{figure}[htbp]
\centering
\begin{tikzpicture}

\draw[dashed] (0,0) -- (8,0);
\draw[dashed] (0,1) -- (8,1);
\draw[dashed] (0,2) -- (8,2);
\draw[dashed] (0,5) -- (8,5);
\draw[dashed] (0,6) -- (8, 6);

\node[right] at (-1,0) {$\ell_0$};
\node[right] at (-1, 1) {$\ell_1$};
\node[right] at (-1,2) {$\ell_2$};
\node[right] at (-1,5) {$\ell_{k-1}$};
\node[right] at (-1,6) {$\ell_k$};

\node[right] at (0.7,-0.3) {$y_1$};
\node[right] at (1.75,-0.3) {$y_2$};
\node[right] at (3,-0.3) {$y_3$};
\node[right] at (4.9, -0.3) {$y_{k-1}$};
\node[right] at (6.7, -0.3) {$y_k$};

\fill (0.5,0) circle (1.5pt);
\fill (0.5,1) circle (1.5pt);
\fill (2,1) circle (1.5pt);
\fill (2,2) circle (1.5pt);
\fill (3.3,2) circle (1.5pt);
\fill (5.3,5) circle (1.5pt);
\fill (7,5) circle (1.5pt);
\fill (7,6) circle (1.5pt);

\draw[thick] (0.5,0) -- (0.5,1) -- (2,1) -- (2,2) --(3.3,2) --(3.3,3);
\draw[thick] (5.3,4) -- (5.3,5) -- (7,5) -- (7,6);

\draw (0.5,-0.25) circle (0.25);
\draw (7,6.25) circle (0.25);

\draw[dashed] (3.9,3.3) -- (4.7,3.7);

\end{tikzpicture}
\caption{\label{fig:wandering} A wandering problem and its solution.}
\end{figure}

\begin{corollary}\label{cor:wandering} Let $w\in\Sigma$ be a returning word of the form
$$ w=b^{-1}a^{\alpha_{2k+1}}b\cdots b^{-1}a^{\alpha_3}ba^{\alpha_2}b^{-1}a^{\alpha_1}b. $$
If the sequence
$$ (0, \alpha_1, \alpha_1+\alpha_3, \cdots, \alpha_1+\alpha_3+\cdots+\alpha_{2k+1}) $$
is wandering, then an $\langle a, w^{-1}aw\rangle$-problem has a $\Pi$-solution.
\end{corollary}

\begin{proof} Let ${}^*\colon F(a, b)\to F(a, b)$ be the homomorphism given by $a^*=a$ and $b^*=b^{-1}ab$. Then $w=u^*$ where 
$$ u=b^{\alpha_{2k+1}}a^{\alpha_{2k}}\cdots b^{\alpha_3}a^{\alpha_2}b^{\alpha_1}. $$
The sequence 
$$ (0, \alpha_1, \alpha_1+\alpha_3, \cdots, \alpha_1+\alpha_3+\cdots+\alpha_{2k+1}) $$
is the level sequence of $u$. Now the corollary follows from Theorem~\ref{thm:wandering} and Proposition~\ref{prop:awbasic} (ii). 
\end{proof}

\subsection{Base problems and basic walks\label{subsec:baseproblem}}
Throughout this subsection let $\Pi$ be a set of prime numbers. We would like to define some generalizations of the notion of wandering word/problem. To set up the definitions, as well as to fix some notation for use in the rest of the paper, we first analyze the proof of Theorem~\ref{thm:wandering}.

In the proof of Theorem~\ref{thm:wandering} we considered a wandering word of the form 
\begin{equation}\label{eqn:w} w=b^{\beta_k}a^{\alpha_{k-1}}\cdots a^{\alpha_1}b^{\beta_1} \end{equation}
for $\alpha_1,\dots, \alpha_{k-1}, \beta_1,\dots, \beta_k\in\mathbb{Z}\setminus\{0\}$. 
First we defined an EP problem $(X, P)$ with trivial fundamenatal group. A more intuitive description of $G(X, P)$ is the following. $G(X, P)$ is a labelled multidigraph which is simply a labelled directed path of length 
$$ |\alpha_1|+\cdots+|\alpha_{k-1}|, $$
where $\alpha_1, \dots, \alpha_{k-1}$ are exponents of $a$ from the expression (\ref{eqn:w}). 
In this path, the first $|\alpha_1|$ many edges are labelled by $\ell_1$, the next $|\alpha_2|$ many edges are labelled by $\ell_2$, etc. where $(\ell_0, \ell_1, \dots, \ell_k)$ is the level sequence of $w$. The directions of the edges are determined by the signs of $\alpha_1$, $\alpha_2$, etc. Thus $|P|=|\{\ell_1,\dots, \ell_{k-1}\}|$. 

Next, using a variation of the EP map {\tt Multiverse} we obtained an $\langle a, w^{-1}aw\rangle$-problem $(Y, Q)$ and a $\Pi$-solution $(Z, R)$ for $(Y, Q)$. We call this $(Y,Q)$ the {\em base problem} for $w$. For a wandering word $w$, the base problem is an $\langle a, w^{-1}aw\rangle$-problem, namely, the two endpoints are fixed points of $a$ and between the two endpoints there is a directed path corresponding to $w$. The proof of Thereom~\ref{thm:wandering} shows that in the wandering case the base problem always has a $\Pi$-solution.

Now, with the same proof of Theorem~\ref{thm:wandering}, we generalize this to arbitrary $w\in \Sigma$ of the form (\ref{eqn:w}). Given $w\in\Sigma$, we define an EP problem $(X_w, P_w)$ as follows. Let $(\ell_0, \dots, \ell_k)$ be the level sequence of $w$. Let
$$ L=\big\{1\leq \lambda\leq k-1\colon \ell_\lambda\in \{\ell_0, \ell_k\}\big\}. $$
Let $\lambda_1<\cdots<\lambda_m$ enumerate the elements of $L$. Let $G_0=G(X_0, P_0)$ be a labelled multidigraph which is a labelled directed path of length
$$ |\alpha_1|+\cdots+|\alpha_{k-1}|$$
as defined above. In particular, it consists of subpaths $\pi_1,\dots, \pi_{k-1}$, where each $\pi_i$ is of length $|\alpha_i|$, with each edge in $\pi_i$ lablled by $\ell_i$ and the direction of the edge determined by the sign of $\alpha_i$. Denote the two endpoints of $G_0$ as $y_0$ and $y_k$ as before. Now obtain a labelled multidigraph $G_1$ by removing the subpaths $\pi_{\lambda_1}, \dots, \pi_{\lambda_m}$ and identifying the two endpoints of each of them. Next perform Stallings folding maps to $G_1$ (provided that only edges with the same label can be folded) until no more foldings are possible. The resulting labelled multidigraph is $G(X_w, P_w)$ for some EP problem $(X_w, P_w)$. Note that $|P_w|=|\{\ell_1,\dots, \ell_{k-1}\}\setminus\{\ell_0, \ell_k\}|$.

It is clear that $G(X_w, P_w)$ is a graph-theoretic tree and thus the fundamental group of $(X_w, P_w)$ is trivial. Thus, just as in the proof of Theorem~\ref{thm:wandering}, we may find a $\Pi$-solutoin $(X'_w, P'_w)$ for $(X_w, P_w)$, and use the same variation of ${\tt Multiverse}(X'_w, P'_w; M)$ (where $M$ is a sufficiently large $\Pi$-number) to obtain a $\Pi$-solution $(Z_w, R_w)$. We also identify a subproblem $(Y_w, Q_w)$ of $(Z_w, R_w)$ in the same fashion as in the proof of Theorem~\ref{thm:wandering}. That is, $(Y_w, Q_w)$ is obtained by applying the word $w$ to $(y_0,0)$. This results in a walk (rather than a directed path in the wandering case) between two fixed points of $a$, namely $(y_0,0)$ and $(y_k,\ell_k)$.
We call this walk the {\em basic walk} for $w$, and denote it by $W_w$. 

We can represent this walk $W_w$ symbolically as follows. Removing the subwords $a^{\alpha_{\lambda_1}}, \dots, a^{\alpha_{\lambda_m}}$ from $w$ results in $m+1$ many subwords $\tilde{w}_1, \dots, \tilde{w}_{m+1}$. That is,
$$ w=\tilde{w}_{m+1}a^{\lambda_m}\tilde{w}_m\cdots \tilde{w}_2a^{\lambda_1}\tilde{w}_1. $$
Consider the word
$$ \tilde{w}=\tilde{w}_{m+1}\tilde{w}_m\cdots \tilde{w}_1. $$
That is,
\begin{equation}\label{eqn:tildew} \tilde{w}=(b^{\beta_k}\cdots a^{\alpha_{\lambda_m+1}}b^{\beta_{\lambda_m+1}})b^{\beta_{\lambda_m}}a^{\alpha_{\lambda_m-1}}\cdots a^{\alpha_{\lambda_1+1}}(b^{\beta_{\lambda_1+1}}b^{\beta_{\lambda_1}}a^{\alpha_{\lambda_1-1}}\cdots a^{\alpha_1}b^{\beta_1}). \end{equation}
Applying the word $\tilde{w}_1$ to $(y_0, 0)$ results in a directed path $\nu_0$. Note that $\tilde{w}_1$ is wandering, and by our construction, both endpoints of $\nu_0$ are fixed points of $a$. Next apply the word $\tilde{w}_2$ to the ending vertex of $\nu_0$ to obtain a directed path $\nu_1$, and repeat in the same fashion until the word $\tilde{w}_{m+1}$ has been used. This procedure results in $m+1$ many directed paths $\nu_0,\dots, \nu_{m}$, with the ending vertex of $\nu_i$ the same as the starting vertex of $\nu_{i+1}$ for all $0\leq i\leq m$. Note that by our construction, both endpoints of each $\nu_i$ are fixed points of $a$, and so the application of $w$ to $(y_0,0)$ and that of $\tilde{w}$ result in the same walk $W_w$. The fixed points of $a$ located on the basic walk $W_w$, together with two endpoints of $W_w$ (the {\em starting} and {\em ending} vertices), are called {\em special}. We denote the set of all special vertices in $W_w$ by $S_w$.

Let $(Y_w, Q_w)$ be the EP problem so that $Y_w$ consists of the same vertices as $W_w$ and $G(Y_w, Q_w)$ is the union of the directed walk $W_w$ with $a$-labelled loops at all special vertices of $W_w$. We call $(Y_w, Q_w)$ the {\em base problem} for $w$. Note that, it follows from the above proof that, the definitions of the basic walk $W_w$ and the base problem $(Y_w, Q_w)$ do not depend on $\Pi$. In other words, if $\Pi\neq \Pi'$ are distinct sets of prime numbers, then the basic walks and the base problems defined with $\Pi$ and $\Pi'$ as parameters are isomorphic to each other.

The fundamental group of $(Y_w, Q_w)$ has $m+2$ many generators. Nevertheless, our construction guarantees the following basic fact.

\begin{proposition} For any $w\in \Sigma$, the base problem $(Y_w, Q_w)$ has a $\Pi$-solution.
\end{proposition}

To recap the notation, the basic walk $W_w$ consists of $m+1$ many directed paths $\nu_0,\dots, \nu_{m}$. For each $0\leq i\leq m$, the special vertices occurring in $\nu_i$ are exactly its endpoints. Thus if we enumerate the elements of $S_w$ in the order that they are visited along the basic walk $W_w$, we have exactly
$$ x_0, x_1, \dots, x_m, x_{m+1}, $$
where $x_0$ is the starting vertex $(y_0,0)$, $x_{m+1}$ is the ending vertex $(y_k,\ell_k)$, and for each $0\leq i\leq m$, $x_{i}$ and $x_{i+1}$ are the two endpoints of $\pi_i$.

\begin{definition} For $w\in \Sigma$, the base problem for $w$ is {\em nondegenerate} if in the enumeration of the special vertices along the basic walk for $w$
$$ x_0, x_1, \dots, x_m, x_{m+1}, $$
we have $x_i\neq x_{i+1}$ for all $0\leq i\leq m$. In this case we also say $w$ is {\em nondegenerate}; otherwise $w$ is {\em degenerate}.
\end{definition}

The following observation shows that, with regard to the existence of $\Pi$-solutions, it is without loss of generality to concentrate on nondegenerate words.

\begin{lemma} For each $w\in \Sigma$, there is a nondegenerate $w^*\in \Sigma$ such that $\langle a, w^{-1}aw\rangle\leftrightsquigarrow_\Pi\big\langle a, (w^*)^{-1}aw^*\big\rangle$.
\end{lemma}

\begin{proof} Let ${}^*\colon F(a,b)$ be the homomorphism given by $a^*=a$ and $b^*=b^{-1}ab$. Then by Proposition~\ref{prop:awbasic} (ii), $\langle a, w^{-1}aw\rangle\leftrightsquigarrow_\Pi\big\langle a, (w^*)^{-1}aw^*\big\rangle$. $w^*$ is a returning word of the form
$$ b^{-1}a^{\alpha_{2k+1}}ba^{\alpha_{2k}}\cdots a^{\alpha_2}b^{-1}a^{\alpha_1}b $$
where $\alpha_1, \dots, \alpha_{2k}\in \mathbb{Z}\setminus\{0\}$. 
Its level sequence is $(0,1,0,1,\dots, 0,1,0)$. Hence the basic walk consists of a concatenation of $\pi_1, \pi_3, \dots, \pi_{2k+1}$. Since the length of $\pi_{2i+1}$ for each $0\leq i\leq k$ is $|\alpha_{2i+1}|\neq 0$, we know that $w^*$ is nondegenerate.
\end{proof}

It is easy to see that if $w$ is wandering and degenerate, then $w\in \langle b\rangle$. From now on we tacitly assume that all  $w\in \Sigma$ we consider are nondegenerate.

We may classify the nondegenerate $\langle a, w^{-1}aw\rangle$-problems according to the behavior of their basic walks as follows.

\begin{definition} Let $w\in \Sigma$ be nondegenerate and let $\kappa\geq 1$ be an integer. 
\begin{enumerate}
\item[(1)] For each special vertex $v\in S_w$, the number of times $v$ is visited on the basic walk $W_w$ is called the {\em visiting frequency} of $v$ and is denoted by $\phi_w(v)$.
\item[(2)] We say that $w$ is {\em $\kappa$-weakly wandering} if the visiting frequency of the ending vertex of the basic walk $W_w$ is exactly $\kappa$. If $w$ is $\kappa$-weakly wandering, we say that an $\langle a, w^{-1}aw\rangle$-problem is {\em $\kappa$-weakly wandering}.
\item[(3)] $w$ is {\em weakly wandering} if it is $1$-weakly wandering. If $w$ is weakly wandering, we say that an $\langle a, w^{-1}aw\rangle$-problem is {\em weakly wandering}.
\end{enumerate}
\end{definition}

It is clear that if $w$ is wandering and nondegenerate then it is weakly wandering. Also note that $w$ is wandering if and only if $w^{-1}$ is wandering, but weakly wandering is not closed under taking inverses.

\subsection{Weakly wandering problems\label{subsec:weakwandering}} Throughout this subsection let $\Pi$ be a set of prime numbers. Our main theorem of this subsection is the following.

\begin{theorem}\label{thm:weakwandering} If $w\in \Sigma$ is weakly wandering, then $\langle a, w^{-1}aw\rangle$ is closed in the pro-$\Pi$ topology of $F(a, b)$. 
\end{theorem}

\begin{corollary} Let $w\in \Sigma$. Suppose either $w$ or $w^{-1}$ is weakly wandering. Then $\langle a, w^{-1}aw\rangle$ is closed in the pro-$\Pi$ topology of $F(a, b)$.
\end{corollary}

The rest of this subsection is devoted to a proof of Theorem~\ref{thm:weakwandering}.

Suppose $w\in \Sigma$ is weakly wandering. We use all notation introduced in the preceding section. In particular, let $W_w$ be the basic walk for $w$. Enumerate all special vertices visited along $W_w$ as $x_0, x_1, \dots, x_{m+1}$. Let $(Y_w, Q_w)$ be the base problem for $w$ and let $(Z_w, R_w)$, where $R_w=\{\tilde{a}, \tilde{b}\}$, be a $\Pi$-solution for $(Y_w, Q_w)$. 

Since $w$ is weakly wandering, we have $x_i\neq x_{i+1}$ for $0\leq i\leq m$ and $x_{m+1}\not\in\{x_0, x_1,\dots, x_m\}$.
If $m=0$, then $w$ is wandering and the theorem holds. Thus we assume $m>0$. Let $M>2|w|$ be a $\Pi$-number.

To obtain a solution $(Z, R)$ for an $\langle a, w^{-1}aw\rangle$-problem, we are going to apply the EP map {\tt Reroute} recursively to define a sequence of EP problems $(Z_j, R_j)$, $R_j=\{\tilde{a}_j, \tilde{b}_j\}$, for $0\leq j\leq m$. To begin, let $(Z_0, R_0)=(Z_w, R_w)$. Thus $\tilde{a}_0=\tilde{a}$ and $\tilde{b}_0=\tilde{b}$. For $j\geq 1$, we will define $Z_j$ to be of the form
$$ Z_0\times T_1\times \cdots \times T_j $$
where for each $1\leq\eta\leq j$, $T_\eta$ is either $\{0, \dots, M-1\}$ or $\{0\}$. We identify it with the set $Z_j\times \{0\}^{m-j}$, which is a subset of $Z_0\times \{0, \dots, M-1\}^m$, and which is in turn identified with a set of the form
$$ Z_0\times \{0, \dots, M-1\}^r $$
for some $1\leq r\leq m$ (by removing the coordinates whose corresponding spaces are $\{0\}$). 
Eventually, we let $(Z, R)=(Z_m, R_m)$. For each $1\leq j\leq m$, let $$\mbox{\rm proj}_j\colon Z_0\times \{0, \dots, M-1\}^m\to \{0, \dots, M-1\}$$ be the projection to the $j$-th coordinate. Let $$\mbox{\rm proj}_0\colon Z_0\times \{0, \dots, M-1\}^m\to Z_0$$ be the projection to the $0$-th coordinate. For $1\leq j\leq m$, let
$$ \mbox{\rm proj}_{[0,j]}\colon Z_0\times \{0, \dots, M-1\}^m\to Z_0\times \{0, \dots, M-1\}^j $$
be the map $(\mbox{\rm proj}_0, \dots, \mbox{\rm proj}_{j})$.

In order to define $(Z_j, R_j)$, we need to keep track of the special vertices and their copies in each $Z_j$. For this, let 
$$ S_0=\{x_0, x_1, \dots, x_m, x_{m+1}\}. $$
This is the set of special vertices of $(Z_0, R_0)$. Since $(Z_0, R_0)$ is a solution for $(Y_w, Q_w)$, we have $\tilde{a}_0(x)=x$ for all $x\in S_0$. We also let $S_0^a=S_0$. $S_0^a$ is the set of all $\tilde{a}_0$-fixed special vertices in $S_0$. 

Now let 
$$ (Z_1, R_1)=\mbox{\tt Reroute}(Z_0, R_0; M, \tilde{a}_0, x_1). $$
Thus $Z_1=Z_0\times T_1$ where $T_1=\{0, \dots, M-1\}$. We identify $Z_0$ with $Z_0\times\{0\}$, which is a subset of $Z_1$. 
Let 
$$ S_1=\{(x_j, i)\in Z_1\colon 0\leq j\leq m+1, 0\leq i\leq M-1\}\subseteq Z_1. $$
Then $S_1$ is the set of all copies of special vertices of $Z_0$ in $Z_1$; we refer to the elements of $S_1$ as the {\em special} vertices of $(Z_1, R_1)$. Let 
$$ S_1^a=\{z\in Z_1\colon \tilde{a}_1(z)=z\} $$
be the set of all $\tilde{a}_1$-fixed special vertices in $S_1$. Then we observe that
$$ S_1^a=S_1\setminus \{(x_1,i)\colon 0\leq i\leq M-1\}. $$

Let $z_0=(x_0,0)$ and $z_1=(x_1,0)$ be elements of $Z_1$. Then the path $\nu_0$ in $G(Z_0,R_0)$ from $x_0$ to $x_1$ becomes a path in $G(Z_1, R_1)$ from $z_0$ to $z_1$ under the identification of $Z_0$ with $Z_0\times\{0\}$. We continue to denote this path by $\nu_0$. Let 
\begin{equation}\label{eqn:z1*} z_1^*=(x_1,\alpha_{\lambda_1})=\tilde{a}_1^{\alpha_{\lambda_1}}(z_1) \end{equation}
and
\begin{equation}\label{eqn:z2} z_2=(x_2, \alpha_{\lambda_1})=\tilde{b}_1^{\beta_{\lambda_2}}\tilde{a}_1^{\alpha_{\lambda_2-1}}\cdots\tilde{a}_1^{\alpha_{\lambda_1+1}}\tilde{b}_1^{\beta_{\lambda_1+1}}(z_1^*). \end{equation}
Expression (\ref{eqn:z1*}) gives rise to a path $\mu_1$ in $G(Z_1, R_1)$ from $z_1$ to $z_1^*$, with edges labelled by $\tilde{a}_1$, and the direction of the edges determined by the sign of $\alpha_{\lambda_1}$. Similarly, expression (\ref{eqn:z2}) gives rise to a path in $G(Z_1, R_1)$ from $z_1^*$ to $z_2$. This path is entirely contained in $Z_0\times \{\alpha_{\lambda_1}\}\subseteq Z_1$, and is a copy of the path $\nu_1$ (viewed as a path from $x_1$ to $x_2$ in $Z_0\times\{0\}\subseteq Z_1$); we continue to denote it by $\nu_1$. Note that in $G(Z_1, R_1)$ the concatenation of the paths
$$ z_0\stackrel{\nu_0}{\longrightarrow} z_1\stackrel{\mu_1}{\longrightarrow}z_1^*\stackrel{\nu_1}{\longrightarrow}z_2 $$
is still a path in $G(Z_1, R_1)$ and contains no special vertices of $S_1^a$ except the endpoints $z_0=(x_0, 0)$ and $z_2=(x_2, \alpha_{\lambda_1})$. In particular $z_0\neq z_2$ and both of them are special vertices in $S_1^a$. This concatenated path will have copies in all $(Z_j, R_j)$ to be defined below, and thus we will denote these copies still by the same notation.

In general, assume we have defined $(Z_{j-1}, R_{j-1})$ and vertices $$z_0, z_1, z_1^*, z_2, z_2^*, \dots, z_{j-1}^*, z_{j}$$ for some $2\leq j\leq m$ and a concatenated path
\begin{equation}\label{eqn:cp} z_0\stackrel{\nu_0}{\longrightarrow} z_1\stackrel{\mu_1}{\longrightarrow}z_1^*\stackrel{\nu_1}{\longrightarrow}z_2\stackrel{\mu_2}{\longrightarrow}{z_2^*}\stackrel{\nu_2}{\longrightarrow}\cdots\stackrel{\mu_{j-1}}{\longrightarrow}z_{j-1}^*\stackrel{\nu_{j-1}}{\longrightarrow}z_{j} 
\end{equation}
in $G(Z_{j-1}, R_{j-1})$ such that for any $0\leq \eta\leq j$, 
$$z_{\eta}=(y_\eta, \varepsilon_{\eta,1},\dots, \varepsilon_{\eta, j-1})\in Z_{j-1}$$ for some $\varepsilon_{\eta,1}, \dots, \varepsilon_{\eta, j-1}\in \{0,\dots, M-1\}$. We have also defined the set of {\em special} vertices of $(Z_{j-1}, R_{j-1})$ to be 
$$ S_{j-1}=S_{j-2}^a\times \{0, \dots, M-1\} $$
and let
$$ S_{j-1}^a=\{z\in S_{j-1}\colon \tilde{a}_{j-1}(z)=z\} $$
be the set of all $\tilde{a}_{j-1}$-fixed special vertices of $S_{j-1}$. We assume further that the concatenated path in (\ref{eqn:cp}) contains no special vertices in $S_{j-1}^a$ except the endpoints $z_0$ and $z_j$. As for $z_0$, which is now represented as $(x_0, 0,\dots, 0)$ as an element of $Z_{j-1}$, we have that $z_0$ is indeed in $S_{j-1}^a$. Note in particular we assume $z_0\neq z_j$.

We consider two cases.

Case 1: $z_j\in S_{j-1}^a$. This case is similar to the construction of $(Z_1, R_1)$. In this case, we let
$$ (Z_j, R_j)=\mbox{\tt Reroute}(Z_{j-1}, R_{j-1}; M, \tilde{a}_{j-1}, z_j). $$
Let $S_j=S_{j-1}\times\{0, \dots, M-1\}$ and $S_j^a=\{z\in S_j\colon \tilde{a}_j(z)=z\}$. Then
$$ S_j^a=S_j\setminus \{(x_j, \varepsilon_{j,1}, \dots, \varepsilon_{j,j-1}, i)\colon 0\leq i\leq M-1\}. $$
Since $z_0\in S_{j-1}^a$ and $z_0\neq z_j$, we have that $z_0\in S_j$ and $z_0\not\in S_j\setminus S_j^a$, hence $z_0\in S_j^a$.

We define $z_j^*\in Z_j$ by
\begin{equation}\label{eqn:zj} z_j^*=\tilde{a}_j^{\alpha_{\lambda_j}}(z_j)=(x_j, \varepsilon_{j,1}, \cdots, \varepsilon_{j, j-1}, \alpha_{\lambda_j}) \end{equation}
and define $z_{j+1}\in Z_{j}$ by
\begin{equation}\label{eqn:zj*} z_{j+1}=(x_{j+1}, \varepsilon_{j,1}, \cdots, \varepsilon_{j,j-1}, \alpha_{\lambda_j})=\tilde{b}_j^{\beta_{\lambda_{j+1}}}\tilde{a}_j^{\alpha_{\lambda_{j+1}-1}}\cdots\tilde{a}_j^{\alpha_{\lambda_{j}+1}}\tilde{b}_j^{\beta_{\lambda_{j}+1}}(z_j^*). \end{equation}
Expression (\ref{eqn:zj}) gives a path in $G(Z_j, R_j)$ from $z_j$ to $z_j^*$, which we denote as $\mu_j$. Expression (\ref{eqn:zj*}) give a path in $G(Z_j, R_j)$ from $z_j^*$ to $z_{j+1}$. This path is entirely contained in 
$$Z_0\times  \{\varepsilon_{j,1}\}\times \cdots\times\{ \varepsilon_{j,j-1}\}\times\{\alpha_{\lambda_j}\}\subseteq Z_{j}$$
and is a copy of the path $\nu_j$ from $x_j$ to $x_{j+1}$ in $Z_0$. We continue to denote it by $\nu_j$. Thus we have defined the concatenated path
$$ z_j\stackrel{\mu_j}{\longrightarrow} z_j^*\stackrel{\nu_j}{\longrightarrow} z_{j+1} $$
in $G(Z_j, R_j)$. Together with the path in (\ref{eqn:cp}), this gives a concatenated path from $z_0$ to $z_{j+1}$ in $G(Z_j, R_j)$. 

Since $\alpha_{\lambda_j}\neq 0$, we have $z_0\neq z_{j+1}$. The other inductive properties are easy to verify. This finishes the inductive definitions in Case 1.

Case 2: $z_j\not\in S_{j-1}^a$. In this case, we let $(Z_j, R_j)=(Z_{j-1}, R_{j-1})$. Formally, we let $Z_j=Z_{j-1}\times\{0\}$, i.e., $T_j=\{0\}$. Let $S_j=S_{j-1}\times \{0\}$ and $S_j^a=\{z\in S_j\colon \tilde{a}_j(z)=z\}=S_{j-1}^a\times\{0\}$.

Since $\mbox{\rm proj}_0(z_j)=x_j\in S_0$, we must have $x_j\in \{x_0, \dots, x_{j-2}\}$. Moreover, we claim that $j$ is not the least positive index such that $x_j=x_0$. To see this claim, note by induction that if $x_1, \dots, x_{j-1}\neq x_0$ ($x_1\neq x_0$ is always true; therefore such $j$ must be at least $2$), then for any $0\leq \eta\leq j-1$ and for any $z\in S_\eta^a$, $\mbox{\rm proj}_0(z)\neq x_0$. Thus if $j$ was the least positive index such that $x_j=x_0$, we would have $z_j\in S_{j-1}^a$, contradicting our case assumption. Thus the claim is proved. With the claim, we have $x_j\in \{x_1, \dots, x_{j-2}\}$.

A similar argument gives that there is $1\leq \eta\leq j-2$ such that 
\begin{equation}\label{eqn:criterion}\mbox{\rm proj}_{[0,\eta-1]}(z_j)=\mbox{\rm proj}_{[0,\eta-1]}(z_\eta). \end{equation} Let $0\leq \eta\leq j-2$ be the least such. Then we must have $z_\eta\in S_{\eta-1}^a$ and $(Z_\eta, R_\eta)$ is constructed from $(Z_{\eta-1}, R_{\eta-1})$ as in Case 1. In particular, for any $z\in Z_\eta$ with $\mbox{\rm proj}_{[0,\eta-1]}(z)=\mbox{\rm proj}_{[0,\eta-1]}(z_\eta)$ and for any $\tau\in \mathbb{Z}$, we have 
$$\mbox{\rm proj}_\eta\big(\tilde{a}_\eta^\tau(z)\big)=\mbox{\rm proj}_\eta(z)+\tau\pmod{M}.$$It follows that for $z\in Z_j$ with $\mbox{\rm proj}_{[0,\eta-1]}(z)=\mbox{\rm proj}_{[0, \eta-1]}(z_\eta)$ and for any $\tau\in \mathbb{Z}$, 
$$\mbox{\rm proj}_\eta\big(\tilde{a}_j^\tau(z)\big)=\mbox{\rm proj}_\eta(z)+\tau\pmod{M}.$$

Now define $z_j^*\in Z_j$ by
\begin{equation}\label{eqn:zj*2} z_j^*=\tilde{a}_j^{\alpha_{\lambda_j}}(z_j)=(x_j, \varepsilon_{j,1}, \cdots, \varepsilon'_{j,\eta}, \cdots, \varepsilon_{j,j-1}, 0) \end{equation}
where
$$\varepsilon'_{j, \eta}=\varepsilon_{j, \eta}+\alpha_{\lambda_j}\pmod{M}. $$
Note that because of our choice of $M$, we have $\varepsilon'_{j, \eta}\neq |w|$. The expression (\ref{eqn:zj*2}) gives a path $\mu_j$ in $G(Z_j, R_j)$. We also define
\begin{equation}\label{eqn:zj+1} \begin{array}{rcl}
z_{j+1}&=&(x_{j+1}, \varepsilon_{j+1,1}, \dots,  \varepsilon_{j+1,\eta}, \cdots, \varepsilon_{j+1, j-1},\varepsilon_{j+1,j}) \\ \\
&=&(x_{j+1}, \varepsilon_{j,1}, \dots,  \varepsilon'_{j,\eta}, \cdots, \varepsilon_{j, j-1},0) \\ \\
&=&\tilde{b}_j^{\beta_{\lambda_{j+1}}}\tilde{a}_j^{\alpha_{\lambda_{j+1}-1}}\cdots\tilde{a}_j^{\alpha_{\lambda_{j}+1}}\tilde{b}_j^{\beta_{\lambda_{j}+1}}(z_j^*). 
\end{array}\end{equation}
Then the expression (\ref{eqn:zj+1}) gives a path $\nu_j$ from $z_j^*$ to $z_{j+1}$, which is entirely contained in the copy
$$ Z_0\times \{\varepsilon_{j,1}\}\times \cdots\times \{\varepsilon'_{j,\eta}\}\times  \cdots\times\{ \varepsilon_{j, j-1}\}\times \{0\}\subseteq Z_j $$
and is a copy of the path $\nu_j$ from $x_j$ to $x_{j+1}$ in $Z_0$. Concatenating the newly defined path
$$ z_j\stackrel{\mu_j}{\longrightarrow} z_j^*\stackrel{\nu_j}{\longrightarrow} z_{j+1} $$
with the path in (\ref{eqn:cp}), we obtain a walk in $G(Z_j, R_j)$ from $z_0$ to $z_{j+1}$. 

We verify that this walk is indeed a path, namely, the vertices on this walk are distinct. For this we define a sequence of subgraphs. Let $W_0=W_w\subseteq G(Z_m, R_m)$ be the basic walk, which is a graph-theoretic tree. For $1\leq \rho\leq j$, we inductively define $W_\rho\subseteq G(Z_\rho, R_\rho)$ by considering two cases.  If $G(Z_{\rho}, R_{\rho})$ is constructed from $G(Z_{\rho-1}, R_{\rho-1})$ as in Case 2, we let $W_\rho$ be 
$$ W_{\rho-1}\times\{0\}\subseteq G(Z_{\rho}, R_\rho), $$
which is now a subtree of $G(Z_{\rho}, R_{\rho})$. If $G(Z_{\rho}, R_{\rho})$ is constructed from $G(Z_{\rho-1}, R_{\rho-1})$ as in Case 1, we let $W_\rho$ be the union of following subgraphs of $G(Z_\rho, R_\rho)$:
\begin{enumerate}
\item[(a)] for each $0\leq i\leq M-1$ such that $i\neq |w|$, the copy of $W_{\rho-1}$ in $Z_{\rho-1}\times \{i\}$, namely $W_{\rho-1}\times\{i\}$;
\item[(b)] the path obtained from the cycle $\{z_\rho\}\times \{0, \dots, M-1\}$ by removing the vertex $(z_\rho, |w|)$.
\end{enumerate}
By induction, $W_{\rho-1}$ is a tree. It follows that $W_\rho$ is also a tree. This finishes the definition of trees $W_0, W_1, \dots, W_j$. 

Now by induction, it is easy to see that the path in (\ref{eqn:cp}) is entirely contained in $W_j$. Moreover, the newly defined path 
$$ z_j\stackrel{\mu_j}{\longrightarrow} z_j^*\stackrel{\nu_j}{\longrightarrow} z_{j+1} $$
is also entirely contained in $W_j$. Thus, to see that the concatenated walk is a path, it suffices to verify that on this walk, the predecessor of $z_j$ and the successor of $z_j$ are distinct. This is now clear, since they have distinct projections to the $\eta$-th coordinate. 

We have thus finished the definitions in Case 2 and verified that the indutive properties continue to hold. 

Finally, let $(Z, R)=(Z_m, R_m)$ and note that we have obtained a path in $G(Z, R)$ from $z_0$ to $z_{m+1}$. Since 
$x_{m+1}\not\in \{x_0, x_1, \dots, x_m\}$, we have that $z_{m+1}\in S_m^a$. This means $\tilde{a}_m(z_{m+1})=z_{m+1}$. Since $(Z, R)$ is obtained from $(Z_0, R_0)$ by repeated applications of the EP map {\tt Reroute}, we get that $(Z, R)$ is a $\Pi$-solution. Now the path from $z_0$ to $z_{m+1}$, together with the facts $\tilde{a}_m(z_0)=z_0$ and $\tilde{a}_m(z_{m+1})=z_{m+1}$, gives a subproblem of $(Z, R)$ which is an $\langle a, w^{-1}aw\rangle$-problem. We have thus shown that an $\langle a, w^{-1}aw\rangle$-problem has a $\Pi$-solution. The proof of Theorem~\ref{thm:weakwandering} is complete.

\section{EP Problems without Nilpotent-Solutions}

All EP problems we considered in the preceding section have $\Pi$-solutions for any $\Pi$. In particular, they have $p$-solutions for any prime $p$, and therefore they have nilpotent-solutions. The $\kappa$-weakly wandering problems no longer have this property when $\kappa\geq 2$. In this subsection we focus on $2$-weakly wandering problems. We completely characterize the $2$-weakly wandering problems which have nilpotent-solutions. For those which do not have nilpotent-solutions, we show that they have odd-solutions. To deal with general $\kappa$-weakly wandering problems, we introduce a new EP map with parameters.

\subsection{Wreath products of EP problems}

We first recall the notion of wreath product between permutation groups (see \cite[Section 1.10]{RobinsonBook}).

\begin{definition}\label{def:wr} Let $X, Y$ be finite sets and $\Gamma\leq \mbox{\rm Sym}(X)$, $\Lambda\leq \mbox{\rm Sym}(Y)$. The {\em wreath product} of $\Gamma$ and $\Lambda$, denoted $\Gamma\wr\Lambda$, as a permutation group on $Z=X\times Y$, is defined as follows. The set $\Gamma\wr\Lambda$ is $\Gamma^Y\times \Lambda$; given $g=((\gamma_i)_{i\in Y}, \lambda)\in \Gamma^Y\times\Lambda$ and $(x, y)\in Z=X\times Y$, let
$$ g(x, y)=(\gamma_{\lambda(y)}(x), \lambda(y)). $$
The product of two elements of $\Gamma\wr\Lambda$ is the composition of them as permutations on $Z$.
\end{definition}

\begin{lemma}\label{lem:basicwr} If $\Gamma$ and $\Lambda$ are $\Pi$-groups of permutations, then so is their wreath product $\Gamma\wr\Lambda$.
\end{lemma}

\begin{proof} If $|\Gamma|$ and $|\Lambda|$ are $\Pi$-numbers, then so is $|\Gamma|^{|Y|}\cdot|\Lambda|$.
\end{proof}

\begin{definition}\label{def:wreath} Let $(X, P)$, $P=\{a_X,b_X\}$, and $(Y, Q)$, $Q=\{a_Y, b_Y\}$, be EP(2) problems. Let $A, B\subseteq Y$ be disjoint subsets. Define the {\em wreath product} $$\mbox{\tt Wreath}(X, P; Y, Q; A, B)$$ to be the EP(2) problem $(Z, R)$, where
$$ Z= X\times Y $$
and
$R=\{\tilde{a}, \tilde{b}\}$, where
$$ \tilde{a}(x, y)=\left\{\begin{array}{ll} (a_X(x), a_Y(y)), & \mbox{ if $y\in A$, } \\ (b_X(x), a_Y(y)), & \mbox{ if $y\in B$,} \\ (x, a_Y(y)), & \mbox{ otherwise.} \end{array}\right. $$
and
$$ \tilde{b}(x,y)=(x, b_Y(y)). $$
\end{definition}

\begin{proposition}\label{prop:wrps} Let $(X, P)$, $P=\{a_X, b_X\}$, and $(Y, Q)$, $Q=\{a_Y, b_Y\}$, be $\Pi$-solutions. Let $A, B\subseteq Y$ be disjoint subsets. Then $\mbox{\tt Wreath}(X, P; Y, Q; A, B)$ is a $\Pi$-solution.
\end{proposition}

\begin{proof} Using the notation of Definition~\ref{def:wreath}, it is easy to verify that if $a_X, b_X\in \mbox{\rm Sym}(X)$ and $a_Y, b_Y\in \mbox{\rm Sym}(Y)$, then $\tilde{a}, \tilde{b}\in \mbox{\rm Sym}(X\times Y)$. Thus if $(X, P)$ and $(Y, Q)$ are solutions, then so is the wreath product. 

Suppose the solvent of $(X, P)$ is $\Gamma$ and the solvent of $(Y, Q)$ is $\Lambda$. Note that
$$ \tilde{a}=\big((\gamma_y)_{y\in Y}, \lambda\big)\in \Gamma\wr\Lambda, $$
where $\lambda=a_Y$ and 
$$ \gamma_y=\left\{\begin{array}{ll} a_X, & \mbox{ if $y\in a_Y(A)$, } \\ b_X, & \mbox{ if $y\in a_Y(B)$,} \\ \mbox{\rm id}, & \mbox{ otherwise.} \end{array}\right. $$
Similarly, $\tilde{b}\in \Gamma\wr\Lambda$. Hence the solvent of the wreath product is $\langle \tilde{a}, \tilde{b}\rangle$, which is a subgroup of $\Gamma\wr\Lambda$. By Lemma~\ref{lem:basicwr}, $\langle \tilde{a}, \tilde{b}\rangle\leq \mbox{\rm Sym}(X\times Y)$ is a $\Pi$-group. Hence $\mbox{\tt Wreath}(X, P; Y, Q; A, B)$ is a $\Pi$-solution.
\end{proof}

\begin{definition} Let $\Sigma(a, b)$ denote the set of all irreducible words $w\in F(a, b)$ which starts with $a$ and ends with $b$. Similarly define $\Sigma(a^{-1}, b)$, $\Sigma(a, b^{-1})$ and $\Sigma(a^{-1}, b^{-1})$. Let 
$$ \Sigma^a=\Sigma(a, b)\cup \Sigma(a^{-1}, b)\cup \Sigma(a, b^{-1})\cup \Sigma(a^{-1}, b^{-1}). $$
\end{definition}

Recall from Definition~\ref{def:*} that for $u\in \Sigma$, $\phi_u\colon F(a, b)\to F(a, b)$ is the homomorphism given by $\phi_u(a)=a$ and $\phi_u(b)=u$. Also ${}^*=\phi_{b^{-1}ab}$.

\begin{proposition}\label{prop:stationary} For any $w\in \Sigma^a$, $\langle a, w^{-1}bw\rangle$ is dense in the pro-nilpotent topology of $F(a, b)$. Consequently, for any $u\in \Sigma$, an $\big\langle a, \phi_u\big((w^{-1}bw)^*\big)\big\rangle$-problem does not have nilpotent-solutions. In particular, $\big\langle a, \phi_u\big((w^{-1}bw)^*\big)\big\rangle$ is not closed in the pro-nilpotent topology of $F(a, b)$. 
\end{proposition}

\begin{proof} For any prime $p$, let $\sigma_p\colon F(a, b)\to (\mathbb{Z}/p\mathbb{Z})^2$ be the group homomorphism defined in the statement of Theorem~\ref{thm:MSW}. Then $\sigma_p(a)=(1,0)$ and $\sigma_p(w^{-1}bw)=(0,1)$. It follows that $\sigma_p^{-1}\sigma_p(\langle a, w^{-1}bw\rangle)=F(a, b)$. Now it follows from Theorems~\ref{thm:MSW2} that $\langle a, w^{-1}bw\rangle$ is dense in the pro-nilpotent topology of $F(a, b)$, and from Theorem \ref{thm:nonil} we know that an $\langle a, w^{-1}bw\rangle$-problem does not have any nilpotent-solutions.

Note that the homomorphism ${}^*$ is a special case of the homomorphism $\phi_v$ for $v=b^{-1}ab$. By applying Proposition~\ref{prop:awbasic} (iii) twice, we get that if an $\big\langle a, \phi_u\big((w^{-1}bw)^*\big)\big\rangle$-problem has a nilpotent-solution, then so does an $\langle a, w^{-1}bw\rangle$-problem. Hence an $\big\langle a, \phi_u\big((w^{-1}bw)^*\big)\big\rangle$-problem does not have nilpotent-solutions. 

The last clause of the proposition follows from Proposition~\ref{prop:pronil}.
\end{proof}

The problem $P_1$ in Subsection~\ref{subsec:case} is an example of an $\langle a, w^{-1}bw\rangle$-problem, where $w=ab^{-1}$. The problem $P_2$ in that subsection is the corresponding $\langle a, (w^{-1}bw)^*\rangle$-problem.

\begin{proposition}\label{prop:wrsolution} Let $w\in \Sigma^a$ and let $(X, P)$ be an $\langle a, w^{-1}bw\rangle$-problem. Let $u\in \Sigma$ and let $(Y, Q)$ be an $\langle a, u^{-1}au\rangle$-problem. Let $y_0$ be the base point of the computation of $\pi_1(Y, Q)$. Then 
$$\pi_1\big(\mbox{\tt Wreath}(X,P; Y, Q; \{y_0\}, \{u(y_0)\})\big)=
\big\langle a, \phi_u\big((w^{-1}bw)^*\big)\big\rangle. $$
\end{proposition}

\begin{proof} Let $(Z, R)=\mbox{\tt Wreath}(X,P; Y, Q; \{y_0\}, \{u(y_0)\})$. Let $y_0, y_1, \dots, y_n=u(y_0)$ enumerate all the vertices in $G(Y, Q)$ along the path from $y_0$ to $y_n$ given by the expression $y_n=u(y_0)$. Let $x_0, x_1, \dots, x_m=w(x_0)$ enumerate all the vertices in $G(X, P)$ along the path from $x_0$ to $x_m$ given by the expression $x_m=w(x_0)$. 

From the definition of $(Z, R)$ we get the following facts:
\begin{enumerate}
\item[(i)] $\tilde{a}(x_0, y_0)=(x_0, y_0)$ and $\tilde{a}(x_m, y_n)=(x_m, y_n)$, and these are the only vertices $z\in Z$ satisfying $\tilde{a}(z)=z$;
\item[(ii)] $G(Z, R)$ restricted to $X\times \{y_0\}$, denoted $G^{y_0}$, is a copy of $G(X, P)$ restricted to only edges labelled by $a_X$; Moreover, every edge in $G^{y_0}$ is labelled by $\tilde{a}$;
\item[(iii)] $G(Z, R)$ restricted to $X\times \{y_n\}$, denoted by $G^{y_n}$, is a copy of $G(X, P)$ restricted to only edges labelled by $b_X$; Moreover, every edge in $G^{y_n}$ is labelled by $\tilde{a}$;
\item[(iv)] for any $x\in X$, 
$$ (x, y_0), (x, y_1), \dots, (x, y_n) $$
is a path in $(Z, R)$ from $(x, y_0)$ to $(x, y_n)$ labelled by $\tilde{u}$;
\item[(v)] for any $x\in X$, $G(Z, R)$ restricted to $\{x\}\times (Y\setminus\{y_0, y_n\})$, denoted by $G_x$, is a copy of $G(Y, Q)$ restricted to $Y\setminus \{y_0, y_n\}$; Moreover, there are no edges between $G_x$ and $G(Z, R)\setminus G_x$ except $\{(x, y_0), (x, y_1)\}$ and $\{(x, y_{n-1}), (x, y_n)\}$.
\end{enumerate}
It follows from (v) that the computation of $\pi_1(Z, R)$ does not depend on $G(Y, Q)\setminus\{y_0, y_1, \dots, y_n\}$, so we may assume without loss of generality that $Y=\{y_0, y_1, \dots, y_n\}$. Now from (i)--(iv), we obtain a path in $G(Z, R)$ restricted to $\{x_0, x_1,\dots, x_m\}\times Y$ from $(x_0, y_0)$ to $(x_m, y_n)$ labelled by 
$$\phi_u\big((w^{-1}bw)^*\big)$$ as follows.

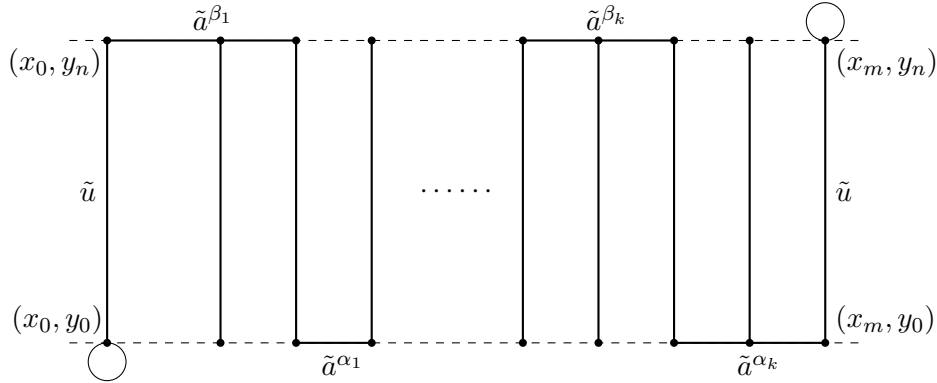
\begin{figure}[htbp]
\centering
\begin{tikzpicture}

\draw[dashed] (0,0) -- (10.5,0);
\draw[dashed] (0,4) -- (10.5,4);


\node[right] at (-0.9,0.3) {$(x_0, y_0)$};
\node[right] at (0,2) {$\tilde{u}$};
\node[right] at (10, 2) {$\tilde{u}$};
\node[right] at (-0.95,3.7) {$(x_0,y_n)$};
\node[right] at (1.5,4.3) {$\tilde{a}^{\beta_1}$};
\node[right] at (3.2, -0.3) {$\tilde{a}^{\alpha_1}$};
\node[right] at (4.5, 2) {$\cdots\cdots$};
\node[right] at (10, 0.3) {$(x_m, y_0)$};
\node[right] at (10, 3.7) {$(x_m, y_n)$};
\node[right] at (6.7, 4.3) {$\tilde{a}^{\beta_k}$};
\node[right] at (8.7, -0.3) {$\tilde{a}^{\alpha_k}$};

\fill (0.5,0) circle (1.5pt);
\fill (0.5,4) circle (1.5pt);
\fill (3,4) circle (1.5pt);
\fill (3,0) circle (1.5pt);
\fill (4,0) circle (1.5pt);
\fill (4,4) circle (1.5pt);
\fill (6,0) circle (1.5pt);
\fill (6,4) circle (1.5pt);
\fill (8,4) circle (1.5pt);
\fill (8,0) circle (1.5pt);
\fill (10,0) circle (1.5pt);
\fill (10,4) circle (1.5pt);
\fill (2,0) circle (1.5pt);
\fill (2,4) circle (1.5pt);
\fill (7,4) circle (1.5pt);
\fill (7,0) circle (1.5pt);
\fill (9,0) circle (1.5pt);
\fill (9,4) circle (1.5pt);

\draw[thick] (0.5,0) -- (0.5,4) -- (3,4) -- (3,0) --(4,0) --(4,4);
\draw[thick] (6,0) -- (6,4) -- (8,4) -- (8,0) -- (10,0) -- (10,4);
\draw[thick] (2,0) -- (2,4);
\draw[thick] (7,0) -- (7,4);
\draw[thick] (9,0) -- (9,4);

\draw (0.5,-0.25) circle (0.25);
\draw (10,4.25) circle (0.25);


\end{tikzpicture}
\caption{\label{fig:wreath} The wreath product of an $\langle a, w^{-1}bw\rangle$-problem with an $\langle a, u^{-1}au\rangle$-problem.}
\end{figure}

Assume
$$ w=a^{\alpha_k}b^{\beta_k}\cdots a^{\alpha_1}b^{\beta_1}. $$
Then 
$$ \phi_u(w^*)=a^{\alpha_k} u^{-1} a^{\beta_k}u\cdots a^{\alpha_1} u^{-1}a^{\beta_1}u. $$
Since 
$$ \phi_u\big((w^{-1}bw)^*\big)=\big(u\phi_u(w^*)\big)^{-1}a\big(u\phi_u(w^*)\big), $$
the path from $(x_0, y_0)$ labelled successively by
$$ \tilde{u}, \tilde{a}^{\beta_1}, \tilde{u}^{-1}, \tilde{a}^{\alpha_1}, \tilde{u},\dots, \tilde{a}^{\beta_k}, \tilde{u}^{-1}, \tilde{a}^{\alpha_k}, \tilde{u} $$
is a subgraph of $G(Z, R)$ restricted to $\{x_0, x_1, \dots, x_m\}\times Y$ with endpoints $(x_0, y_0)$ and $(x_m, y_n)$. In Figure~\ref{fig:wreath} we illustrate the graph $G(Z, R)$ restricted to $\{x_0, x_1,\dots, x_m\}\times Y$, in particular this path from $(x_0, y_0)$ to $(x_m, y_n)$.

We check that there are no other cycles in $G(Z, R)$, and thus $\pi_1(Z, R)=\big\langle a, \phi_u\big((w^{-1}bw)^*\big)\big\rangle$. By (ii)--(v), every cycle in $G(Z, R)$ other than the demonstrated loops at $(x_0, y_0)$ and $(x_m, y_n)$ must be of the form 
$$ (x, y_0)\stackrel{\mu}{\longrightarrow}(x',y_0)\stackrel{\nu}{\longrightarrow}(x', y_n)\stackrel{\mu'}{\longrightarrow}(x, y_n)\stackrel{\nu'}{\longrightarrow}(x, y_0), $$
where $\mu$ is a path entirely contained in $X\times\{y_0\}$, $\mu'$ is a path entirely contained in $X\times \{y_n\}$, $\nu$ is the path 
$$ (x', y_0), (x', y_1), \dots, (x', y_n), $$
and $\nu'$ is the path
$$ (x, y_n), (x, y_{n-1}, \dots, (x, y_1), (x, y_0). $$
This means that there is a path in $G(X,P)$ from $x$ to $x'$ labelled entirely by $b_X$, and a path in $G(X,P)$ from $x'$ to $x$ labelled entirely by $a_X$. In particular, they give a cycle in $G(X,P)$ which is disjoint from the loops at $x_0$ and $x_m$. This is a contradiction to the assumption that $\pi_1(X, P)$ is generated by two elements.
\end{proof}

\begin{theorem}\label{thm:stationary} Let $w\in \Sigma^a$ and $u\in \Sigma$. Let $\Pi$ be a set of prime numbers. Suppose an $\langle a, u^{-1}au\rangle$-problem has a $\Pi$-solution. Then
$$\langle a, w^{-1}bw\rangle \leftrightsquigarrow_{\Pi} \big\langle a, \phi_u\big((w^{-1}bw)^*\big)\big\rangle.$$
\end{theorem}

\begin{proof} Let $(X, P)$ be an $\langle a, w^{-1}bw\rangle$-problem and $(Y, Q)$ be an $\langle a, u^{-1}au\rangle$-problem. Let $(X', P')$ be a $\Pi$-solution for $(X, P)$ and $(Y', Q')$ be a $\Pi$-solution for $(Y, Q)$. Let $y_0$ be the base point for the computation of $\pi_1(Y, Q)$. Then $y_0, u(y_0)\in Y'$. Let $(Z, R)=\mbox{\tt Wreath}(X', P'; Y', Q'; \{y_0\}, \{u(y_0)\})$ and $(Z', R')=\mbox{\tt Wreath}(X', P'; Y', Q'; \{y_0\}, \{u(y_0)\})$. Then $(Z, R)$ is subproblem of $(Z', R')$. By Proposition~\ref{prop:wrsolution}, $(Z, R)$ is an $\big\langle a, \phi_u\big((w^{-1}bw)^*\big)\big\rangle$-problem. By Proposition~\ref{prop:wrps}, $(Z', R')$ is a $\Pi$-solution. Thus we have
$$ \langle a, w^{-1}bw\rangle \rightsquigarrow_{\Pi} \big\langle a, \phi_u\big((w^{-1}bw)^*\big)\big\rangle.$$
The converse follows by applying Proposition~\ref{prop:awbasic} (iii) twice to $\phi_u\big((w^{-1}bw)^*\big)$.
\end{proof}

\subsection{$2$-weakly wandering problems\label{subsec:2ww}}

By our Theorem~\ref{thm:weakwandering}, every $1$-weakly wandering problem has a $\Pi$-solution for any $\Pi$, and in particular they have nilpotent solutions. In this subsection we consider $2$-weakly wandering problems. 

\begin{definition} A word $w\in \Sigma$ is {\em stationary} if there exist $u\in \Sigma$ and $v\in \Sigma^a$ such that $w=\phi_u(bv^*)=u\phi_u(v^*)$. If $w$ is stationary, then we say that an $\langle a, w^{-1}aw\rangle$-problem is {\em stationary}.
\end{definition}

By Theorem~\ref{thm:stationary}, every stationary problem has no nilpotent-solutions. Our main theorem of this subsection is the following.

\begin{theorem}\label{thm:2ww} For a  $2$-weakly wandering $w\in\Sigma$, the following are equivalent:
\begin{enumerate}
\item[(i)] $w$ is not stationary.
\item[(ii)] For any $\Pi$, an $\langle a, w^{-1}aw\rangle$-problem has a $\Pi$-solution.
\item[(iii)] For some prime $p$, an $\langle a, w^{-1}aw\rangle$-problem has a $p$-solution.
\item[(iv)] An $\langle a, w^{-1}aw\rangle$-problem has a nilpotent-solution.
\item[(v)] For any $\Pi$, $\langle a, w^{-1}aw\rangle$ is closed in the pro-$\Pi$ topology of $F(a, b)$.
\item[(vi)] For some prime $p$, $\langle a, w^{-1}aw\rangle$ is closed in the pro-$p$ topology of $F(a, b)$.
\end{enumerate}
\end{theorem}

\begin{proof} The equivalence of (ii) with (v) and (iii) with (vi) are given by Theorem~\ref{thm:solutionprotopology}. The implication (ii)$\Rightarrow$(iii)  is obvious. The implication (iii)$\Rightarrow$(iv) is given by Theorem~\ref{thm:pnil}. The implication (iv)$\Rightarrow$(i) is given by Proposition~\ref{prop:stationary}. All of these implications do not use the $2$-weakly wandering assumptions. It remains to show (i)$\Rightarrow$(ii) for a $2$-weakly wandering $w$.

Let $w\in \Sigma$ be a $2$-weakly wandering word of the form (\ref{eqn:w}) and let
$$ x_0, x_1, \dots, x_m, x_{m+1} $$
be the enumeration of special vertices along the basic walk for $w$ in $G(X_w, P_w)$, where $m\geq 1$. Let $0\leq m^*<m$ be the unique index such that $x_{m^*}=x_{m+1}$. 

We construct $\Pi$-solutions using either the same construction as in the proof of Theorem~\ref{thm:weakwandering} or a modification of that construction.  For this we first recall some details of that proof and perform an analysis of its parameters. 

Recall from that proof that we started with noting that $G(X_w, P_w)$ is a tree and obtained a $\Pi$-solution $(X', P')$ for $(X_w, P_w)$. Then we defined a $\Pi$-solution $(Y, Q)$ by a construction similar to the {\tt Multiverse} construction. In $G(Y, Q)$ we obtained a walk whose special vertices are enumerated as
$$ y_0, y_1, \dots, y_{m+1}. $$
Next, we repeatedly used {\tt Reroute} to inductively define a sequence of $\Pi$-solutions $(Y_j, Q_j)$, with $(Y_0, Q_0)=(Y, Q)$, for $0\leq j\leq m$. Each $Y_j$ is of the form
$$ Y\times T_1\times\cdots \times T_j$$
where for each $1\leq \eta\leq j$, $T_\eta$ is either $\{0, \dots, M-1\}$ or $\{0\}$, and it is further identified with 
$$ Y\times T_1\times \cdots\times T_j\times\{0\}\times\cdots \times\{0\} $$
which is a subspace of $Y_m$. Along with the construction of $(Y_j, Q_j)$, we also constructed a path
\begin{equation}\label{eqn:zm} z_0\stackrel{\nu_0}{\longrightarrow} z_1\stackrel{\mu_0}{\longrightarrow} z_1^*\stackrel{\nu_1}{\longrightarrow} z_2\stackrel{\mu_2}{\longrightarrow}\cdots\stackrel{\nu_{m-1}}{\longrightarrow}z_m\stackrel{\mu_m}{\longrightarrow}z_m^*\stackrel{\nu_m}{\longrightarrow}z_{m+1}. 
\end{equation}
We eventually let $(Z, R)=(Y_m, Q_m)$. Let $R=\{\tilde{a}, \tilde{b}\}$ for simplicity of notation. We have that $\tilde{a}(z_0)=z_0$ and, under the $1$-weakly wandering assumption, we also have $\tilde{a}(z_{m+1})=z_{m+1}$. Therefore the path given by (\ref{eqn:zm}) witnesses that $(Z, R)$ is a $\Pi$-solution for an $\langle a, w^{-1}aw\rangle$-problem.

At each stage of the construction, we maintained sets $S_j$ and $S_j^a$. $S_j=S^a_{j-1}\times T_j$ is the set of all special vertices inherited from the previous stage, and $S_j^a$ is the set of all $\tilde{a}$-fixed special vertices of $(Y_j, Q_j)$. If $z_{j+1}\in S_j^a$, then {\tt Reroute} is applied to obtain a new EP problem $(Y_{j+1}, Q_{j+1})$, and in this process some elements of $S_j^a$ (specifically, $\{z_{j+1}\}\times \{0, \dots, M-1\}$) lost their status as $\tilde{a}$-fixed special vertices. If $z_{j+1}\not\in S_j^a$, then the path goes on in $(Y_j, Q_j)$ and the EP problem $(Y_{j+1}, Q_{j+1})$ is essentially the same as $(Y_j, Q_j)$, and therefore not new. Throughout this whole process, each element $z_j$ is represented by a stable sequence of coordinates
\begin{equation}\label{eqn:zjcoordinates} (y_j, \varepsilon_{j,1}, \dots, \varepsilon_{j, m}) \end{equation}
where for all $j\leq \eta\leq m$, $\varepsilon_{j, m}=0$. In fact, for $1\leq \eta\leq j-1$, each $\varepsilon_{j, \eta}$ is the sum of some elements of the sequence
$$ \alpha_{\lambda_1}, \dots, \alpha_{\lambda_m}. $$
The coordinates of $z_j$ in (\ref{eqn:zjcoordinates}) can be computed as in (\ref{eqn:zj*}) and (\ref{eqn:zj+1}). The condition to apply (\ref{eqn:zj+1}) is given by (\ref{eqn:criterion}), which is based purely on the coordinates of $z_\eta$ for $1\leq \eta<j$.

With this analysis, we now consider a $2$-weakly wandering $w$. We consider two cases.

Case 1: $x_{m+1}=x_0$, i.e., $m^*=0$. In this case we claim that the above construction works, and therefore an $\langle a, w^{-1}aw\rangle$-problem has a $\Pi$-solution for any $\Pi$. 

Note that we have $z_0\in S_m^a$ as before. We only need to verify that $z_{m+1}\in S_m^a$. For this, note that our assumption implies $x_0=x_{m+1}\not\in\{x_1, \dots, x_m\}$, and therefore $y_0=y_{m+1}\not\in \{y_1, \dots, y_m\}$. Since $\mbox{\rm proj}_0(z_{m+1})=y_{m+1}\neq y_j=\mbox{\rm proj}_0(z_j)$ for any $1\leq j\leq m$, we easily see that $z_{m+1}\in S_m^a$. This means that $\tilde{a}(z_{m+1})=z_{m+1}$, and thus the path given by (\ref{eqn:zm}) is an $\langle a, w^{-1}aw\rangle$-problem.

Case 2: $0<m^*<m$. In this case we modify the above construction as follows. 

Consider the part of the basic walk in $G(X', P')$ from $x_{m^*}$ to $x_{m+1}$. The corresponding walk $\mu$ in $G(Y, Q)$ visits the following special vertices in sequence
$$ y_{m^*}=y_{m+1}, \  y_{m^*+1}, \  \dots\ , \ y_m, \ y_{m+1}. $$
Note that $y_{m^*}=y_{m+1}\not\in \{y_{m^*+1},\dots, y_m\}$. Since $\mu$ is contained entirely in a subtree of $G(Y, Q)$, there exists a vertex in $\{y_{m^*+1},\dots, y_m\}$ which is visited exactly once by $\mu$. 

Now we assume
\begin{equation}\label{NS} \tag{NS} \begin{array}{l} \mbox{there is a vertex $y'\in\{y_{m^*+1}, \dots, y_m\}$ which is visited} \\
\mbox{exactly once by $\mu$, and $y'\neq y_0$.}\end{array}
\end{equation}
Under this assumption we construct a $\Pi$-solution. Suppose $m'$ is the unique index such that $y'=y_{m'}$. Let $(Y, Q)$ be as defined in the above proof. Instead of letting $(Y_0, Q_0)=(Y, Q)$, we let
$$ (Y_0, Q_0)=\mbox{\tt Reroute}(Y, Q; M, \tilde{a}, y'). $$
Then $Y_0=Y\times T_0=Y\times \{0, \dots, M-1\}$. Next we inductively define $(Y_j, Q_j)$ for $1\leq j\leq m$ as before. Eventually let $(Z, R)=(Y_m, Q_m)$. Let
$$ \mbox{\rm proj}_{-1}\colon Y\times T_0\times T_1\times \cdots \times T_m\to Y $$
and $$\mbox{\rm proj}_0\colon Y\times T_0\times T_1\times \cdots \times T_m\to T_0 $$
be the projection maps. We also define the path of the form (\ref{eqn:zm}) from $z_0=(y_0, 0, \dots, 0)$ to $z_{m+1}$ in $(Z, R)$ as before. Note that each $z_j$ is now represented by a sequence of coordinates of the form
\begin{equation}\label{eqn:zjcoordinatesnew} (y_j, \varepsilon^*_{j,0}, \varepsilon^*_{j, 1},\dots, \varepsilon^*_{j, m}).
\end{equation}
Suppose 
$$  z_{m^*}=(y_{m+1}, \varepsilon^*_{m^*,0}, \varepsilon^*_{m^*, 1},\dots, \varepsilon^*_{m^*, m}). $$
Then
$$ z_{m+1}=(y_{m+1}, \varepsilon^*_{m^*,0}+\alpha_{\lambda_{m'}}, \varepsilon^*_{m, 1}, \dots, \varepsilon^*_{m, m}) $$
by our assumption (NS). Since $z_{m^*}$ and $z_{m+1}$ are the only special vertices $z'$ on the basic walk with $\mbox{\rm proj}_{-1}(z')=y_{m+1}$, from 
$$ \mbox{\rm proj}_0(z_{m^*})=\varepsilon^*_{m^*, 0}\neq \varepsilon^*_{m^*,0}+\alpha_{\lambda_{m'}}=\mbox{\rm proj}_0(z_{m+1}), $$
we know that $z_{m+1}\in S_m^a$. This again shows that the path given by (\ref{eqn:zm}) is an $\langle a, w^{-1}aw\rangle$-problem, and $(Z, R)$ is a $\Pi$-solution for it.

Next we assume that (NS) fails. Then the unique 
$$y'=y_{m'}\in\{y_{m^*+1}, \dots, y_m\}$$ visted exactly once by $\mu$ is $y_0$. Now the walk in $G(Y, Q)$ corresponding to the basic walk, which we denote by $\nu$, can be divided into three parts: $\nu_1$ from $y_0$ to $y_{m^*}=y_{m+1}$, $\nu_2$ from $y_{m^*}$ to $y'=y_0=y_{m'}$, and $\nu_3$ from $y'=y_0=y_{m'}$ to $y_{m+1}$. We have that $y_0$ does not occur in $\nu_2$ and $\nu_3$ other than as an endpoint, and $y_{m+1}$ does not occur in any of $\nu_1$, $\nu_2$ and $\nu_3$ other than as an endpoint. Each of these walks corresponds to a subword of $w$, which we denote by $u_1$, $u_2$ and $u_3$, respectively. Note that 
$$ w=u_3\, a^{\alpha_{\lambda_{m'}}}\, u_2\, a^{\alpha_{\lambda_{m^*}}}\,  u_1. $$
If $u_1=u_3=u_2^{-1}$, then $w=\phi_{u_1}(bv^*)$ for 
$$ v=a^{\alpha_{\lambda_{m'}}} b^{\alpha_{\lambda_{m^*}}}, $$
and hence $w$ is stationary. We show below that if $u_1\neq u_2^{-1}$ or $u_3\neq u_2^{-1}$, then an $\langle a, w^{-1}aw\rangle$-problem has a $\Pi$-solution for any $\Pi$. This will complete the proof of (i)$\Rightarrow$(ii). 

The desired $\Pi$-solution will be constructed in 5 steps. We will denote the result of each step of the construction by $(Z_i, R_i)$ for $i=1,\dots, 5$. To construct $(Z_1, R_1)$, we note that $u_1$ is weakly wandering. Starting from the above $(Y, Q)$, we follow the construction in the proof of Theorem~\ref{thm:weakwandering} for $u_1$. Let $(Z_1, R_1)$ be the result of this construction. In other words, 
$$ Z_1=Y\times T_1\times\cdots \times T_{m^*-1}. $$
In $G(Z_1, R_1)$, we have a path $\rho_1$ from $z_0$ to $z_{m^*}$ so that the labels of the edges along the path give exactly $u_1$. We also have a subtree $W_1$ of $G(Z_1, R_1)$ as defined in the proof of Theorem~\ref{thm:weakwandering}. Note that $\rho_1$ is a subgraph of $W_1$, and the projection of $W_1$ to $Y$ is exactly the subgraph of $G(Y,Q)$ traversed by $\nu_1$. 

To construction $(Z_2, R_2)$, we use $(Z_1, R_1)$ as the starting problem in place of $(Y, Q)$, use $z_{m^*}$ as the starting vertex, and follow the construction of the proof of Theorem~\ref{thm:weakwandering} for $u_2$. Let $(Z_2, R_2)$ be the result of this construction. Note that we have
$$ Z_2=Z_1\times T_{m^*+1}\times \cdots \times T_{m'-1}. $$
But we may still speak of $\mbox{\rm proj}_i(Z_2)$ for $i=0,1,\dots, m^*-1$; in particular, $\mbox{\rm proj}_0*(Z_2)=Y$. We observe that if $u_2=u_1^{-1}$, then in this construction we have $T_j=\{0\}$ for all $m^*+1\leq j\leq m'-1$, and the path from $z_{m^*}$ to $z_{m'}$ is exactly the inverse of the path $\rho_1$; in particular $z_{m'}=z_0$ (again, we identify $z_0\in Z_1$ with $(z_0, 0, \dots, 0)\in Z_2$). If, on the other hand, $u_2\neq u_1^{-1}$, then we obtain in $G(Z_2, R_2)$ a path $\rho_2$ from $z_{m^*}$ to $z_{m'}$, where $\rho_2$ is not the inverse of $\rho_1$. Moreover, $z_{m'}\neq z_0$ but $\mbox{\rm proj}_0(z_{m'})=y_0$. It follows that in $G(Z_2, R_2)$, we have $\tilde{a}(z_0)=z_0$, $\tilde{a}(z_{m^*})=z_{m^*}$, and $\tilde{a}(z_{m'})=z_{m'}$.

Next, to construct $(Z_3, R_3)$, we use $(Z_2, R_2)$ as the starting problem in place of $(Y, Q)$, use $z_{m'}$ as the starting vertex, and follow the construction of the proof of Theorem~\ref{thm:weakwandering} for $u_3$. Let $(Z_3, R_3)$ be the result of this construction. Then 
$$ Z_3=Z_2\times T_{m'+1}\times\cdots\times T_{m}. $$
We obtain a path $\rho_3$ from $z_{m'}$ to $z_{m+1}$ in $G(Z_3, R_3)$. Note that in $G(Z_3, R_3)$, all of $z_0$, $z_{m^*}$, $z_{m'}$ and $z_{m+1}$ are $\tilde{a}$-fixed points. If $u_3=u_2^{-1}$, then $\rho_3$ is exactly the inverse of $\rho_2$, and in particular $z_{m+1}=z_{m^*}$. Otherwise, if $u_3\neq u_2^{-1}$, then $\rho_3$ is not exactly the inverse of $\rho_2$, and $z_{m+1}\neq z_{m^*}$.

Now suppose $u_3\neq u_2^{-1}$. Then we let
$$ (Z_4, R_4)=\mbox{\tt Reroute}(Z_3, R_3; M, \tilde{a}, z_{m^*}) $$
and 
$$ (Z_5, R_5)=\mbox{\tt Reroute}(Z_4, R_4; M, \tilde{a}, z_{m'}). $$
Let $\hat{z}_0=(z_0, 0,0)\in Z_5$ and $\hat{z}_{m^*}=(z_{m^*}, 0, 0)\in Z_5$. Then the path $\rho_1$ in $G(Z_3, R_3)$ gives a path $\hat{\rho}_1$ from $\hat{z}_0$ to $\hat{z}_{m^*}$ in $G(Z_5, R_5)$. Let 
$$ \hat{z}_{m^*}^*=(z_{m^*}, \alpha_{\lambda_{m^*}}, 0) $$
and let $\gamma$ be the path from $\hat{z}_{m^*}$ to $\hat{z}_{m^*}^*$. Let
$$ \hat{z}_{m'}=(z_{m'}, \alpha_{\lambda_{m^*}}, 0) $$
and let $\hat{\rho}_2$ be the path from $\hat{z}_{m^*}^*$ to $\hat{z}_{m'}$ in $G(Z_5, R_5)$ as a copy of $\rho_2$. Let $$\hat{z}_{m'}^*=(z_{m'},\alpha_{\lambda_{m^*}}, \alpha_{\lambda_{m'}})$$
and let $\gamma'$ be the path from $\hat{z}_{m'}$ to $\hat{z}_{m'}^*$. Finally, let
$$ \hat{z}_{m+1}=(z_{m+1}, \alpha_{\lambda_{m^*}}, \alpha_{\lambda_{m'}})$$
and let $\hat{\rho_3}$ be the path from $\hat{z}_{m'}^*$ to $\hat{z}_{m+1}$ in $G(Z_5, R_5)$ as a copy of $\rho_3$. Then
$$ \hat{z}_0\stackrel{\hat{\rho}_1}{\longrightarrow} \hat{z}_{m^*}\stackrel{\gamma}{\longrightarrow} \hat{z}_{m^*}^* \stackrel{\hat{\rho}_2}{\longrightarrow} \hat{z}_{m'} \stackrel{\gamma'}{\longrightarrow} \hat{z}_{m'}^*\stackrel{\hat{\rho}_3}{\longrightarrow} \hat{z}_{m+1} $$
is a path correspondent to $w$. In $G(Z_5, R_5)$ we have $\tilde{a}(\hat{z}_{m+1})=\hat{z}_{m+1}$ because $z_{m+1}\neq z_{m^*}$ in $G(Z_3, R_3)$. We conclude that $(Z_5, R_5)$ is a $\Pi$-solution for an $\langle a, w^{-1}aw\rangle$-problem.

Finally suppose $u_1\neq u_2^{-1}$ but $u_3=u_2^{-1}$. In this case we let
$$ (Z_4, R_4)=\mbox{\tt Reroute}(Z_3, R_3; M, \tilde{a}, z_{m'}) $$
and 
$$ (Z_5, R_5)=\mbox{\tt Reroute}(Z_4, R_4; M, \tilde{a}, z_{m^*}). $$
In this case we may construct  a path given by $w^{-1}$ similarly to the above case. The argument is clear with considering the inverses of the paths. By a similar argument, we conclude that $(Z_5, R_5)$ is a $\Pi$-solution for an $\langle a, waw^{-1}\rangle$-problem. 

This completes the definition of $(Z_5, R_5)$ in all cases, and finishes the proof of the theorem.
\end{proof}

\subsection{Solutions for $1$-strongly stationary problems\label{subsec:1ss}}

Our objective is to show that all $2$-weakly wandering problems have odd-solutions. By Theorem~\ref{thm:2ww}, the non-stationary $2$-weakly wandering problems have $\Pi$-solutions for any $\Pi$, and in particular they have $p$-solutions for all primes $p$. If we choose $p$ to be an odd prime, the $p$-solutions are odd-solutions. Thus the problem reduces to stationary ones. By Proposition~\ref{prop:stationary}, they do not have nilpotent-solutions. Thus we need to consider new constructions.

\begin{definition} A word $w\in \Sigma$ is {\em strongly stationary} if there is a $v\in \Sigma^a$ such that $w=bv^*$. In other words, $w$ is of the form
$$ w=ba^{\alpha_k}b^{-1}a^{\beta_k}b\cdots a^{\alpha_1}b^{-1}a^{\beta_1}b, $$
whereas $v$ is of the form
$$ v=a^{\alpha_k}b^{\beta_k}\cdots a^{\alpha_1}b^{\beta_1}. $$
When $w$ has the above form we also say $w$ is {\em $k$-strongly stationary}.
\end{definition}

A strongly stationary $w\in \Sigma$ is $k$-strongly stationary for some $k\geq 1$. It is also clear that $w$ is stationary if and only if there are $u\in \Sigma$ and strongly stationary $w_0\in\Sigma$ such that $w=\phi_u(w_0)$.  By Theorem~\ref{thm:stationary}, under the assumption that an $\langle a, u^{-1}au\rangle$-problem has an odd-solution, if $w=\phi_u(w_0)$, then an $\langle a, w^{-1}bw\rangle$-problem has an odd-solution if and only if an $\langle a, w_0^{-1}aw_0\rangle$-problem has an odd-solution. 
When $w$ is $2$-weakly wandering and $w=\phi_u(w_0)$ for strongly stationary $w_0$, then $u$ is weakly wandering and so the above assumption holds. Our problem reduces further to strongly stationary ones.

Moreover, when $w$ is $2$-weakly wandering and $w=\phi_u(w_0)$ for strongly stationary $w_0$, then $w_0$ is $1$-strongly stationary and $w_0$ is of the form
$$ w_0=ba^{\alpha_1}b^{-1}a^{\beta_1}b. $$

Our main theorem of this subsection is the following.

\begin{theorem}\label{thm:1ss} For any $1$-strongly stationary word $w\in \Sigma$, there is an odd-solution to an $\langle a, w^{-1}aw\rangle$-problem.
\end{theorem}

By the above discussion, we see that this theorem has the following immediate corollaries.

\begin{corollary} Any $2$-weakly wandering problem has an odd-solution.
\end{corollary}

\begin{corollary} Let $w\in \Sigma$. If $w$ or $w^{-1}$ is $2$-weakly wandering, then an $\langle a, w^{-1}aw\rangle$-problem has an odd-solution.
\end{corollary}

The rest of this subsection is devoted to a proof of Theorem~\ref{thm:1ss}.

\subsubsection{The strategy of the proof}

Consider a new technique we call {\em rescaling}. For any word $w\in \Sigma$ and $m,n\in\mathbb{Z}\setminus\{0\}$, let $w^{(m,n)}$ be the word obtained from $w$ by replacing all $a$ by $a^m$ and all $b$ by $b^n$. The following lemma shows that rescaling a word creates a harder problem; if we can solve the harder problem, then the original problem can be solved.

\begin{lemma}\label{lem:rescaling} For any set $\Pi$ of prime numbers, $m, n\in\mathbb{Z}\setminus\{0\}$, $w\in \Sigma$ and $u\in \Sigma^a$, the following hold:
\begin{enumerate}
\item[\rm (i)] $\big\langle a, w^{(m,n)}\big\rangle\rightsquigarrow_\Pi\langle a, w\rangle$;
\item[\rm (ii)] $\big\langle a, \big(w^{(m,n)}\big)^{-1}aw^{(m,n)}\big\rangle\rightsquigarrow_{\Pi}\langle a, w^{-1}aw\rangle$;
\item[\rm (iii)] $\big\langle a, \big(u^{(m,n)}\big)^{-1}bu^{(m,n)}\big\rangle\rightsquigarrow_{\Pi}\langle a, u^{-1}bu\rangle$.
\end{enumerate}
\end{lemma}

\begin{proof} For each of these instances, if we have a $\Pi$-solution $(X, P)$ with $P=\{\tilde{a}, \tilde{b}\}$ for the rescaled problem, with solvent $\Lambda$, then $(X, Q)$ where $Q=\{\tilde{c}, \tilde{d}\}$, $\tilde{c}=\tilde{a}^m$ and $\tilde{d}=\tilde{b}^n$, is a solution for the original problem. Since $\langle \tilde{c}, \tilde{d}\rangle\leq \langle \tilde{a}, \tilde{b}\rangle=\Lambda$, the original problem has a $\Pi$-solution.
\end{proof}

Suppose
\begin{equation}\label{eqn:1ss} w=ba^{\alpha}b^{-1}a^{\beta}b \end{equation}
for $\alpha, \beta\in \mathbb{Z}\setminus\{0\}$. Without loss of generality, we may assume $\alpha+\beta>0>\alpha$. To see this, let $\Pi$ be the set of all odd primes. Then by Theorem~\ref{thm:stationary}, we have
$$\langle a, w^{-1}aw\rangle\leftrightsquigarrow_\Pi\big\langle a, (a^{\alpha}b^{\beta})^{-1}b(a^{\alpha}b^{\beta})\big\rangle. $$ 
Now choose $n, m\in \mathbb{Z}$ so that $n\beta>0$, $m\alpha<0$ and $n\beta+m\alpha>0$. By Lemma~\ref{lem:rescaling} (iii), 
\begin{equation} \big\langle a, \big(a^{m\alpha}b^{n\beta})^{-1}b(a^{m\alpha}b^{n\beta})\big\rangle\rightsquigarrow_\Pi\big\langle a, (a^{\alpha}b^{\beta})^{-1}b(a^{\alpha}b^{\beta})\big\rangle. \end{equation}
Again by Theorem~\ref{thm:stationary}, for $w'=ba^{n\alpha}b^{-1}a^{m\beta}b$, we have
\begin{equation}\langle a, w'^{-1}aw'\rangle\leftrightsquigarrow_\Pi\big\langle a, (a^{m\alpha}b^{n\beta})^{-1}b(a^{m\alpha}b^{n\beta})\big\rangle.
\end{equation}
Thus by transtivity, we have
$$ \langle a, w'^{-1}aw'\rangle\rightsquigarrow_\Pi \langle a, w^{-1}aw\rangle. $$

Our strategy for a proof of Theorem~\ref{thm:1ss} is the following. We will first try and find an odd-group with two generators $\Gamma=\langle a_\Gamma, b_\Gamma\rangle$ such that $b_\Gamma\not\in \langle a_\Gamma\rangle$ but $w_\Gamma\in \langle a_\Gamma\rangle$. Then, in the Cayley graph of $\Gamma$, starting from the identity $e_\Gamma$, by following the word $w$ we obtain a walk $\mu$ in the graph from $e_\Gamma$ to $w_\Gamma$. Using (\ref{eqn:1ss}) one can observe the following properties of $\mu$:
\begin{enumerate}
\item[(a)] None of the vertices of $\mu$ other than the endpoints are in $\langle a_\Gamma\rangle$;
\item[(b)] Suppose $\beta$ is less than the order of $a$ in $\Gamma$. Then the walk $\mu$ is a path if $w_\Gamma\neq e_\Gamma$, and is a cycle if $w_\Gamma=e_\Gamma$.
\end{enumerate}
Here (b) corresponds to the algebraic fact that no proper nonempty subword of $w$ except possibly $a^\gamma$, where $\gamma$ is the order of $a$ in $\Gamma$, is evaluated to be identity in $\Gamma$. 

In the second part of the proof we consider the EP problem $(X, P)$ where $X=\Gamma$ and $P=\{\tilde{a}, \tilde{b}\}$, where
$$ \tilde{a}(g)=\left\{\begin{array}{ll} g, & \mbox{ if $g\in \langle a_\Gamma\rangle$, }\\
a_\Gamma g, & \mbox{ otherwise;}
\end{array}\right. $$
and 
$$ \tilde{b}(g)=b_\Gamma g. $$
We will show that $(X, P)$ is an odd-solution. It is easy to see that the walk $\mu$ is a subproblem of $(X,P)$. 

In the third part of the proof we obtain an odd-solution for an $\langle a, w^{-1}aw\rangle$-problem by further applications of EP maps.

\subsubsection{The construction of an odd-group}

We now start to construct the group $\Gamma$ assuming $\alpha+\beta>0>\alpha$. In our construction, we will find a suitable prime number $q$ with $q\equiv 3 \pmod 4$. For now, suppose such a $q$ has been chosen. 
Let $\mathbb{F}_q$ be the field with $q$ many elements. Let $\mathbb{F}_q^\times$ be the multiplicative group of nonzero elements of $\mathbb{F}_q$. Consider the following subgroup of $SL_2(\mathbb{F}_q)=SL(2,q)$:
$$ H_q=\left\{\ \left[\begin{array}{cc} x & 0 \\ z & y\end{array}\right] \colon x,y\in \mathbb{F}_q^\times, z\in \mathbb{F}_q\right\}. $$

For the rest of the computations, we let
\begin{equation}\label{eqn:abH} a_H=\left[\begin{array}{cc} x& 0 \\ z & y\end{array}\right] \mbox{ and } b_H=\left[\begin{array}{cc} r & 0 \\ t & s\end{array}\right]. \end{equation}

The following lemmas come from direct computations.

\begin{lemma}\label{lem:anH} For any $n\in\mathbb{Z}$, 
$$ a_H^n=\left[\begin{array}{cc} x^n & 0 \\ \\ zA_n(x,y) & y^n\end{array}\right] $$
for some Laurent polynomial $A_n(x,y)$ in $x, y$.
\end{lemma}



\begin{lemma}\label{lem:wH} Let $w\in \Sigma$ be as in (\ref{eqn:1ss}). Then
$$ w_H=\left[\begin{array}{cc} rx^{\alpha+\beta} & 0 \\ \\ W(x,y,z,r,s,t) & sy^{\alpha+\beta} \end{array}\right] $$
where
$$ W(x,y,z,r,s,t)= t(x^{\alpha+\beta}+y^{\alpha+\beta}-x^\beta y^{\alpha})+z\big(ry^\alpha A_\beta(x,y)+sx^\beta A_\alpha(x,y)\big). $$
\end{lemma}

Let
\begin{equation} \Delta(v)=v^{\beta}-v^{\alpha+\beta}-1. \end{equation}
We note the following elementary number theoretic fact.

\begin{lemma}\label{lem:4q} There exist a prime $q\equiv 3\pmod 4$ and $v\in \mathbb{F}_q^\times$ such that $q>3$, $v$ is not the identity of $\mathbb{F}_q^\times$, the order of $v$ in $\mathbb{F}_q^\times$ is odd, and $\Delta(v)\equiv 0\pmod q$.
\end{lemma}

\begin{proof} Let $v=4$. Then $\Delta(4)\equiv 3\pmod 4$. It follows that  $\Delta(4)$ must have a prime factor $q\equiv 3\pmod 4$. We have $q>3$ since $\Delta(4)\equiv -1\not\equiv 0\pmod 3$. By Fermat's little theorem, $2^{q-1}\equiv 1\pmod q$. Hence $4^{\frac{q-1}{2}}\equiv 1\pmod q$. Since $q\equiv 3\pmod 4$, $(q-1)/2$ is odd. 
\end{proof}

Let $q$ and $v$ be given by Lemma~\ref{lem:4q}. Let 
\begin{equation} a_H=\left[\begin{array}{cc} v& 0 \\ 0 & 1\end{array}\right] \mbox{ and } b_H=\left[\begin{array}{cc} 1 & 0 \\ 1 & 1\end{array}\right]. \end{equation}
That is, we have $x=v$, $y=1$, $z=0$ and $r=s=t=1$. Then
$$ x^{\alpha+\beta}+y^{\alpha+\beta}-x^\alpha y^\beta=-y^{-(\alpha+\beta)}\Delta\left(\frac{x}{y}\right)=0$$
in $\mathbb{F}_q$. By Lemmas~\ref{lem:anH} and \ref{lem:wH}, we have that $W(x,y,z,r,s,t)=0$ and
$$ w_H=\left[\begin{array}{cc} v^{\alpha+\beta} & 0 \\ 0 & 1\end{array}\right]=a_H^{\alpha+\beta}\in \langle a_H\rangle. $$
Since $v\neq 1$ in $\mathbb{F}_q^\times$, it is straightforward to check that $a_H$ and $b_H$ do not commute. Thus if we let $\Gamma=\langle a_H, b_H\rangle$, then $\Gamma$ is a nonabelian group with $w_\Gamma\in \langle a_\Gamma\rangle$. 

Finally, we check that $\Gamma$ is an odd-group. Consider the group homomorphism
$$ \varphi\colon H_q\to \mathbb{F}_q^\times\times \mathbb{F}_q^\times $$
defined by
$$ \varphi\Big( \left[\begin{array}{cc} x & 0 \\ z & y\end{array}\right]\Big)=(x, y). $$
Note that $|\mbox{\rm ker}(\varphi)|=q$. Thus for any subgroup $G\leq H_q$, if $\varphi(G)$ is an odd-group, then so is $G$.  Note that $\varphi(\Gamma)=\langle (v, 1)\rangle\cong \langle v\rangle \leq \mathbb{F}_q^\times$. Since $v$ has odd order in $\mathbb{F}_q^\times$, we thus conclude that $\Gamma$ is an odd-group. This finishes the first part of the proof. 

We remark that when $\beta=1$ and $\alpha=-1$, we have $q=7$ and the group $\Gamma$ obtained from this construction is the Frobenius group of order $21$. For the problem $P_2$ in Subsection~\ref{subsec:case}, the word is $w=bab^{-1}a^{-1}b$, which is $1$-strongly stationary. It happens that $w_\Gamma=e_\Gamma$ in this case. This motivates the definition of the problems $P_3$ and $Q_3$ in Subsection~\ref{subsec:case}. More generally, whenever $\alpha+\beta=0$, our construction yields a nonabelian odd-group $\Gamma$ where $w_\Gamma=e_\Gamma$. In this case, we obtain an odd-solution for an $\langle a, w\rangle$-problem.

\subsubsection{The EP map {\tt Pin}}

Now we proceed to the second part of the proof. For this we give a more general definition.

\begin{definition} \label{def:Pin} Let $\Gamma$ be a finite group. Suppose $H_1$ and $H_2$ are nontrivial subgroups of $\Gamma$ such that
$H_1\cap H_2=\{e_\Gamma\}$ and $\Gamma=\langle H_1, H_2\rangle$. Let $S_1\subseteq H_1\setminus \{e_\Gamma\}$ be a generating set of $H_1$ and $S_2\subseteq H_2\setminus\{e_\Gamma\}$ be a generating set of $H_2$. We define an EP problem
$$ (X, P)={\tt Pin} (\Gamma, S_1, S_2) $$
by letting $X=\Gamma$, $P=\big\{ \tilde{h}\colon h\in S_1\cup S_2\big\}$ where
$$ \tilde{h}(g)=\left\{\begin{array}{ll} g, & \mbox{ if $h\in S_1$ and $g\in H_1$,} \\
hg, & \mbox{ otherwise.}
\end{array}\right.
$$
\end{definition}

Our proof of this step will be complete with the following.

\begin{proposition} Let $\Pi$ be a set of primes. Let $\Gamma, H_1, H_2, S_1, S_2$ be as in Definition~\ref{def:Pin}. If $\Gamma$ is a $\Pi$-group, then $\mbox{\tt Pin}(\Gamma, S_1, S_2)$ is a $\Pi$-solution.
\end{proposition}

\begin{proof} It is clear that $\mbox{\tt Pin} (\Gamma, S_1, S_2)$ is a solution. Let $\Lambda$ be the resolvent of $\mbox{\tt Pin} (\Gamma, S_1, S_2)$. We show that $\Lambda$ is a subgroup of a wreath product $H_1\wr \Gamma$. 

To define the wreath product, we let $X=H_1$ and let $$Y=\{y_0=e_\Gamma, y_1, \dots, y_n\}$$ be a complete set of left coset representatives of $\Gamma/H_1$. For each $\eta\in H_1$, let $\hat{\eta}\in \mbox{\rm Sym}(X)$ be defined by 
$$ \hat{\eta}(g)=\eta g. $$
In this way, we identify $H_1$ as a subgroup of $\mbox{\rm Sym}(X)$.  For each $\gamma\in \Gamma$, let $\bar{\gamma}\in \mbox{\rm Sym}(Y)$ be defined by
$$ \bar{\gamma}(y_i)=y_j \iff \gamma y_iH_1=y_j H_1. $$
In this way, we identify $\Gamma$ as a subgroup of $\mbox{\rm Sym}(Y)$. These identifications allow us to define the wreath product $H_1\wr \Gamma$ as in Definition~\ref{def:wr}.

Now we identify $\Gamma$ with $X\times Y$ via the bijection $\varphi\colon \Gamma\to X\times Y$ defined by
$$ \varphi^{-1}(\eta, y_i)=y_i\eta. $$
Then the wreath product $H_1\wr \Gamma$ can be identified with a subgroup of $\mbox{\rm Sym}(\Gamma)$. 

Let $h\in S_2$. Then $h\not\in H_1$ and $\tilde{h}(g)=hg$ for all $g\in \Gamma$. Define $\Phi(h)\in H_1\wr \Gamma$ by 
$$ \Phi(h)=\big((k_y)_{y\in Y}, \bar{h}\big)\in H_1^Y\times \Gamma $$
where, for all $y\in Y$ and $g\in H_1$,
$$ k_y(g)=\bar{h}(y)^{-1}hyg=\hat{\eta}_y(g)  $$
for $\eta_y=\bar{h}(y)^{-1}hy\in H_1$. 
Then it is straightforward to check that $\varphi\circ \tilde{h}=\Phi(h)\circ \varphi$, in other words, $\tilde{h}\in P$ can be identified with $\Phi(h)\in H_1\wr \Gamma$.

Let $h\in S_1$. Then $h\in H_1$, $\tilde{h}(g)=g$ if $g\in H_1$ and $\tilde{h}(g)=hg$ if $g\in \Gamma\setminus H_1$. In this case we define $\Phi(h)\in H_1\wr \Gamma$ by
$$ \Phi(h)=\big((k_y)_{y\in Y}, \bar{h}\big)\in H_1^Y\times \Gamma $$
where, for all $y\in Y$ and $g\in H_1$, 
$$ k_y=\left\{\begin{array}{ll} \mbox{\rm id}, & \mbox{ if $y=y_0=e_\Gamma\in Y$,} \\ \hat{\eta}_y, & \mbox{ otherwise}, \end{array}\right. $$
where $\hat{\eta}_y$ is defined as above. It is again the case that $\varphi\circ \tilde{h}=\Phi(h)\circ \varphi$, or $\tilde{h}\in P$ can be identified with $\Phi(h)\in H_1\wr\Gamma$.  

We have shown that $\Lambda$ is a subgroup of $H_1\wr \Gamma$. Since $\Gamma$ is a $\Pi$-group, it follows that $H_1$ is a $\Pi$-group and, from Lemma~\ref{lem:basicwr}, that $\Lambda$ is a $\Pi$-group. 
\end{proof}

\subsubsection{Further adjustments of the solution}

In this third part of the proof of Theorem~\ref{thm:1ss} we consider the order of $a$ in $\Gamma$, denoted as $\gamma$. 

If $\beta<\gamma$, then we have two cases. In case $w_\Gamma\neq e_\Gamma$, then $\mu$ is a path representing an $\langle a, w^{-1}aw\rangle$-problem, and we have already obtained an odd-solution $(X, P)$ for it. On the other hand, in case $w_\Gamma=e_\Gamma$, then $\mu$ represents an $\langle a, w\rangle$-problem, and thus $(X, P)$ is an odd-solution for an $\langle a, w\rangle$-problem. Let $\Pi$ be the set of all odd primes. Then $(X, P)$ is a $\Pi$-solution for an $\langle a, w\rangle$-problem. By Proposition~\ref{prop:awbasic} (iv), we obtain a $\Pi$-solution for an $\langle a, w^{-1}aw\rangle$-problem. This is an odd-solution. 

Now suppose $\beta\geq\gamma$. First, assume $\gamma\leq |\alpha|<\beta$. In this case the walk $\mu$ contains two cycles of length $\gamma$, corresponding to the subword $a^\alpha$ and $a^\beta$ respectively, and are denoted as $C_\alpha$ and $C_\beta$. These cycles are illustrated in Figure~\ref{fig:1ss1}. Let $x_0$ be the base point for the computation of the fundamental group. 

\begin{figure}[htbp]
\centering
\begin{tikzpicture}
  	\coordinate (A)  at (0,0);
  	\coordinate (B) at (0.5,0);
 	\coordinate (C) at (2.5,0);
 	\coordinate (D) at (3,0);
	\coordinate (E) at (5,0);
	\coordinate (F) at (5.5,0);

  	\fill (A) circle (2pt);
	\fill (B) circle (1pt);
	\fill (C) circle (1pt);
	\fill (D) circle (1pt);
	\fill (E) circle (1pt);
	\fill (F) circle (1pt);

\draw (-0.5,0) circle (0.5);
\draw (A) -- (B);
\draw (1.5,0) circle (1);
\draw (C) -- (D);
\draw (4,0) circle (1);

\draw (E) -- (F);
\draw (6, 0) circle (0.5);

    \node[right] at (1.16,-0.7) {$C_\beta$};
    \node[right] at (3.65,-0.7) {$C_\alpha$};
\node[right] at (-0.6,0) {$x_0$};
  
\end{tikzpicture}
\caption{\label{fig:1ss1} The walk $\mu$ when $\gamma\leq |\alpha|<\beta$.}
\end{figure}
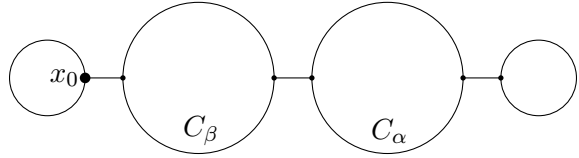

Let $K>\beta/\gamma$ be an integer. Pick a vertex $y_0$ on $C_\beta$, and consider
$$ (Y, Q)=\mbox{\tt Reroute}(X, P; K, a, y_0). $$
Let $\mu_Y$ be the lift of $\mu$ in $(Y, Q)$, that is, $\mu_Y$ is the walk starting from $x_0$ obtained by following the word $w$. In $G(Y, Q)$, the cycle $C_\beta$ from $G(X, P)$ becomes a cycle of length $K\gamma>\beta$. Hence the part of the walk $\mu_Y$ corresponding to $C_\beta$ is no longer a cycle but a path. However, the part of the walk $\mu_Y$ corresponding to $C_\alpha$ continues to be a cycle. Thus we pick a vertex $z_0$ on it and form
$$ (Z, R)=\mbox{\tt Reroute}(Y, Q; K, a, z_0). $$
By a similar argument, and because $|\alpha|<\beta$, this cylce also becomes a path in $G(Z, R)$. Let $\mu_Z$ be the lift of $\mu_Y$ in $(Z, R)$. Then $\mu_Z$ represents an $\langle a, w^{-1}aw\rangle$-problem. Since $(Y, Q)$ and $(Z, R)$ are odd-solutions, we obtain an odd-solution for an $\langle a, w^{-1}aw\rangle$-problem. 

Finally, assume $|\alpha|<\gamma\leq \beta$. This is similar to the above case except that we do not have a cycle $C_\alpha$ in $\mu$. In this case we construct $(Y, Q)$ similarly and $(Y, Q)$ is already an odd-solution to an $\langle a, w^{-1}aw\rangle$-problem. Thus our proof of Theorem~\ref{thm:1ss} is complete.

\bigskip

{\sc Statement of AI use.} All research results in this paper were obtained without any use of AI. AI was also not used in the preparation of this paper.

\end{document}